\documentclass[reqno, 11pt,twosides]{amsart}
\usepackage{amssymb,amsmath,amsthm}
\usepackage{a4wide}
\usepackage{amscd}
\usepackage{amsfonts}
\usepackage{amssymb}
\usepackage{latexsym}
\usepackage{color}
\usepackage{esint}
\usepackage{graphicx}
\usepackage{float}
\graphicspath{{Figures/}}

\usepackage{graphicx}
\usepackage[breaklinks=true,bookmarks=false]{hyperref}
\usepackage{cleveref}
\usepackage{mathrsfs}

\allowdisplaybreaks

\usepackage[makeroom]{cancel}

\newtheorem{theorem}{Theorem}

\newtheorem{definition}{Definition}
\newtheorem{lemma}{Lemma}
\newtheorem{proposition}[theorem]{Proposition}
\newtheorem{remark}{Remark}
\let\a=\alpha
\let\e=\varepsilon
\let\d=\delta

\let\p=\partial

\let\g=\gamma

\let\b=\beta

\numberwithin{equation}{section}

\let\hide\iffalse
\let\unhide\fi

\newcommand{\avkl}{\langle k-\ell\rangle}
\newcommand{\avl}{\langle \ell\rangle}
\newcommand{\avk}{\langle k\rangle}

\newcommand{\R}{\mathbb{R}}

\newcommand{\be}{\begin{equation}}
\newcommand{\bm}{\begin{multline}}
\newcommand{\ee}{\end{equation}}
\newcommand{\dd}{\mathrm{d}}

\newcommand{\xb}{x_{\mathbf{b}}}
\newcommand{\tb}{t_{\mathbf{b}}}

\newcommand{\Bes}{\begin{eqnarray*}}
\newcommand{\Ees}{\end{eqnarray*}}
\newcommand{\Be}{\begin{equation} }
\newcommand{\Ee}{\end{equation}}

\def\p{\partial}

\def\R{\mathbb{R}}
\def\d{\mathrm{d}}
\def\B{\begin{equation}}
\def\E{\end{equation}}
\def\BN{\begin{eqnarray*}}
\def\EN{\end{eqnarray*}}

\begin{document}
\title[Half-space problem on the Boltzmann equation]{Half-space problem on the Boltzmann equation with zero Mach number at infinity}

\author[H.-X. Chen]{Hongxu Chen}
%\thanks{}
\address[HXC]{School of Mathematical Sciences, Shenzhen University, Shenzhen, Guangdong 518060, China}
\email{hongxuchen.math@gmail.com}

\author[J.-L. Chen]{Jun-Ling Chen}
%\thanks{}
\address[JLC]{Department of Mathematics, The Chinese University of Hong Kong,
Shatin, Hong Kong, P.R.~China}
\email{junlingchen@cuhk.edu.hk}

\author[R.-J. Duan]{Renjun Duan}
%\thanks{}
\address[RJD]{Department of Mathematics, The Chinese University of Hong Kong,
Shatin, Hong Kong, P.R.~China}
\email{rjduan@math.cuhk.edu.hk}

\begin{abstract}
We study the long-time dynamics of the time-evolutionary Boltzmann equation with hard sphere collisions in the three-dimensional half-space  \( \mathbb{R}^2 \times \mathbb{R}^+\), subject to diffuse reflection boundary conditions and small perturbations around a global Maxwellian equilibrium. The far-field velocity is assumed to be at rest; namely, we take the zero Mach number at infinity. In the first goal, we construct global-in-time low-regularity solutions near Maxwellians. We leverage time-decay properties along the two-dimensional tangential direction to establish polynomial decay rates of solutions matching the 2D heat equation.  In the second goal, we further prove the propagation of Gevrey regularity: analyticity (Gevrey index 1) in the tangential spatial variable \(x_\parallel\), and Gevrey class with index 2 in the tangential velocity variable \(v_\parallel\), under suitably regular initial data. The proofs combine an \(L^1_k \cap L^p_k\) Fourier-space approach for decay estimates, macro-micro decomposition with \(L^2 - L^\infty\) frameworks adapted to unbounded domains, and weighted Gevrey norms to control regularity propagation, overcoming challenges from boundary effects and nonlinear interactions. 
\end{abstract}

\date{\today}
	%35Q20 	Boltzmann equations
    %35B35  Stability in context of PDEs 
    %35B40  Asymptotic behavior of solutions to PDEs
    %35B65  Smoothness and regularity of solutions to PDEs

\maketitle
\tableofcontents

\section{Introduction}

This paper is concerned with the initial-boundary value problem on the Boltzmann equation in the half-space $\Omega=\mathbb{R}^2\times \mathbb{R}^+$ of three dimensions:
\begin{align}\label{op}
\begin{cases}
 & \p_t F + v\cdot \nabla_x F = Q(F,F),\\
 & F(t,x_1,x_2,0,v)|_{v_3> 0} = c_\mu \mu(v) \int_{u_3 < 0} F(t,x_1,x_2,0,u) |u_3| \,\dd u,\\
 &\lim\limits_{x_3\to \infty}F(t,x,v)=\mu,\\
 & F(0,x,v)= F_0(x,v).
 \end{cases}
\end{align}
Here, 
$F=F(t,x,v)\geq 0$ stands for the velocity distribution function of gas particles with velocity $v=(v_1,v_2,v_3)\in \R^3$ at time $t\geq 0$ and position $x=(x_1,x_2,x_3)\in \Omega$. The Boltzmann collision term is a bilinear integral operator acting only on velocity variable, and through the paper we consider only the hard sphere model taking the form of
\begin{equation*}
Q(F,G)=\int_{\mathbb{R}^3}\int_{\mathbb{S}^2}|(v-u)\cdot \omega|[F(u')G(v')-F(u)G(v)]\,\d\omega \dd u,
\end{equation*}
where the velocity pairs $(v,u)$ and $(v',u')$ are related by the $\omega$-representation:
\begin{equation*}
 v'=v-[(v-u)\cdot\omega]\omega,\quad u'=u+[(v-u)\cdot\omega]\omega,\quad \omega\in\mathbb{S}^2,
\end{equation*}
in terms of the conservation of momentum and energy for elastic collisions between molecules:
\begin{equation*}
 v+u=v'+u',\quad |v|^2+|u|^2=|v'|^2+|u'|^2.
\end{equation*}
For the boundary conditions, we have imposed the diffuse reflection boundary condition at the wall $x_3=0$ and the rest global Maxwellian at infinity $x_3=\infty$. Here,  
\begin{equation}\label{def.gm}
\mu=\mu(v):=\frac{1}{(2\pi)^{3/2}}e^{-\frac{|v|^2}{2}}
\end{equation}
is a normalized global Maxwellian with zero bulk velocity and the constant $c_\mu=\sqrt{2\pi}$ is chosen to satisfy $\int_{v_3\gtrless 0}c_\mu\mu(v) |v_3|dv=1$ so that $c_\mu\mu(v) |v_3|$ is a probability measure on the half velocity spaces $\{\R^3: v_3\gtrless 0\}$. Note that under the diffuse boundary, the mass flux is vanishing at the wall $x_3=0$, namely,
\begin{equation*}
\int_{\R^3} v_3F(t,x_1,x_2,x_3=0, v)\,\d v=0.
\end{equation*} 

We consider the perturbation solution of \eqref{op} around the global Maxwellian \eqref{def.gm} as $F=\mu+\sqrt{\mu}f$. Plugging this into \eqref{op}, the problem can be reformulated as
\begin{align}\label{f_eqn}
\begin{cases}
 & \p_t f + v\cdot \nabla_x f + \mathcal{L}f = \Gamma(f,f), \\
 & f(t,x_1,x_2,0, v)|_{v_3> 0} = c_\mu \sqrt{\mu(v)} \int_{u_3 < 0} f(t,x_1,x_2,0,u)\sqrt{\mu(u)}|u_3|\,\dd u, \\
 &\lim\limits_{x_3\to \infty}f(t,x,v)=0,\\
 & f(0,x,v)= f_0(x,v) := (F(0,x,v)-\mu)/\sqrt{\mu},
 \end{cases}
\end{align}
where the linearized collision term and nonlinear term are respectively denoted by 
\begin{align}
 & \mathcal{L}f:= -\mu^{-1/2}[Q(\mu,\sqrt{\mu}f) + Q(\sqrt{\mu}f,\mu)], \label{def.L}\\
 & \Gamma(f,f) := \mu^{-1/2} Q(\sqrt{\mu}f,\sqrt{\mu}f).\label{def.Ga}
\end{align}
 When applying Leibniz's formula, it will be convenient to work with the trilinear operator $\mathcal{T}$ defined by
\begin{equation}\label{matht}
\mathcal T (f,g, h)= \int_{\mathbb{R}^3}\int_{\mathbb{S}^2} |(v-u)\cdot \omega| h(u) [f(u')g(v')-f(u)g(v)] \dd \omega\dd u.
\end{equation}
%where $h$ is a function of $v$ variable only.
The bilinear operator $\Gamma$ in \eqref{def.Ga} and the above $\mathcal{T}$ are linked by
\begin{equation*}
	\Gamma(g,h)=\mathcal{T}(g,h,\mu^{1/2}).
\end{equation*}

Half‑space problems for the Boltzmann equation arise naturally in the kinetic description of gases in contact with a wall. They model, for instance, the Knudsen layer that forms near a solid boundary, where the distribution function undergoes rapid changes to accommodate the boundary condition, cf.~\cite{Sone}. In the general setting, we define the far field Maxwellian $M_\infty(v)$, the sound speed $c_\infty$ at rest and the Mach number $\mathcal{M}^\infty$ as in \cite{UYY03ex} by
\begin{equation*}
M_\infty(v)=\mu(v_1-u_\infty,v_2,v_3)=\frac{1}{(2\pi)^{3/2}}e^{-\frac{|v_1-u_\infty|^2+|v_2|^2+|v_3|^2}{2}},\quad c_\infty=\sqrt{\frac{5}{3}},\quad  \mathcal{M}^\infty=\frac{u_\infty}{c_\infty}.
\end{equation*}
The steady linearized and nonlinear half‑space problems have been studied extensively in case of inflow data at the wall. For the classical Milne and Kramer’s problems, the existence of boundary layer solutions was established by Bardos-Caflisch-Nicolaenko \cite{BCN86CPAM} for the linearized hard‑sphere Boltzmann equation under the assumption of zero Mach number $\mathcal{M}^\infty=0$. The same problem was studied by Coron-Golse-Sulem \cite{CGS88CPAM} for any $\mathcal{M}^\infty\in \R$. For the nonlinear Boltzmann equation with the hard sphere case, Ukai-Yang-Yu studied the existence of boundary layer solutions in \cite{UYY03ex} for $\mathcal{M}^\infty\notin \{0,\pm 1\}$ and the asymptotic stability in \cite{UYY04sta} for $\mathcal{M}^\infty<-1$; see also \cite{Golse08} and the survey \cite{BGS}.  Sakamoto-Suzuki-Zhang \cite{SSZJDE22} extent the one-dimensional half-line result in \cite{UYY04sta} to the three-dimensional half-space problem. Bernhoff-Golse \cite{BG21ARMA} proved the existence and uniqueness of a uniformly decaying boundary layer type solution of the Boltzmann equation. Results for other physical boundary conditions were also obtained. For the specular reflection boundary condition, Golse-Perthame-Sulem \cite{GPSu88} proved the existence and exponential spatial decay of boundary layer solutions to the nonlinear hard-sphere Boltzmann equation with the zero Mach number $\mathcal{M}^\infty=0$. Huang-Wang \cite{HW} generalized the result to the case of the diffuse reflection boundary condition. In those stationary boundary‑layer problems as in  \cite{UYY03ex}, the far‑field state is a non‑rest Maxwellian, and the solution exhibits exponential convergence towards it. The time‑asymptotic stability of such non‑trivial stationary solutions is a fundamental and still largely open problem for general $\mathcal{M}^\infty$, which would validate the physical relevance of the boundary‑layer profile.  

When the far‑field Mach number is zero (i.e., the gas at infinity is at rest and in thermal equilibrium), the stationary solution reduces to the trivial global Maxwellian with zero bulk velocity. In that case, the problem becomes the stability of the constant equilibrium state $\mu$ in \eqref{def.gm}. In case of the diffuse reflection boundary condition, for the zero‑Mach‑number half‑space, the far‑field decay of the linearized problem is slower than in the non‑zero‑Mach case, because the shear and heat modes become purely diffusive. In particular, when the tangential variables are one‑dimensional, the expected decay rate in large time is that of a one‑dimensional heat equation, which is too weak to close a nonlinear energy estimate with existing techniques; the global well‑posedness and asymptotic stability of small perturbations around the global Maxwellian $\mu$ in the half‑space \(\mathbb{R}\times\mathbb{R}^+\) remains an outstanding open problem. By contrast, when the tangential space is two‑dimensional, the diffusive effect is enhanced and the solutions of the linearized problem decay in large time like those of the 2D heat equation, which provides enough integrability in time to handle the nonlinear terms. This geometric effect is essential for our proof.

In the present paper, we consider the three‑dimensional half‑space \(\mathbb{R}^2\times\mathbb{R}^+\) with diffuse reflection boundary conditions, assuming that the far‑field velocity vanishes (i.e., zero Mach number $\mathcal{M}^\infty=0$ at infinity). We construct global‑in‑time small‑amplitude solutions near the global Maxwellian, prove their polynomial decay with the 2D heat‑equation rate, and establish the propagation of regularity: analytic regularity in the tangential spatial variable \(x_\parallel\) and Gevrey‑\(2\) regularity in the tangential velocity variable \(v_\parallel\). The three‑dimensional nature plays a crucial role in obtaining the decay estimates, while the Gevrey regularity in \(v_\parallel\) with index \(2\) is optimal in view of the polynomial growth along characteristics.

%The Boltzmann equation describes the statistical behavior of the dilute gas, and plays an important role in the kinetic theory. Its mathematical analysis has drawn considerable interest.

%In this paper, we study global Gevrey solutions on the Boltzmann equation in half-space $\Omega=\mathbb{R}^2\times \mathbb{R}^+$ with diffusive boundary condition on $x_3=0$.We construct global-in-time solutions and propagation in half-space $\Omega=\mathbb{R}^2\times \mathbb{R}^+$. Specially, we obtain the large-time behavior of solutions in domain $\Omega=\mathbb{R}^2 \times \mathbb{R}^+$.

%\subsection{Arrangement of the paper} 
%\textbf{Outline.} 

The rest of this paper is arranged as follows.
In Section \ref{sec.mr}, we represent the main results of this paper. 
In Section \ref{sec:prelim}, we recall several preliminary estimates. 
In Section \ref{sec:withoutder}, we obtain the global well-posedness in the half-space of $\mathbb{R}^2\times \mathbb{R}^+$.
In Section \ref{anainx}, we obtain the global analytic $x_\parallel$ estimate in the half-space
 of $\mathbb{R}^2\times \mathbb{R}^+$.
In Section \ref{anainmixy}, we obtain analytic $x_\parallel$ estimate and Gevrey $v_\parallel$ estimate in the half-space of $\mathbb{R}^2\times \mathbb{R}^+$. In the appendix Section \ref{appendi}, we prove the key equality in the proof of Lemma \ref{estdivk1ta}.

\section{Main results}\label{sec.mr}

%We focus on the half-space domains $\Omega=\mathbb{R}^2\times \mathbb{R}^+$ .

We use $k: = (k_1,k_2)\in \R^2$ to represent two-dimensional Fourier variable of the tangent physical variable $x_\parallel:=(x_1,x_2)\in \R^2$. Thus the Fourier transform of $f(t,x,v)$ is defined as
\begin{align*}
& \hat{f}(t,k,x_3,v) = \int_{\mathbb{R}^2} f(t,x,v) e^{-i k\cdot x_\parallel} \dd x_\parallel. 
\end{align*}
Denote $v_\parallel:=(v_1,v_2)\in \mathbb{R}^2$ as the tangent velocity variables. From \eqref{f_eqn}, the problem for $\hat{f}=\hat{f}(t,k,x_3,v)$ can be formulated as
\begin{align}
\begin{cases}
 & \p_t \hat{f} + iv_\parallel\cdot k \hat{f} + v_3 \p_{x_3} \hat{f} + \mathcal{L}\hat{f} = \hat{\Gamma}(\hat{f},\hat{f}),\\
 & \hat{f}(t,k,0,v)|_{v_3 > 0} = c_\mu \sqrt{\mu(v)} \int_{u_3 < 0 } \hat{f}(t,k,0,u) \sqrt{\mu(u)}|u_3| \dd u, \\
 &\lim\limits_{x_3\to \infty}\hat{f}(t,k,x_3,v)=0,\\
 & \hat{f}(0,k,x_3,v) = \hat{f_0}(k,x_3,v). 
\end{cases} \label{f_eqn_r2}
\end{align}
Here, recalling \eqref{def.Ga}, $\hat{\Gamma}(\hat{f},\hat{g})$ is the Fourier transform of the nonlinear term $\Gamma(f,g)$ with respect to $x_\parallel=(x_1,x_2)$:
\begin{align}
 & \hat{\Gamma}(\hat{f},\hat{g}) := \int_{\mathbb{R}^3}\int_{\mathbb{S}^2} |(v-u)\cdot \omega|\sqrt{\mu(u)}[\hat{f}(x_3,v')*_k \hat{g}(x_3,u')-f(x_3,v)*_k g(x_3,u)] \dd \omega \dd u,\notag
\end{align}
namely, it can be explicitly written as
\begin{align}
 \hat{\Gamma}(\hat{f},\hat{g})(t,k,x_3,v)= \int_{\mathbb{R}^3}\int_{\mathbb{S}^2}\int_{\R^2} |(v-u)\cdot \omega|\sqrt{\mu(u)}[&\hat{f}(t,k-\ell,x_3,v') \hat{g}(t,\ell,x_3,u')\notag\\
 &-f(t,k-\ell,x_3,v) g(t,\ell,x_3,u)]\d \ell \dd \omega \dd u.
 \notag
\end{align}

 Similarly as above we denote by $\hat{\mathcal{T}}(\hat{f},\hat{g},h)$ the Fourier transform of $ \mathcal T (f,g,h)$ in \eqref{matht} with respect to $x_\parallel$, that is, for any functions $h=h(v)$ that only depends on $v$ variable,
\begin{multline}\label{def.hatgamma}
\hat{\mathcal{T}}(\hat{f}, \hat g, h)
=\int_{\mathbb{R}^3}\int_{\mathbb{S}^2} \int_{\R^2}|(v-u)\cdot \omega| h(u)[\hat{f}(t,k-\ell,x_3,v') \hat{g}(t,\ell,x_3,u')\\
 -f(t,k-\ell,x_3,v) g(t,\ell,x_3,u)]\d \ell \dd \omega \dd u.
\end{multline}
Here, we observe that 
\begin{equation}\label{lingata}
 \hat\Gamma(\hat f, \hat g)=\hat {\mathcal T}(\hat f, \hat g, \mu^{1/2}). 
\end{equation}

 Denote the macroscopic component as $\mathbf{P}{\hat{f}}$, which represents the projection from $L^2_v$ to $\ker \mathcal{L}=\text{span} (\{\sqrt{\mu(v)},v\sqrt{\mu(v)},\frac{1}{2}(|v|^2-3)\sqrt{\mu(v)}\}) $:
\[\mathbf{P}\hat{f}:= \Big(\hat{a}+\hat{\mathbf{b}}\cdot v + \hat{c}\frac{|v|^2-3}{2} \Big)\sqrt{\mu(v)},\]
where $\hat{a}$, $\hat{\mathbf{b}}=(\hat{b}_1,\hat{b}_2,\hat{b}_3)$ and $\hat{c}$ are functions of $(t,k,x_3)$ for $\hat{f}=\hat{f}(t,k,x_3,v)$. Denote an exponential velocity weight as
\begin{align}\label{weight_w}
w(v) := e^{\theta |v|^2}, \ 0<\theta< \frac{1}{8}.
\end{align}
Denote the $P_\gamma {\hat{f}}$ as the projection to the diffuse reflection at the wall $x_3=0$:
\begin{align*}
 & P_\gamma {\hat{f}} (t,k,0, v) := c_\mu \sqrt{\mu(v)} \int_{u_3<0 } \hat{f}(t,k,0,u) \sqrt{\mu(u)} |u_3| \dd u,\quad v_3 > 0.
\end{align*}

Our first result answers the global well-posedness and asymptotic stability of the problem in the half space $\mathbb{R}^2\times \mathbb{R}^+$.

\begin{theorem}\label{thm:l1k_lpkexg}
Let $2<p\leq \infty$ and $\sigma = 2(1-\frac{1}{p})-2\e > 1$ with $\e>0$ small enough, then there exist constants $0<\delta_0 \ll 1$ and $C>0$ such that if the initial data $f_0(x,v)$ with $F_0(x,v):=\mu+\sqrt{\mu}f_0(x,v)\geq 0$ satisfies
\begin{equation}
\Vert \hat{f}_0\Vert_{L^1_k L^2_{x_3,v}} +\Vert w\hat{f}_0\Vert_{L^1_k L^\infty_{x_3,v}} + \Vert \hat{f}_0\Vert_{L^p_k L^2_{x_3,v}} < \delta_0, \label{initial_assumptionexg}
\end{equation}
then there exists a unique solution $\hat{f}(t,k,x_3,v)$ to \eqref{f_eqn_r2} such that $F(t,x,v)=\mu+\sqrt{\mu}f(t,x,v)\geq 0$ 
%with $\int_{\O}\int_{\mathbb{R}^3}\sqrt{\mu(v)}f(t,x,v) \dd v \dd x=0$ 
and the following estimate is satisfied:
\begin{align}
&\Vert (1+t)^{\sigma/2} \hat{f}\Vert_{L^1_kL^\infty_T L^2_{x_3,v}} +\Vert (1+t)^{\sigma/2} w \hat{f}\Vert_{L^1_k L^\infty_{T,x_3,v}} + \Vert \hat{f}\Vert_{L^p_k L^\infty_T L^2_{x_3,v}} \notag\\& \leq C\big[\Vert \hat{f}_0\Vert_{L^1_k L^2_{x_3,v}} +\Vert w\hat{f}_0\Vert_{L^1_k L^\infty_{x_3,v}} + \Vert \hat{f}_0\Vert_{L^p_k L^2_{x_3,v}}\big], 
\label{f_estimateexg}
\end{align}
for any $T>0$. 
\iffalse Moreover, it also holds that
\begin{align}
 & \Vert (1+t)^{\sigma/2} \hat{f}\Vert_{L^1_k L^\infty_T L^2_{x_3,v}} + \Vert (1+t)^{\sigma/2} (\mathbf{I}-\mathbf{P}) \hat{f} \Vert_{L^1_k L^2_{T,x_3,\nu}} + |(1+t)^{\sigma/2}(1-P_\gamma)\hat{f}|_{L^1_k L^2_{T,\gamma_+}} \notag \\
 & + \Big\Vert (1+t)^{\sigma/2} \frac{|k|}{\sqrt{1+|k|^2}}(\hat{a},\hat{\mathbf{b}},\hat{c})\Big\Vert_{L^1_k L^2_{T,x_3}} \leq C\Vert w\hat{f}_0\Vert_{L^1_k L^\infty_{x_3,v}} + C\Vert \hat{f}_0\Vert_{L^p_k L^2_{x_3,v}}. \label{f_estimate_2}
\end{align}\fi
\end{theorem}

\begin{remark}
Theorem \ref{thm:l1k_lpkexg} provides the first global solution to the nonlinear Boltzmann equation in the half-space with the same decay rate as the 2D heat equation. A global solution to the linear Boltzmann equation and the Fokker-Planck equation with absorbing boundary condition has been constructed in \cite{bouin2025half},  in which the decay rate is the same as the three dimensional heat equation. It remains an open and interesting problem to improve the nonlinear decay rate with diffuse reflection boundary condition to the optimal one in the current problem.
\end{remark}

With the well-posedness result in Theorem \ref{thm:l1k_lpkexg}, we further pursue the propagation of analytic and Gevrey regularity of the obtained global solutions. It should be noted that high-order derivatives are not possible in the vertical direction in the presence of boundary; therefore, we only focus on the regularity estimate in the horizontal direction. We denote $\p_{v_\parallel}^\a=\p_{v_1}^{\a_1}\p_{v_2}^{\a_2}, \ \a=(\a_1, \a_2) \in \mathbb{Z}^2_+$ as the tangent $v$-derivative. In the following, we define Gevrey regularity in the tangential variable. We refer to all norm notations in Section \ref{sec:notation}. 
\begin{definition}\label{defgevr2}
The function space $X_\rho (\mathbb{R}^2 \times \mathbb{R}^+\times \mathbb{R}^3) $ consists of all smooth in $x_\parallel,v_\parallel$ functions $h(x,v)$ such that the norm $ \Vert h\Vert_{X_\rho} < +\infty$, where
\begin{align*}
 & \Vert h\Vert_{X_\rho} := \sup_{m, \a\geq 0}L_{\rho, m, |\a|} \Vert w \langle k \rangle ^{m}\p_{v_\parallel}^{\a}\hat h \Vert_{L^1_k L^\infty_{x_3,v}}+\sup_{m\geq 0}L_{\rho, m,0 }\Big(\Vert \avk^{m}\hat h\Vert_{L^1_k L^2_{x_3,v}} +\Vert \avk^{m}\hat h\Vert_{L^p_k L^2_{x_3,v}}\Big),
\end{align*}
with $2<p\leq \infty$ and
\begin{equation}\label{r2dalp}
L_{\rho,m, |\a|} := \frac{(16/\nu_0)^{m}\rho^{m+|\a|+1}(m+2)^2(|\a|+2)^2}{(m+|\a|)!|\a|!}.
\end{equation}
Here $\rho>0$ is the analytic radius, $m\in \mathbb{Z}_+$ represents the order of derivative in $x_\parallel$($k$ in frequency space), and $\alpha \in \mathbb{Z}_+^2$ represents the order of derivative in $v_\parallel$.

We say $ h=h(x, v) $ is $G^{r}_{x_\parallel}(\Omega \times \mathbb{R}^3)$ in $x_\parallel$ and $G^s_{v_\parallel}(\Omega\times \mathbb{R}^3)$ in $v_\parallel$ with index $r, s>0$
if $h \in C^\infty_{x_\parallel,v_\parallel}(\Omega\times \mathbb{R}^3)$ and there exists a constant $C>0$ such that
\begin{eqnarray*}
\Vert\avk^m\partial_{v_\parallel}^\a \hat h \Vert _{L^1_k L^\infty_{x_3,v}} \leq C^{m+|\a| +1}(m!)^r(|\a|!)^s  , \ \ \  \text{for any }	 m \in \mathbb{Z}_+ \text{ and } \a \in \mathbb{Z}_+^2 .
\end{eqnarray*}
Here $r, s$ are called the Gevrey index. For instance, $G^1_{x_\parallel}(\Omega\times \mathbb{R}^3)$ is the space of analytic functions in $x_\parallel$, and $G^r_{x_\parallel}(\Omega\times \mathbb{R}^3)$ with $0<r<1$ is the space of ultra-analytic functions in $x_\parallel$.
\end{definition}

Below we state the regularity propagation results in the domain $\Omega=\mathbb{R}^2\times \mathbb{R}^+$.

\begin{theorem}\label{gevr2}
Let all assumptions in Theorem \ref{thm:l1k_lpkexg} be satisfied. Then there exist constants $0<\delta_0 \ll 1$ and $C>0$ such that if the initial data $f_0(x,v)$ further satisfies
\begin{equation}
\Vert f_0\Vert_{X_{\rho}} < \delta_0, \label{initial_assumption}
\end{equation}
for some $0<\rho \leq 1/(2C)$, then the solution $f(t,x,v)$ further satisfies the following estimate:
\begin{multline*}
\sup_{m\geq 0}L_{\rho,m,0}\Big(\Vert(1+t)^{\sigma/2}\avk^{m} \hat{f}\Vert_{L^1_kL^\infty_T L^2_{x_3,v}}+\Vert\avk^{m} \hat{f}\Vert_{L^p_kL^\infty_T L^2_{x_3,v}} \Big) \\ \quad+\sup_{m,\a\geq 0}L_{\rho,m,|\a|}\Vert(1+t)^{\sigma/2} w\langle k \rangle ^{m}\p_{v_\parallel}^{\a}\hat{f}\Vert_{L^1_k L^\infty_{T,x_3,v}}\leq C \Vert f_0\Vert_{X_{\rho}}, 
\end{multline*}
for any $T>0$.
Equivalently, the solution $f(t,x,v)$ is $G^1_{x_\parallel}$ in $x_\parallel$ and $G^2_{v_\parallel}$ in $v_\parallel$, and satisfies that
\begin{equation*}
\forall m\in\mathbb{Z}_+, \ \a \in \mathbb{Z}^2_+, \ \Vert (1+t)^{\sigma/2}w \langle k \rangle ^{m}\p_{v_\parallel}^{\a}\hat f(t)\Vert_{L^1_k L^\infty_{T,x_3,v}} \leq C^{m+|\a|+1}m!(|\a|!)^2
\end{equation*}
for any $T>0$. 
\end{theorem}

\begin{remark}%\label{rmk:regularity}
The analytic regularity is achieved for the spatial variable $x_\parallel$. However, for the velocity variable, the characteristic $x-tv$ introduces an extra polynomial growth in $t$ when estimating $\p_{v_\parallel}^\alpha$. Consequently, analytic regularity in $v_\parallel$ can only be expected for local solution. To achieve a global result, this polynomial growth is controlled at the cost of increasing the Gevrey index to $G^2_{v_\parallel}$ in the velocity variable, as shown in Lemma \ref{lemma:enut_control}. The constant $16/\nu_0$ in Definition \ref{defgevr2} is introduced to absorb the constant term $8/\nu_0$ in \eqref{derienut}.
We refer to a detailed discussion of the Gevrey index and analytic radius $\rho$ in Section \ref{sec:proof_strategy}.
\end{remark}

\begin{remark}
It is straightforward to adapt the similar argument to establish global $G^r_{x_\parallel}$ in $x_\parallel$ and global $G^{r+1}_{v_\parallel}$ in $v_\parallel$ for $r > 1$, provided that the initial condition has the same regularity. For conciseness, we restrict to the case $r = 1$ in the paper.
\end{remark}

The proof of Theorem \ref{gevr2} is based on Theorem \ref{thminxm} and Theorem \ref{thm:apriestxvr2}, which gives the analytic estimate in $x_{\parallel}$ and the Gevrey estimate in $v_{\parallel}$, respectively.

\hide
{\color{red} Not sure if it would be better to include the other two statements regarding the global well-posedness and analytic $x_\parallel$ regularity.}
\unhide

\hide
\begin{remark}
The smallness condition is only required for $f$, but not required for the regularity. {\color{red} Not sure if this is correct yet.}
 
\end{remark}
\unhide

\subsection{Literature} 
The analysis of the Boltzmann equation can be classified into two frameworks: non-perturbative and perturbative. In the non-perturbation framework, one seeks global solutions under large initial data with finite mass, energy, and entropy. DiPerna and Lions \cite{MR1014927} proved the existence of renormalized solutions for the cutoff Boltzmann equation. Later, 
 Alexandre and Villani \cite{MR1857879} proved the case of non-cutoff Boltzmann equation. Readers can also refer to \cite{MR2116276,MR972541,MR1278244,MR1392006,MR1650006}. However, uniqueness and regularity remain open problems. %Currently, only results on conditional regularity are available: Imbert, Silvestre, Golse, Mouhot, Vasseur, $\cdots$.

 In the perturbation framework, one studies solutions where the initial data is a small perturbation of a global equilibrium state, such as Maxwellian.
There have been extensive works on the global well-posedness of the cutoff or non-cutoff Boltzmann equation over the past decades.
%We first aim to establish the global well-posedness in the half space $\Omega=\mathbb{R}^2\times \mathbb{R}^+$ or $\mathbb{T}^2 \times \mathbb{R}^+$. There have been extensive studies on the global well-posedness for the Boltzmann equation.

On the one hand, spectral analysis provides an alternative approach. This approach was initiated by Ukai \cite{MR363332}
in the whole space $\mathbb{R}^3$. Later on, Ukai and Yang \cite{ukai2006boltzmann} obtained the optimal decay rate through spectral analysis and semi-group method. For exterior domains, it can be viewed as a compact perturbation of the whole space problem. Ukai and Asano \cite{ukai1983steady,ukai1986steady} investigated the exterior problem in both steady and unsteady cases.

On the other hand, energy method has been found to have broad application. Here, we refer to some relevant works. In the absence of boundaries, the most fundamental domain are the torus and the whole space.
In torus, \cite{duan2021global} established global solutions with exponential decay for the non-cutoff Boltzmann equation. 
In the whole space $\mathbb{R}^3$, it is well-known that we can only expect polynomial decay rates of solutions; see \cite{bouin2020hypocoercivity} for a general hypocoercivity approach. 
 Guo \cite{guo2004boltzmann,guo2010bounded} constructed a global solution without time decay via a nonlinear energy method and entropy method, respectively. \cite{MR4688691} established global solutions together with almost optimal decay rate for the non-cutoff Boltzmann equation via an $L^1_k \cap L^p_k$ method in the Fourier frequency space without relying on the embedding $H^2(\mathbb{R}^3)\subset L^\infty(\mathbb{R}^3)$.
While for an infinite channel $\mathbb{R}\times \mathbb{T}^2$, Wang and Wang \cite{Teng2019} investigated the time-asymptotic stability of planar rarefaction
 wave for the three-dimensional Boltzmann equation by high-order Sobolev energy methods. Besides, \cite{duan20213d} studied the Vlasov-Poisson-Landau system utilizing the same method.

When physical boundaries are present,
a major breakthrough was the development of the $L^2$–$L^\infty$ framework by Guo \cite{G}. The $L^2$ estimate relies on macro-micro decomposition, while the $L^\infty$ estimate is typically achieved through characteristics with repeated interaction with the boundary. Global solutions with exponential convergence rate near global Maxwellian were established under several boundary conditions, including the diffusive reflection and specular reflection. It has a significant impact on the study of kinetic boundary value problems, leading to a period of subsequent development \cite{CKL, duan2019effects, EGKM, EGKM2,KL}. 

When extending the $L^2$-$L^\infty$ framework to an infinite layer domain, it is natural to begin with the case where the tangent variable $(x_1,x_2)$ is bounded with specular boundary condition, or where $(x_1,x_2)\in \mathbb{T}^2$. Related works include \cite{chen_mixed} and \cite{duan20243d}. The first work focused on a mixed boundary condition in which 
specular reflection occurs at two parallel plates, whereas the remaining area between these two specular regions is diffusive reflection. The second work investigates the Couette flow in the region $\mathbb{T}^2 \times (-1,1)$ with diffusive boundary condition. 
In a recent work, \cite{chenduanzhang2024} employed the $L^1_k\cap L^p_k$ method in the horizontal direction combined with an $L^2$–$L^\infty$ framework in the vertical direction to the infinite layer $\mathbb{R}^2\times(-1,1)$. The authors constructed global solutions for the cutoff Boltzmann equation with diffuse boundary in both $\mathbb{R}^2\times(-1,1)$ and $\mathbb{R}\times(-1,1)$, and further established a polynomial decay rate for $\mathbb{R}^2\times(-1,1)$. In the first domain of $\mathbb{R}^2\times(-1,1)$, the authors utilize the $L^1_k \cap L^p_k$ approach in the Fourier space, noting that the additional $L^p_k$ norm provides sufficient time decay. And in the second domain of $\mathbb{R}\times(-1,1)$, the same method fails due to the slower time decay property of solutions along the one-dimensional horizontal direction. An alternative method is utilizing a time-derivative combined with a direct $L^2 \cap L^\infty$ approach in physical space, but the large-time behavior is left unknown. In this paper, we consider the half-space $\Omega=\mathbb{R}^2\times \mathbb{R}^+$. Our work not only proves the global well-posedness, but also further demonstrates global Gevrey solutions.

The dissipation estimate in \cite{G} for the macroscopic component crucially relies on Poincaré inequality, which restricts its direct application to unbounded domains.
$L^6$ control of macroscopic quantities is considerable, thereby extending the $L^2-L^\infty$ framework to the unbounded domain. It is initially proposed by Esposito-Guo-Kim-Marra \cite{EGKM,EGKM2}. This method crucially utilize the Sobolev embedding $W^{2,\frac{6}{5}}\subset H^1 \subset L^6$ and compactness of the boundary to obtain proper control over the trace in macroscopic estimate. 
Recently, Jung \cite{MR4891727} used $L^2-L^3-L^6$ approach to prove global diffusive expansion of the Boltzmann equation in exterior domain.
%We note that extensive studies have been conducted on

Research on the propagation of regularity for the Boltzmann equation has yielded notable results. 
One line of research is dedicated to Sobolev or Lebesgue regularity. %we refer the reader to 
These works \cite{MR1324404,MR2123118,MR1737547,MR816620,MR946968,MR2081030} demonstrate that the solution maintains the Sobolev or Lebesgue regularity with initial data.
%Desvillettes-Mouhot\cite{MR2123118}
%Desvillettes-Villani\cite{MR1737547}landau, homogeneous,hard potentital 
%Gustafsson 1986 \cite{MR816620}
%Gustafsson 1988 \cite{MR946968}
%Mouhot-Villani\cite{MR2081030}
Another significant line focuses on Gevrey regularity.
 A foundational result was established by Ukai \cite{MR839310}, who proved the propagation of the Gevrey regularity in a finite time for both the cutoff and non-cutoff Boltzmann equation.
Subsequently, Desvillettes-Furioli-Terraneo \cite{MR2465814} proved the uniform propagation of Gevrey regularity in all time for the spatially homogenous Boltzmann equation with Maxwellian molecules, using the Wild expansion and the characterization of Gevrey regularity by the Fourier transform.
Recently, Liu \cite{MR4408391} proved the propagation of Gevrey regularity for the spatially homogenous Boltzmann equation with soft potential, which extended \cite{MR2465814}.
Zhang and Yin \cite{MR2925909} first studied the propagation of Gevrey regularity for the spatially homogeneous Boltzmann equation with index $\gamma,s$ satisfying $\gamma+2s\in(-1,1),s\in (0,1/2)$, and then they \cite{MR3110569} extended this analysis to the inhomogeneous case.
 In a related context, Chen-Li-Xu \cite{MR2425602} established Gevrey regularity propagation for landau equation in both Maxwellian molecules and hard potential.
 Moreover, a parallel line of inquiry concerns the propagation of singularities. 
In \cite{MR2435186}, the authors proved that the singularities in the initial data propagated
 like the free transportation for the spatially inhomogeneous Boltzmann equation.

 It is well-known that the angular singularity will induce the fractional diffusion in velocity for the non-cutoff Boltzmann equation. Here, we briefly discuss works concerning the regularity of weak solutions, such as \cite{MR4375857,MR4930523,MR4356815}. In contrast, such diffusion effects are absent for the cutoff Boltzmann equation. So, it is natural to investigate whether we can obtain the global analytic solutions in both $x$ and $v$ when the initial data is analytic. Our main goal is to construct global analytic $x_\parallel$ and Gevrey $v_\parallel$ estimate in the half-space, under analytic initial data. We also mention the well-posedness in Gevrey space for Prandtl equation without structural assumption \cite{MR4465902}. In recent work, Li-Xu-Zhang \cite{li2025globalgevreysolution3d} demonstrated that solutions to the 3D anisotropic Navier-Stokes system in a strip domain lie in space-time analytic or Gevrey space, provided the initial data is in Gevrey class only for the vertical direction; that is, they give the propagation in the vertical direction and regularity in the horizontal direction.

\subsection{Strategy of the proof}\label{sec:proof_strategy}

\begin{itemize}
    \item Global well-posedness
\end{itemize}
We begin by considering global well-posedness. In the domain $\Omega = \mathbb{R}^2 \times \mathbb{R}^+$, the primary difficulty in estimating the macroscopic components arises from the absence of a Poincaré inequality and the presence of the boundary. Our key observation is to exploit the diffusive effect in the horizontal directions and apply the dual argument by Esposito-Guo-Kim-Marra \cite{EGKM}. This leads to a degenerate macroscopic dissipation estimate \eqref{degenerate_macro}. To construct this estimate, we introduce a degenerate test function \eqref{elliptic_a}, whose right-hand side carries the additional frequency weight $\frac{|k|^2}{1+|k|^2}$. This weight preserves both a degenerate elliptic regularity estimate and a trace estimate in the half-space $x_3 \in (0,\infty)$. This degenerate structure yields a decay rate analogous to that of the two-dimensional heat equation (see Lemma \ref{lemma:full_energy_decayexg}) in the same spirit of the $L^1_k \cap L^p_k$ method employed in \cite{chenduanzhang2024}.

\begin{itemize}
    \item Analytic regularity in $x_\parallel$.
\end{itemize}
Since derivatives in the horizontal direction $x_\parallel$ do not change the equation's structure, the proof of the analytic-in-$x_\parallel$ estimate follows from a similar treatment to the well-posedness problem, except for the nonlinear term, where the Leibniz’s formula is applied. The term $(m+2)^2$ in Definition \ref{defgevr2} is introduced precisely to ensure the convergence of the resulting series arising from Leibniz's formula.

\hide
Next, we construct the global analytic $x_\parallel$ estimate in the half-space
in Section \ref{anainx}. Here, a smallness condition on $x$-regularity of $f_0$ is required. The proof follows a similar structure to the well-posedness argument in Section \ref{sec:withoutder}, except for the nonlinear term. Therefore, we only provide a detailed analysis of the nonlinear term and Leibniz’s formula is applied to deal with the nonlinear term. Moreover, the term $(m+2)^2$ in Definition \ref{defgevr2} or $(|\zeta|+2)^2$ in Definition \ref{defgevt2} are introduced to guarantee the convergence of the summation arising from Leibniz's formula. Note that the term $16/\nu_0$ in Definition \ref{defgevr2} and \ref{defgevt2} is not needed in this section.
\unhide

\begin{itemize}
    \item Gevrey regularity in $v_\parallel$.
\end{itemize}
For the mixed Gevrey derivative $x_\parallel,v_\parallel$ estimate in the half-space, one obstacle lies in the $L^2$ estimate to the commutator term $[ L, \p_v ]f$ due to the lack of the Poincaré inequality in the half-space. Instead of attempting to overcome this difficulty within an $L^2$ framework, we change our approach and estimate the $v_\parallel$-derivative directly via the method of characteristics.

Along the characteristic, a natural polynomial growth in time occurs from the $v_\parallel$ derivative: 
\begin{align*}
    &\p_{v_\parallel}^\alpha [f_0(x-vt,v)]   \sim t^{\alpha}\p_{x_\parallel}^\alpha f_0(x-vt,v).
\end{align*}
To absorb this growth and obtain a global estimate, we exploit the damping term $e^{-\nu(v)t}$ along the characteristic and control the polynomial growth as(also see Lemma \ref{lemma:enut_control})
\begin{align*}
    &   e^{-\nu(v) t/8}t^{|\alpha|} \leq \bigg(\frac{8}{\nu_0 }\bigg)^{|\alpha|}|\alpha|!, \ \ \forall i\in \mathbb{Z}_+.
\end{align*}
The extra exponential term and factorial, together with the analytic regularity in $x_\parallel$, indicate that the Gevrey-$2$ regularity in $v_\parallel$ is necessary.

The main difficulty then lies in controlling the contribution of the Boltzmann collision operator (given in \eqref{chara:K} and \eqref{chara:Gamma}) in the Gevrey-2 norm. A key issue arises when applying high-order derivatives $\p_{v_\parallel}^\alpha$ to the linearized Boltzmann operator $Kf= \int_{\mathbb{R}^3}\mathbf{k}(v,u)\dd u$. The kernel derivative $ \int_{\mathbb{R}^3}\p_{v_\parallel}^\alpha\mathbf{k}(v,u)\dd u$ is singular and not integrable when $|\a|\geq 2$. To treat the singularity, we rewrite the derivative $\partial_{v_\parallel}^\alpha$ in the form of $(\partial_{v_\parallel} + \partial_{u_\parallel} - \partial_{u_\parallel})^{\alpha}$. Crucially, the combined operator $(\partial_{v_\parallel} + \partial_{u_\parallel})^\alpha$ acting on $\mathbf{k}(v,u)$ does not produce a singularity, since the singular term $\frac{1}{|v-u|}$ in $\mathbf{k}(v,u)$ is symmetric in $v$ and $u$. Besides the singularity issue, the complexity of $\mathbf{k}(v,u)$ makes high order derivative computations delicate. We address this issue by establishing an inductive formula for the derivative of $\mathbf{k}(v,u)$ in Lemma \ref{estdivk1ta}, which shows that $(\p_{v_\parallel}+\p_{u_\parallel})^\alpha$ of $\mathbf{k}(v,u)$ can be controlled by $|\alpha|!$. This bound is naturally compatible with a Gevrey-2 framework.

For the remaining $\partial_{u_\parallel}^\alpha$ terms, we integrate by parts and move $\partial_{u_\parallel}^{\alpha-1}$ derivatives onto $f$, while retain one derivative on $\mathbf{k}(v,u)$, since $\partial_{u_\parallel} \mathbf{k}(v,u)$ is integrable. We cannot move the full $\alpha$ derivatives to $f$, because the $L^2$ estimate for the commutator is unavailable as mentioned earlier. The mismatch between $\p_{v_\parallel}^\alpha$ on LHS and only $v_\parallel^{\alpha-1}$ on RHS leads to a smallness condition on the analytic radius $\rho \leq \frac{1}{2C}$ in Theorem \ref{gevr2}.

The estimate to the nonlinear term \eqref{chara:Gamma} can be handled through careful treatment of the Leibniz's formula under the Gevrey $2$ norm. The contribution of the boundary term in the \eqref{chara:bdr} is directly controllable because the $v_\parallel$ only acts on the Maxwellian $\sqrt{\mu(v)}$ in the boundary condition \eqref{f_eqn}. We refer to details of the discussion above to proof of Proposition \ref{prop:inftyxvr2}.

\hide

\subsection{Outline}

In Section \ref{sec:prelim}, we collect the necessary preliminary estimates. Section \ref{sec:withoutder} establishes the global well-posedness of the system in both $\mathbb{R}^2\times \mathbb{R}^+$ and $\mathbb{T}^2\times \mathbb{R}^+$, thereby completing the proofs of Theorem \ref{thm:l1k_lpkexg} and Theorem \ref{gevr2}. In Section \ref{anainx}, we prove analytic regularity in $x_\parallel$ for both domains. Finally, Section \ref{anainmixy} constructs the Gevrey-2 regularity in $v_\parallel$, thereby concluding both Theorem \ref{thm:t2} and Theorem \ref{gevt2}.

\unhide

\section{Preliminary}\label{sec:prelim}

In this section, we first give notations and then give basic inequalities and estimates on the linearized collision operator $\mathcal{L}$ and nonlinear collision operator $\Gamma(\cdot,\cdot)$ as in \eqref{def.L} and \eqref{def.Ga}.

\subsection{Notation}\label{sec:notation}
Denote by $[p]$ the largest integer less than or equal to $p$.
For $\alpha = (\alpha_1,\alpha_2)$, denote $|\alpha|=\alpha_1+\alpha_2, \ \a!=\a_1!\a_2!.$
Correspondingly, the binomial coefficient for multi-index is given by 
\begin{align*}
 \binom{\a}{\b}=\binom{\a_1}{\b_1}\binom{\a_2}{\b_2}. 
 \end{align*}

We use general norms:
\begin{align*}
 & \Vert f\Vert_{L^2_\nu} := \Vert \nu^{1/2}f(v)\Vert_{L^2_v}=\Big(\int_{\R^3}\nu(v)|f(v)|^2\dd v\Big)^{1/2}, \\
 & \Vert f\Vert_{L^2_T} := \Big(\int_0^T |f(t)|^2 \dd t \Big)^{1/2}, \\ 
 & \Vert f\Vert_{L^\infty_T} := \sup_{0\leq t\leq T}|f(t)|,\\
&\gamma_{\pm}=\{(x,v)\in\mathbb{R}^2 \times \{x_3=0\}\times\mathbb{R}^{3}: v_3 \lessgtr 0 \} , \\
 & | \hat{f}|_{L^2_{\gamma_+}} := \Big( \int_{v_3<0} |\hat{f}(k,0,v)|^2 |v_3| \dd v\Big)^{1/2}, \\
 & |\hat{f}|_{L^1_k L^2_{T,\gamma_+}} := \int_{\mathbb{R}^2} \Big( \int_0^T \int_{v_3<0} |\hat{f}(t,k,0,v)|^2 |v_3| \dd v \dd t \Big)^{1/2} \dd k, 
\end{align*}
and 
\begin{align*}
 & \Vert \hat{f}\Vert_{L^1_k L^\infty_{T,x_3,v}}:= \int_{\mathbb{R}^2} \sup_{0\leq t\leq T,x_3\in (0,+\infty),v\in \mathbb{R}^3} |\hat{f}(t,k,x_3,v)| \dd k, \\
 & \Vert \hat{f} \Vert_{L^1_k L^2_{T,x_3,v}} := \int_{\mathbb{R}^2} \Big(\int_{0}^T \int_{0}^{+\infty} \int_{\mathbb{R}^3} |\hat{f} (t,k,x_3,v)|^2 \dd v \dd x_3 \dd t \Big)^{1/2} \dd k , \\
 & \Vert \hat{f} \Vert_{L^1_k L^2_{T,x_3,\nu}} := \int_{\mathbb{R}^2} \Big(\int_{0}^T \int_{0}^{+\infty} \int_{\mathbb{R}^3}\nu(v) |\hat{f} (t,k,x_3,v)|^2 \dd v \dd x_3 \dd t \Big)^{1/2} \dd k , \\
 & \Vert \hat{f} \Vert_{L^1_k L^\infty_T L^2_{x_3,v}}:= \int_{\mathbb{R}^2} \sup_{0\leq t \leq T} \Big(\int_{0}^{+\infty} \int_{\mathbb{R}^3} |\hat{f}(t,k,x_3,v)|^2 \dd v \dd x_3 \Big)^{1/2} \dd k, \\
 & \Vert \hat{f}\Vert_{L^p_k L^\infty_T L^2_{x_3,v}}:= \Big(\int_{\mathbb{R}^2} \sup_{0\leq t \leq T} \Big(\int_{0}^{+\infty} \int_{\mathbb{R}^3} |\hat{f}(t,k,x_3,v)|^2 \dd v \dd x_3 \Big)^{p/2} \dd k \Big)^{1/p}, \\
 & \Vert \hat{f}\Vert_{L^p_k L^2_{x_3,v}} := \Big(\int_{\mathbb{R}^2} \Big(\int_{0}^{+\infty}\int_{\mathbb{R}^3} |\hat{f} (k,x_3,v)|^2 \dd v \dd x_3 \Big)^{p/2} \dd k\Big)^{1/p},
\end{align*}
with $1\leq p<\infty$, and for $p=\infty$, the norms of $L^\infty_k L^\infty_T L^2_{x_3,v}$ and $L^\infty_k L^2_{x_3,v}$ are similarly defined in the standard way.

Moreover, $f \lesssim g$ means that there exists $C>1$ such that $f\leq C g$, and $f\leq o(1)g$ and $f \lesssim o(1)g$ both mean that there exists $0<\delta_0\ll 1$ such that $f\leq \delta_0 g$.

\subsection{Elementary inequalities}
In this subsection, we list some elementary inequalities that will be used frequently.
\begin{lemma}
For $\alpha \geq \beta,$ it holds
\begin{equation}\label{absab}
 \binom{\a}{\b}\leq \binom{|\a|}{|\b|}, 
\end{equation}
and 
\begin{align}
 & m!n!\leq (m+n)!\leq 2^{m+n} m!n!, \text{ for } m,n\geq 0,\label{jiechenmn}
 \\
 &\frac{n!}{(m+n)!}\leq \frac{1}{ m!}, \ \frac{(m+n)!}{(m+l)!} \leq \frac{m!}{(m+l-n)!}, \text{ for } \ l\geq n, \label{framn} \\
 & \frac{(m+l)!}{(m+n)!} \leq \frac{(m+a+l)!}{(m+a+n)!} \text{ for } l\geq n, \ a\geq 0, \label{framn2} \\
 &\frac{[(m+2)/2]!C_0^{(m+1)/2}}{m!} \leq C, \text{ for } m\geq 0.\label{MN2}
\end{align}

\end{lemma}
\begin{proof}
We give the proof for the second inequality of \eqref{jiechenmn}. From the binomial theorem we have
 \begin{align*}
2^{m+n}=(1+1)^{m+n}=\sum_{j\leq m+n}\frac{(m+n)!}{j!(m+n-j)!}\geq \frac{(m+n)!}{n!m!}.
\end{align*}

Next, we prove \eqref{framn}.
The first inequality of \eqref{framn} follows directly from the first inequality of \eqref{jiechenmn}. For the second one, since $l\geq n$, we derive
 \begin{align*}
\frac{(m+n)!}{(m+l)!} =\frac{1}{\prod_{j=1}^{l-n} (m+n+j)},\quad \frac{m!}{(m+l-n)!}= \frac{1}{\prod_{j=1}^{l-n} (m+j)}.
\end{align*}
Both products contain $l-n$ factors. Moreover one has $m+n+j \geq m+j$ for each $j = 1,\dots, l-n$.
Hence
 \begin{align*}
\frac{1}{\prod_{j=1}^{l-n} (m+n+j)} \le \frac{1}{\prod_{j=1}^{l-n} (m+j)}.
\end{align*}
This completes the proof of \eqref{framn}.

Then, we prove \eqref{framn2}.
Since $l\ge n$, we derive
 \begin{align*}
\frac{(m+l)!}{(m+n)!}= \prod_{j=1}^{l-n} (m+n+j),\quad \frac{(m+a+l)!}{(m+a+n)!}= \prod_{j=1}^{l-n} (m+a+n+j).
\end{align*}
Both products contain $l-n$ factors. Each factor satisfies
 \begin{align*}
m+n+j\le m+a+n+j, \ \text{ for } j=1,\dots ,l-n.
\end{align*}
Consequently, we obtain the desired inequality.

Last, we give the proof of \eqref{MN2}.
Observe that 
 \begin{align*}
[(m+2)/2]!\leq \frac{m+2}{2}[m/2]!.
\end{align*}
Using the above inequality and the first inequality in \eqref{framn} we obtain
 \begin{align*}
\frac{[(m+2)/2]!C_0^{(m+1)/2}}{m!} \leq \frac{m+2}{2}\frac{[m/2]!C_0^{(m+1)/2}}{m!}\leq\frac{m+2}{2}\frac{C_0^{(m+1)/2}}{(m/2)!}\leq C.
\end{align*}
Here the last inequality holds because $(m/2)!$ grows faster than any polynomial and exponential function of $m$.
 This establishes \eqref{MN2}. Thus,  we complete the proof.
\end{proof}

Next, we establish some important estimates concerning the summations.
\begin{lemma}
It holds that
\begin{align}
 & \sum_{\beta\leq \alpha} \frac{(|\alpha|+2)^2}{(|\alpha-\beta|+2)^2(|\beta|+2)^2} \leq C, \label{able}\\
 & \sum_{\g\leq \b }\frac{(|\b|+2)^2}{(|\g|+2)^2}\bigg(\frac{1}{2}\bigg)^{|\b-\g|} \leq C, \label{abdiv2}\\
 & \sum_{\beta\leq \alpha} \frac{(|\alpha|+2)^2}{(|\beta|+2)^2} \frac{|C_0|^{\alpha-\beta}}{|\alpha-\beta|!} \leq C,\label{inmati}
 \\ & \sum_{\tau\leq \delta}\frac{(8\rho )^{|\tau|}}{|\tau|!}(|\tau|+2)^2 \leq C.\label{sum_tau_poly}
\end{align} 
\end{lemma}
\begin{proof}
We start with the first inequality \eqref{able}.
\begin{align*}
\sum_{\beta\leq \alpha} \frac{(|\alpha|+2)^2}{(|\alpha-\beta|+2)^2(|\beta|+2)^2} \leq C\sum_{|\b|=0}^{[ |\a|/2]} \frac{1}{(|\beta|+2)^2} +C\sum_{|\b|=[ |\a|/2]+1}^{[ |\a|]} \frac{1}{(|\alpha-\beta|+2)^2} \leq C.
\end{align*}

Next, we prove the second inequality \eqref{abdiv2}. We compute 
 \begin{align*}
\sum_{\g\leq \b }\frac{(|\b|+2)^2}{(|\g|+2)^2}\bigg(\frac{1}{2}\bigg)^{|\b-\g|} \leq C
\sum_{|\g|=0}^{[ |\b|/2]}(|\b|+2)^2\bigg(\frac{1}{2}\bigg)^{| \b-\g|}+C\sum_{|\g|=[ |\b|/2]+1}^{[ |\b|]}\bigg(\frac{1}{2}\bigg)^{| \b-\g|}\leq C.
\end{align*}
In the last inequality, we have used 
\begin{equation*}
 \sum_{|\g|=0}^{[ |\b|/2]}(|\b|+2)^2\bigg(\frac{1}{2}\bigg)^{| \b-\g|} \leq \sum_{|\g|=0}^{ [|\b|/2]}(|\b|+2)^2\bigg(\frac{1}{2}\bigg)^{| \b|/2} \leq (|\b|+2)^3\bigg(\frac{1}{2}\bigg)^{| \b|/2}\leq C.
\end{equation*}

Then, we prove the third inequality \eqref{inmati}. We compute 
\begin{align*}
&\sum_{\b\leq \a}\frac{(|\a|+2)^2}{(|\b|+2)^2}\frac{|C_0|^{|\a-\b|}}{|\a-\b|!} 
\leq
\sum_{|\b|=0}^{[ |\a|/2]}(|\a|+2)^2\frac{|C_0|^{|\a-\b|}}{|\a-\b|!} +C\sum_{|\b|=[ |\a|/2]+1}^{[ |\a|]}\frac{|C_0|^{|\a-\b|}}{|\a-\b|!} 
 \leq C.
\end{align*}
In the last inequality, we have used 
\begin{multline*}
 \sum_{|\b|=0}^{[ |\a|/2]}(|\a|+2)^2\frac{|C_0|^{|\a-\b|}}{|\a-\b|!} \leq \sum_{|\b|=0}^{[ |\a|/2]}\frac{(|\a|+2)^2}{(|\a-\b|!)^{1/2}}\frac{|C_0|^{|\a-\b|}}{(|\a-\b|!)^{1/2}} \\\leq \sum_{|\b|=0}^{[ |\a|/2]}\frac{(|\a|+2)^2}{[(|\a|/2)!]^{1/2}}\frac{|C_0|^{|\a-\b|}}{(|\a-\b|!)^{1/2}} \leq C \sum_{|\b|=0}^{[ |\a|/2]}\frac{|C_0|^{|\a-\b|}}{(|\a-\b|!)^{1/2}} \leq C.
\end{multline*}

Last, we prove \eqref{sum_tau_poly}. We compute
 \begin{align*}
 \sum_{\tau\leq \delta}\frac{(8\rho )^{|\tau|}}{|\tau|!}(|\tau|+2)^2 \leq\sum_{\tau\leq \delta}\frac{1}{(|\tau|!)^{1/2}}\leq C.
 \end{align*}
The first inequality holds because
 $(|\tau|!)^{1/2}$ grows faster than any polynomial and exponential function of $\tau$.
Thus,  we complete the proof.
\end{proof}

\begin{lemma}
 Recall $L_{\rho, m, |\a|}$  is defined in \eqref{r2dalp}.  It holds that
\begin{align}
\sum_{j=0}^{m}\frac{m!}{j!(m-j)!}\frac{L_{\rho,m,0}}{L_{\rho,j,0}L_{\rho,m-j,0}}\leq C.\label{sumlrhom0}
\end{align}
\end{lemma}
\begin{proof}
 For any $0\leq j\leq m$, we compute
 \begin{align}\label{mjbion}
&\frac{m!}{j!(m-j)!}\frac{L_{\rho,m,0}}{L_{\rho,j,0}L_{\rho,m-j,0}}\notag\\ 
 &=\frac{m!}{j!(m-j)!}\frac{(16/\nu_0)^m\rho^{m+1}(m+2)^2}{m!}\frac{j!}{(16/\nu_0)^j\rho^{j+1}(j+2)^2}\frac{(m-j)!}{(16/\nu_0)^{m-j}\rho^{m-j+1}(m-j+2)^2}\notag\\ 
 &\leq\rho^{-1}\frac{(m+2)^2}{(j+2)^2(m-j+2)^2}.
 \end{align}
  Taking summation in $j$, we use \eqref{able} with replacing $|\a|$ and $|\b|$ by $m$ and $j$ respectively to obtain 
 \begin{align*}
\sum_{j=0}^{m}\eqref{mjbion}\leq\rho^{-1}\sum_{j=0}^{m}\frac{(m+2)^2}{(j+2)^2(m-j+2)^2} \leq C.
 \end{align*}
  Therefore, this completes the proof.
\end{proof}

\subsection{Basic properties of the linear operator $\mathcal{L}$ and $K$}

For $\mathcal{L}$, we have the following three lemmas.

\begin{lemma}[\cite{R}]\label{lemma:basic_K}
It holds that $\mathcal{L}=\nu(v)-K$, where 
\begin{equation*}
\nu(v)=\int_{\mathbb{R}^3}\int_{\mathbb{S}^2}|(v-u)\cdot \omega|\mu(u)\,\d\omega\d u,
\end{equation*} 
and
\begin{equation*}
Kf(v)=\int_{\mathbb{R}^3}\int_{\mathbb{S}^2}B(v-u,\omega)[\sqrt{\mu(v)\mu(u)}f(u)-\sqrt{\mu(u)\mu(u')}f(v')-\sqrt{\mu(u)\mu(v')}f(u')]\,\d\omega\d u.
\end{equation*}
Here, the collision frequency $\nu(v)$ satisfies
\begin{equation*}
\nu(v) \geq \nu_0 \sqrt{|v|^2+1} \geq \nu_0
\end{equation*}
for a positive constant $\nu_0>0$. The integral operator $K$ is given by
\[
Kf(x,v)=\int_{\mathbb{R}^3}\mathbf{k}(v,u)f(x,u)\,\dd u,
\]
with the integral kernel $\mathbf{k}(v,u)$ satisfying
\Be\notag%\label{k_varrho}
 |\mathbf{k} (v,u)| \lesssim \mathbf{k}_\varrho (v,u), \ \ \mathbf{k}_\varrho (v,u) := e^{- \varrho |v-u|^2}/ |v-u|,
\Ee
for a constant $\varrho>0$. %Let $0<\lambda<1$ be a constant such that $\lambda<\nu_0$.
\end{lemma}

Note that $\nu(v)$ serves as a damping term along the characteristic.  When the exponential $e^{-\nu(v)t}$ meets the polynomial $t^i$, we have the following estimate.

\begin{lemma}\label{lemma:enut_control}
 \begin{align}
 & e^{-\nu(v) t/8}t^{i} \leq \bigg(\frac{8}{\nu_0 }\bigg)^{i}i!, \ \ \forall i\in \mathbb{Z}_+.\label{derienut}
 % \\ & e^{- \nu(v)t/4} (1+t)^{\sigma/2} \lesssim e^{- \nu_0t/4} (1+t)^{\sigma/2} \lesssim 1,\label{tenut1t}
 \end{align}
\end{lemma}

\begin{proof}
The lemma follows from straightforward calculation:
\begin{align*}
 & e^{-\nu(v) t/8}t^{i} \leq e^{-\nu_0t/8} \bigg(\frac{\nu_0 t}{8}\bigg)^{i} \bigg(\frac{8}{\nu_0 }\bigg)^{i} \leq \bigg(\frac{8}{\nu_0 }\bigg)^{i}i!, \ \ \forall i\in \mathbb{Z}_+.
 % \\ & e^{- \nu(v)t/4} (1+t)^{\sigma/2} \lesssim e^{- \nu_0t/4} (1+t)^{\sigma/2} \lesssim 1,\label{tenut1t}
 \end{align*}
\end{proof}

\begin{lemma}[\cite{R}]%\label{lemma:k_theta}
Let $0\leq \theta < \frac{1}{4}$, and $\mathbf{k}_\theta(v,u) := \mathbf{k}(v,u) \frac{e^{\theta |v|^2}}{e^{\theta |u|^2}}$, then there exists $C_\theta > 0$ such that
\begin{equation}\label{k_theta}
\int_{\mathbb{R}^3} \mathbf{k}(v,u) \frac{e^{\theta |v|^2}}{e^{\theta |u|^2}} \,\dd u \leq \frac{C_\theta}{1+|v|}.
\end{equation}
Moreover, for $N\gg 1$, we have
\begin{equation*}
\mathbf{k}_\theta(v,u) \mathbf{1}_{|v-u|> \frac{1}{N}} \leq C_N,
\end{equation*}
and
\begin{equation*}
\int_{|u|>N \text{ or } |v-u|\leq \frac{1}{N}} \mathbf{k}_\theta(v,u) \,\dd u \lesssim \frac{1}{N} \leq o(1).
\end{equation*}
\end{lemma}

\subsection{Basic properties of the nonlinear operator $\Gamma$}
 
The estimate for the nonlinear operator is given by the following lemma.

\begin{lemma}[\cite{chenduanzhang2024}]%\label{lemma:gamma_est}
For $1\leq p\leq \infty$, we have the following estimates to the nonlinear operator $\hat{\Gamma}(\hat{f},\hat{g})$:
\begin{align}
 & \Big| \int_{\mathbb{R}^3} \hat{\Gamma}(\hat{f},\hat{g}) \bar{\hat{h}}(k) \dd v \Big| \lesssim \int_{\mathbb{R}^2} \Vert \hat{f}(k-\ell)\Vert_{L^2_v} \Vert \hat{g}(\ell)\Vert_{L^2_\nu} \Vert (\mathbf{I}-\mathbf{P})\hat{h}\Vert_{L^2_\nu} \dd \ell, 
 \label{gamma_product}
\end{align}
\begin{align}
 & \Big\Vert \Big|\int_0^T \int_{0}^\infty \int_{\mathbb{R}^3} \hat{\Gamma}(\hat{f},\hat{g})\bar{\hat{h}}(k) \dd v \dd x_3 \dd t \Big|^{1/2} \Big\Vert_{L^p_k} \notag \\
 &\lesssim o(1) \Vert (\mathbf{I}-\mathbf{P}) \hat{h}\Vert_{L^p_k L^2_{T,x_3,\nu}} + \Vert \hat{f} \Vert_{L^p_k L^\infty_T L^2_{x_3,v}} \Vert \hat{g}\Vert_{L^1_k L^2_T L^\infty_{x_3} L^2_\nu} \notag \\
 &\lesssim o(1) \Vert (\mathbf{I}-\mathbf{P}) \hat{h}\Vert_{L^p_k L^2_{T,x_3,\nu}} + \Vert \hat{f} \Vert_{L^p_k L^\infty_T L^2_{x_3,v}} \Vert w\hat{g}\Vert_{L^1_k L^2_T L^\infty_{x_3,v}}, \notag
\end{align}
\begin{align}
 & \Big\Vert \Big|\int_0^T \int_{0}^\infty \int_{\mathbb{R}^3} (1+t)^\sigma \hat{\Gamma}(\hat{f},\hat{g})\bar{\hat{h}}(k) \dd v \dd x_3 \dd t \Big|^{1/2} \Big\Vert_{L^1_k}\notag \\
 &\lesssim o(1) \Vert (1+t)^{\sigma/2}(\mathbf{I}-\mathbf{P}) \hat{h}\Vert_{L^1_k L^2_{T,x_3,\nu}} + \Vert (1+t)^{\sigma/2} \hat{f}\Vert_{L^1_k L^\infty_T L^2_{x_3,v}} \Vert w\hat{g}\Vert_{L^1_k L^2_T L^\infty_{x_3,v}}, 
\notag
\end{align}
and
\begin{align} 
 \Vert \nu^{-1}(1+t)^{\sigma/2} w\hat{\Gamma}(\hat{f},\hat{g})\Vert_{L^\infty_{T,x_3,v}} \lesssim \Vert (1+t)^{\sigma/2} w\hat{f}\Vert_{L^\infty_{T,x_3,v}}*_k\Vert (1+t)^{\sigma/2} w\hat{g}\Vert_{L^\infty_{T,x_3,v}}, \notag
\end{align}
\begin{align}
 \Vert \nu^{-1}(1+t)^{\sigma/2}w\hat{\Gamma}(\hat{f},\hat{g}) \Vert_{L^1_k L^\infty_{T,x_3,v}} \lesssim \Vert (1+t)^{\sigma/2} w \hat{f}\Vert_{L^1_k L^\infty_{T,x_3,v}} \Vert (1+t)^{\sigma/2} w \hat{g}\Vert_{L^1_k L^\infty_{T,x_3,v}}, \notag
\end{align}
where all estimates are independent of $T>0$.
\end{lemma}

\subsection{Derivative estimate to the linear operator $\nu(v)$ and $K$}

\begin{lemma}\label{lemma:deri_nu}
 There exists a constant $L$ %which depends only on $\nu_0$ 
 such that 
\begin{align}
 & |\p_{v_\parallel}^\a e^{-\nu(v)t}| \lesssim L^{|\a|} |\a| ! e^{-\nu(v)t/2},\ \ \forall \a\in \mathbb{Z}_+^2. \label{tenut}
 \end{align}
\end{lemma}

\begin{proof}
As a preliminary step, we first show that 
 \begin{align} \label{derinu}
&\p_{v_\parallel}^{\a}\nu(v)=\int_{\mathbb{R}^3}\int_{\mathbb{S}^2}\p_{v_\parallel}^{\a}|(v-u)\cdot \omega|\mu(u)\,\d\omega\d u=\int_{\mathbb{R}^3}\int_{\mathbb{S}^2}[(-\p_{u_\parallel})^{\a}|(v-u)\cdot \omega|]\mu(u)\,\d\omega\d u\notag\\
 &=\int_{\mathbb{R}^3}\int_{\mathbb{S}^2}|(v-u)\cdot \omega|\p_{u_\parallel}^{\a}\mu(u)\,\d\omega\d u\leq 8^{|\a|+1}|\a|!\int_{\mathbb{R}^3}\int_{\mathbb{S}^2}|(v-u)\cdot \omega|\mu^{1/2}(u)\,\d\omega\d u\notag \\
 &\lesssim 8^{|\a|+1}|\a|!\nu(v),
\end{align}
where for the inequality in the second line, we have used:
\begin{equation}\label{kmu}
\forall \g \in \mathbb{Z}_+^2, \ |\g | \geq 0,  \ |\p_{v_\parallel}^{\g}\mu^{1/2}|\leq 8^{|\g |+1}|\g |!\mu^{1/4}.
\end{equation}
Then we give the two-dimensional Faà di Bruno formula: for the composite function $h(v) = f(g(v))$, %where $\mathbf{v} = (v_1, v_2),\ \alpha = (\alpha_1,\alpha_2)\in\mathbb{Z}_+^2$,\ $\beta = (\beta_1,\beta_2)\in\mathbb{Z}_+^2$,
\begin{equation*}
\partial^\alpha h = \alpha! \sum \frac{f^{(m)}(g(v))}{\prod_{\beta} k_\beta!} \prod_{\beta} \left( \frac{ \p^\beta g}{\beta!} \right)^{k_\beta}.
\end{equation*}
Here, $\beta = (\beta_1,\beta_2) \in \mathbb{N}_0^2$ with $|\beta| = \beta_1 + \beta_2 \ge 1$, and $k_\beta$ is a non‑negative integer indicating how many times the partial derivative $\partial^{\beta}$  appears.
 The summation is taken over all non-negative integer sequences ${k_{\beta}}$ satisfying $\sum_{\beta} \beta k_{\beta} = \alpha$. And $m=\sum_{\beta} k_{\beta} $.
 Therefore,   the multi‑index $\beta$ has to  satisfy   $\beta_1\leq \a_1, \beta_2\leq \a_2$.

Let $f(x)=e^x$ and $g(v)=-\nu(v)t$. We derive
\begin{equation*}
\partial_{v_\parallel}^\alpha e^{-\nu(v)t} = \alpha! \sum%_{\substack{k_\beta }} 
\frac{e^{-\nu(v)t}}{\prod_{\beta} k_\beta!} \prod_{\beta} \left( \frac{ t\p_{v_\parallel}^{\b}\nu(v)}{\beta!} \right)^{k_\beta}(-1)^{k_\beta}
.
\end{equation*}
Then we use \eqref{derinu} and to get 
\begin{align*}
|\partial_{v_\parallel}^\alpha e^{-\nu(v)t}|& \leq \alpha!e^{-\nu(v)t/2}\sum%_{\substack{k_\beta }}
\frac{e^{-\nu(v)t/2}}{\prod_{\beta} k_\beta!} \prod_{\beta} \left( \frac{8^{|\b|+1}|\beta|!\nu(v)t}{\beta!}\right)^{k_\beta}
\\& \leq \alpha!e^{-\nu(v)t/2}\sum%_{\substack{k_\beta }}
\frac{e^{-\nu(v)t/2}}{\prod_{\beta} k_\beta!} \prod_{\beta} \left( 16^{|\b|+1}\nu(v)t\right)^{k_\beta}
\\&\leq \alpha!e^{-\nu(v)t/2}\sum%_{\substack{k_\beta }}
\frac{1}{\prod_{\beta} k_\beta!} \prod_{\beta}16^{(|\b|+1)k_\beta} k_\beta!2^{k_\beta}
.
\end{align*}
Here in the second line, we use \eqref{jiechenmn} for the second inequality; that is 
$$\frac{|\beta|!}{\beta!}=\frac{(\beta_1+\beta_2)!}{\beta_1!\beta_2!}\leq 2^{|\beta|}.$$
Note that $\sum \beta k_\beta = \alpha$. So $\sum|\beta| k_\beta = |\alpha|,\ \sum k_\beta \leq |\alpha|$. Besides, the total number of summation is bounded by $4^{|\alpha|}$. Therefore, there exists a constant $L$ large enough such that 
\begin{align*}
|\partial_{v_\parallel}^\alpha e^{-\nu(v)t}|&\leq\alpha!e^{-\nu(v)t/2}\sum%_{\substack{k_\beta }}
\prod_{\beta}16^{(|\b|+1)k_\beta} 2^{k_\beta}\leq \alpha!e^{-\nu(v)t/2}16^{(|\a|+|\a|)}2^{|\a|}4^{|\alpha|}
\leq L^{|\a|} |\a| ! e^{-\nu(v)t/2}.
\end{align*}
We then complete the proof of Lemma \ref{lemma:deri_nu}.
\end{proof}

\begin{lemma}\label{estdivk1ta}
Let $\tau=(\tau_1, \tau_2)\in \mathbb{Z}^2_+.$ Denote $(\p_{v_\parallel}+\p_{u_\parallel})^{\tau} : =(\p_{v_1}+\p_{u_1})^{\tau_1}(\p_{v_2}+\p_{u_2})^{\tau_2}$ and $\iota=(1,0)$ or $(0,1)$.
The following derivative estimate of $\mathbf{k}$ in Lemma \ref{lemma:basic_K} holds:
 \begin{align*}
\bigg|\int_{\mathbb{R}^3}\frac{w(v)}{w(u)} \p_{v_\parallel}^\iota(\p_{v_\parallel}+\p_{u_\parallel})^{\tau} \mathbf{k}(v,u) \dd u \bigg|\leq C|\tau|!\nu(v).
\end{align*} 

\end{lemma}
\begin{proof}
Recall the Grad estimate \cite{R}
$\mathbf{k}(v,u)=\mathbf{k_1}(v,u)+\mathbf{k_2}(v,u)$ with
\begin{align*}
 &\mathbf{k_1}(v,u)=|v-u|e^{-\frac{1}{4}(|v|^2+|u|^2)},
 \\&\mathbf{k_2}(v,u)=|v-u|^{-1}e^{-\frac{1}{8}|v-u|^2-\frac{1}{8}\frac{(|v|^2-|u|^2)^2}{|v-u|^{2}}}.
\end{align*}

The high order derivative to $\mathbf{k}_1$ and $\mathbf{k}_2$ involves inductive application of chain rule. Thus, for the bound of $\mathbf{k_1}$, we denote 
\begin{align*}
&\bar A=-\frac{1}{4}(|v|^2+|u|^2), \ \bar{B}_1 = -\frac{v_1+u_1}{2}, \ \bar B_2=-\frac{v_2+u_2}{2}, \ \bar C =-\frac{1}{2}.
\end{align*}

In view of Lemma \ref{lemma:n_derivative} in the appendix, we obtain
\begin{align*}
& (\p_{v_1}+\p_{u_1})^{\tau_1}(\p_{v_2}+\p_{u_2})^{\tau_2}e^{-\frac{1}{4}(|v|^2+|u|^2)}
\\&=(\p_{v_1}+\p_{u_1})^{\tau_1} \bigg[e^{\bar A}\sum_{m=0}^{[\tau_2/2]}\frac{\tau_2!}{m!(\tau_2-2m)!}\bigg(\frac{\bar C}{2}\bigg)^m\bar B_2^{\tau_2-2m}\bigg]
\\&=\sum_{m=0}^{[\tau_2/2]}\frac{\tau_2!}{m!(\tau_2-2m)!}\bigg(\frac{\bar C}{2}\bigg)^m\bar B_2^{\tau_2-2m}(\p_{v_1}+\p_{u_1})^{\tau_1} e^{\bar A}
\\&=\sum_{m=0}^{[\tau_2/2]}\frac{\tau_2!}{m!(\tau_2-2m)!}\bigg(\frac{\bar C}{2}\bigg)^m\bar B_2^{\tau_2-2m}\sum_{j=0}^{[\tau_1/2]}\frac{\tau_1!}{j!(\tau_1-2j)!}e^{\bar A}\bigg(\frac{\bar C}{2}\bigg)^j \bar{B}_1^{\tau_1-2j}.
\end{align*}
In the second equality, we applied the property that $(\p_{v_1} + \p_{u_1})\bar{B}_2 = (\p_{v_1}+\p_{u_1})\bar{C} = 0$.

This yields 
 \begin{align*}
& \bigg|\int_{\mathbb{R}^3}\frac{w(v)}{w(u)} \p_{v_\parallel}^\iota(\p_{v_\parallel}+\p_{u_\parallel})^{\tau} \mathbf{k_1}(v,u) \dd u\bigg| \\&= \bigg|\int_{\mathbb{R}^3}\frac{w(v)}{w(u)} \p_{v_\parallel}^\iota \Big[|v-u|(\p_{v_\parallel}+\p_{u_\parallel})^{\tau}e^{-\frac{1}{4}(|v|^2+|u|^2)}\Big] \dd u\bigg| \\&=\sum_{m=0}^{[\tau_2/2]}\frac{\tau_2!}{m!(\tau_2-2m)!}\sum_{j=0}^{[\tau_1/2]}\frac{\tau_1!}{j!(\tau_1-2j)!}\bigg|\bigg(\frac{\bar C}{2}\bigg)^m\bigg(\frac{\bar C}{2}\bigg)^j\int_{\mathbb{R}^3}\frac{w(v)}{w(u)} \p_{v_\parallel}^\iota\bigg(|v-u|e^{\bar A}\bar B_2^{\tau_2-2m}\bar B_1^{\tau_1-2j}\bigg)\dd u\bigg| 
\\&=\tau_1!\tau_2!\sum_{m=0}^{[\tau_2/2]}\frac{1}{m!(\tau_2-2m)!}\sum_{j=0}^{[\tau_1/2]}\frac{1}{j!(\tau_1-2j)!}\\&\times\bigg|\bigg(\frac{1}{4}\bigg)^m\bigg(\frac{1}{4}\bigg)^j\int_{\mathbb{R}^3}\frac{w(v)}{w(u)} \p_{v_\parallel}^\iota\bigg[|v-u|e^{-\frac{1}{4}(|v|^2+|u|^2)}\bigg(-\frac{v_2+u_2}{2}\bigg)^{\tau_2-2m}\bigg(-\frac{v_1+u_1}{2}\bigg)^{\tau_1-2j}\bigg]\dd u\bigg| 
\\&\leq C\tau!\nu(v).
\end{align*}
In the last inequality, we have applied Lemma \ref{lemma:derk12} in the appendix.

The rest is devoted to proving the bound of $\mathbf{k_2}$.
Denote 
\begin{align}
&A=-\frac{(|v|^2-|u|^2)^2}{8|v-u|^2},\ B_1=-\frac{(|v|^2-|u|^2)(v_1-u_1)}{2|v-u|^2}, \ B_2=-\frac{(|v|^2-|u|^2)(v_2-u_2)}{2|v-u|^2}, \notag\\
&C_1 =-\frac{(v_1-u_1)^2}{|v-u|^2}, \ C_2 =-\frac{(v_2-u_2)^2}{|v-u|^2}, \ D=-\frac{(v_1-u_1)(v_2-u_2)}{|v-u|^2}. \label{ABCDEF}
\end{align}

In view of Lemma \ref{lemma:n_derivative} in the appendix, we obtain
 \begin{align*}
&(\p_{v_1}+\p_{u_1})^{\tau_1}(\p_{v_2}+\p_{u_2})^{\tau_2}e^{-\frac{1}{8}\frac{(|v|^2-|u|^2)^2}{|v-u|^{2}}}\\
&=(\p_{v_1}+\p_{u_1})^{\tau_1} \bigg[e^A\sum_{m=0}^{[\tau_2/2]}\frac{\tau_2!}{m!(\tau_2-2m)!}\bigg(\frac{C_2}{2}\bigg)^m B_2^{\tau_2-2m}\bigg]\\
&=\sum_{m=0}^{[\tau_2/2]}\frac{\tau_2!}{m!(\tau_2-2m)!}\bigg(\frac{C_2}{2}\bigg)^m\sum_{j=0}^{\tau_1}\frac{\tau_1!}{j!(\tau_1-j)!}\big[(\p_{v_1}+\p_{u_1})^je^A\big](\p_{v_1}+\p_{u_1})^{\tau_1-j}B_2^{\tau_2-2m}\\
&=\sum_{m=0}^{[\tau_2/2]}\frac{\tau_2!}{m!(\tau_2-2m)!}\bigg(\frac{C_2}{2}\bigg)^m\sum_{j=0}^{\tau_1}\frac{\tau_1!}{j!(\tau_1-j)!}\sum_{p=0}^{[j/2]}\frac{j!}{p!(j-2p)!}e^A\bigg(\frac{C_1}{2}\bigg)^pB_1^{j-2p}\\
&\times\frac{(\tau_2-2m)!}{[\tau_2-2m-(\tau_1-j)]!}B_2^{\tau_2-2m-(\tau_1-j)}D^{\tau_1-j}\\
&=\tau_1!\tau_2!\sum_{m=0}^{[\tau_2/2]}\frac{1}{m!}\sum_{j=0}^{\tau_1}\frac{1}{(\tau_1-j)!}\frac{1}{[\tau_2-2m-(\tau_1-j)]!}\sum_{p=0}^{[j/2]}\frac{1}{p!(j-2p)!}\\
&\times\bigg(\frac{C_2}{2}\bigg)^me^A\bigg(\frac{C_1}{2}\bigg)^p B_1^{j-2p}B_2^{\tau_2-2m-(\tau_1-j)}D^{\tau_1-j}.
\end{align*}
In the third line, we have used Leibniz's formula.

This yields 
 \begin{align*}
& \bigg|\int_{\mathbb{R}^3}\frac{w(v)}{w(u)} \p_{v_\parallel}^\iota(\p_{v_\parallel}+\p_{u_\parallel})^{\tau} \mathbf{k_2}(v,u) \dd u\bigg| \\
&= \bigg|\int_{\mathbb{R}^3}\frac{w(v)}{w(u)} \p_{v_\parallel}^\iota \Big[|v-u|^{-1}e^{-\frac{1}{8}|v-u|^{2}}(\p_{v_\parallel}+\p_{u_\parallel})^{\tau}e^{-\frac{1}{8}\frac{(|v|^2-|u|^2)^2}{|v-u|^{2}}}\Big] \dd u\bigg| \\
& =\tau_1!\tau_2!\sum_{m=0}^{[\tau_2/2]}\frac{1}{m!}\sum_{j=0}^{\tau_1}\frac{1}{(\tau_1-j)!}\frac{1}{[\tau_2-2m-(\tau_1-j)]!}\sum_{p=0}^{[j/2]}\frac{1}{p!(j-2p)!}\\&\times \bigg|\int_{\mathbb{R}^3}\frac{w(v)}{w(u)} \p_{v_\parallel}^\iota\bigg[|v-u|^{-1}e^{-\frac{1}{8}|v-u|^{2}}\bigg(\frac{C_2}{2}\bigg)^me^A\bigg(\frac{C_1}{2}\bigg)^pB_1^{j-2p}B_2^{\tau_2-2m-(\tau_1-j)}D^{\tau_1-j}\bigg] \dd u \bigg|
\\& \leq C \tau! \nu(v)% \bigg|\int_{\mathbb{R}^3}\frac{w(v)}{w(u)}\big(1+|v-u|^{-1}+|v-u|^{-2}\big)e^{-\frac{1}{8}|v-u|^{2}}e^{-\frac{1}{8}\frac{(|v|^2-|u|^2)^2}{|v-u|^{2}}} \dd u \bigg|
.
\end{align*}
In the last inequality, we have applied Lemma \ref{lemma:derk12} in the appendix. We then complete the proof of Lemma \ref{estdivk1ta}.
\end{proof}

\subsection{Derivative estimate to the nonlinear operator $\Gamma$} 
%The following estimates holds for the nonlinear operator. 
Recall  $\hat{\mathcal{T}}(\hat{f}, \hat g, h)$ is defined in \eqref{def.hatgamma}.

\begin{lemma}
For any multi-indices $\a$ and $\g$, we have the following estimate:
  \begin{align} \label{gamma_estk}
 \Vert \nu^{-1} (1+t)^{\sigma/2}w\hat{\mathcal T}( \hat f, \hat g,\p_{v_\parallel}^{\a}\mu^{1/2})\Vert_{L^\infty_{x_3,v}} \leq 8^{|\a|+1}|\a|! \Vert (1+t)^{\sigma/2}w\hat f\Vert_{L^\infty_{x_3,v}}*_k\Vert(1+t)^{\sigma/2} w\hat g\Vert_{L^\infty_{x_3,v}}.
\end{align}
Moreover, it holds 
 \begin{align} \label{gamma_dexhk}
 \Vert \nu^{-1} (1+t)^{\sigma/2}w\p_{v_\parallel}^{\g}\hat{\mathcal T}( \hat f, \hat g,&\mu^{1/2})\Vert_{L^\infty_{x_3,v}} \leq\sum_{\delta\leq \g}\sum_{\tau\leq \delta} \binom{\g}{\delta}\binom{\delta}{\tau}8^{|\tau|+1}|\tau|!\notag\\
	& \times \Vert (1+t)^{\sigma/2}w\p_{v_\parallel}^{\g-\delta}\hat f\Vert_{L^\infty_{x_3,v}}*_k\Vert(1+t)^{\sigma/2} w\p_{v_\parallel}^{\delta-\tau}\hat g\Vert_{L^\infty_{x_3,v}}.
\end{align}

\end{lemma}
\begin{proof}
 We first give the proof of \eqref{gamma_estk}.
We use \eqref{kmu}
to compute
\begin{align*}
&|w(v)\hat{\mathcal T}( \hat f, \hat g,\p_{v_\parallel}^{\a}\mu^{1/2})| \\
&= \Big|w(v)\int_{\mathbb{R}^3}\int_{\mathbb{S}^2} |(v-u)\cdot \omega| (\p_{v_\parallel}^{\a}\mu^{1/2}) \big[\hat f(u')*_k \hat g(v')-\hat f(u)*_k \hat g(v)\big] \dd \omega\dd u \Big|\\
 & \leq8^{|\a|+1}|\a|! w(v) \Big|\int_{\mathbb{R}^3}\int_{\mathbb{S}^2}|(v-u)\cdot \omega| \mu^{1/4}(u)\big[\hat f(u')*_k \hat g(v')-\hat f(u)*_k \hat g(v)\big] \dd \omega\dd u \Big| \\
 & \leq8^{|\a|+1}|\a|! w(v) \Big| \int_{\mathbb{R}^3} |v-u|\mu^{1/4}(u)w^{-1}(v)w^{-1}(u) \Vert w \hat f \Vert_{L^\infty_v} *_k\Vert w \hat  g\Vert_{L^\infty_v} \dd u \Big| \\
 & \leq 8^{|\a|+1}|\a|!\nu(v) \Vert w \hat f \Vert_{L^\infty_{v}}*_k \Vert w\hat g\Vert_{L^\infty_{v}} . 
\end{align*}
Together with the extra term $\nu^{-1}$, we take $L^\infty$ in $x_3$ and conclude \eqref{gamma_estk}.

Then, we prove  \eqref{gamma_dexhk}.
Leibniz's formula gives 
\begin{align*}
\p_{v_\parallel}^{\g}\hat{\mathcal T}( \hat f, \hat g,&\mu^{1/2})=\sum_{\delta\leq \g}\sum_{\tau\leq \delta} \binom{\g}{\delta}\binom{\delta}{\tau}\hat{\mathcal T}( \p_{v_\parallel}^{\g-\delta}\hat f, \p_{v_\parallel}^{\delta-\tau}\hat g,\p_{v_\parallel}^{\tau}\mu^{1/2}).
\end{align*}
We use \eqref{gamma_estk} to conclude \eqref{gamma_dexhk}.
Thus, we complete the proof.
\end{proof}

\section{Global well-posedness and asymptotic stability in $\mathbb{R}^2\times \mathbb{R}^+$ }\label{sec:withoutder}

In this section, we focus on the well-posedness of the equation $\hat f$ in \eqref{f_eqn_r2} in the domain of $\mathbb{R}^2\times\mathbb{R}^+$. We use the method with time-weighted introduced in Section 3 and 4 of \cite{chenduanzhang2024}. To prove Theorem \ref{thm:l1k_lpkexg}, we will establish the global-in-time existence of low-regularity solutions in the domain of $\mathbb{R}^2\times\mathbb{R}^+$ with the corresponding estimates \eqref{f_estimateexg} under the smallness condition \eqref{initial_assumptionexg}.

Below, we point out the main difference to \cite{chenduanzhang2024}.
In \cite{chenduanzhang2024}, $x_3\in (-1,1)$ is bounded, and the estimate can be closed in $L^\infty_{x_3}$, as $L^\infty_{x_3}\subset L^2_{x_3}$. In our setting, $x_3\in \mathbb{R}^+$ is unbounded. Consequently, we must close the estimate under the space $L^2_{x_3}\cap L^\infty_{x_3}$, also see Remark \ref{remark:L2_Linfty}. 

In general, the Poincaré inequality fails on an unbounded domain, making it challenging to obtain a zero-order macroscopic dissipation estimate. However, the degenerate macroscopic dissipation estimate remains valid even on an unbounded domain, due to the presence of the extra horizontal frequency variables in the elliptic system \eqref{elliptic_a}. Consequently, we obtain the following crucial $L^2$ energy estimate and macroscopic dissipation estimate in Lemma \ref{prop:full energyexg}.

\begin{lemma}[\textbf{Energy estimate}]\label{prop:full energyexg}
Let $\hat{f}$ be the solution to \eqref{f_eqn_r2}, with initial condition $\hat{f}_0$ satisfying \eqref{initial_assumptionexg}, then
\begin{align*}
 & \Vert \hat{f}\Vert_{L^1_k L^\infty_T L^2_{x_3,v}} + \Vert (\mathbf{I}-\mathbf{P})\hat{f}\Vert_{L^1_k L^2_{T,x_3,\nu}} + |(1-P_\gamma)\hat{f}|_{L^1_k L^2_{T,\gamma_+}} + \Big\Vert \frac{|k|}{\sqrt{1+|k|^2}} (\hat{a},\hat{\mathbf{b}},\hat{c})\Big\Vert_{L^1_k L^2_{T,x_3}} \\
 &\lesssim \Vert \hat{f}_0\Vert_{L^1_k L^2_{x_3,v}} + \Big\Vert \int_{\mathbb{R}^2} \Vert \hat{f}(k-\ell)\Vert_{L^\infty_T L^2_{x_3,v}} \Vert \hat{f}(\ell)\Vert_{L^2_T L^\infty_{x_3}L^2_\nu} \dd \ell \Big\Vert_{L^1_k} ,
\end{align*}
and for $2< p\leq \infty$,
\begin{align*}
 & \Vert \hat{f}\Vert_{L^p_k L^\infty_T L^2_{x_3,v}} + \Vert (\mathbf{I}-\mathbf{P})\hat{f}\Vert_{L^p_k L^2_{T,x_3,\nu}} + |(1-P_\gamma)\hat{f}|_{L^p_k L^2_{T,\gamma_+}} \notag\\
 & \lesssim  \Vert \hat{f}_0\Vert_{L^p_k L^2_{x_3,v}} + \Big\Vert \int_{\mathbb{R}^2} \Vert \hat{f}(k-\ell)\Vert_{L^\infty_T L^2_{x_3,v}} \Vert \hat{f}(\ell)\Vert_{L^2_T L^\infty_{x_3}L^2_\nu} \dd \ell \Big\Vert_{L^p_k}.
\end{align*}
\end{lemma}

\begin{proof}
To prove the lemma, we first need the $L^2_{x_3,v}$ energy estimate
\begin{align*}
  & \Vert \hat{f}(T)\Vert_{L^p_k L^2_{x_3,v}} + \Vert (\mathbf{I}-\mathbf{P})\hat{f}\Vert_{L^p_k L^2_{T,x_3,\nu}} + |(I-P_\gamma)\hat{f}|_{L^p_k L^2_{T,\gamma_+}} \\
  &\lesssim \Vert \hat{f}_0 \Vert_{L^p_k L^2_{x_3,v}} + \Big \Vert \Big(\int_0^T \left|\int_{0}^\infty \int_{\mathbb{R}^3} \hat{\Gamma}(\hat{f},\hat{f}) \bar{\hat{f}} \dd v \dd x_3\right| \dd t \Big)^{1/2} \Big\Vert_{L^p_{k}} \\
  & \lesssim \Vert \hat{f}_0 \Vert_{L^p_k L^2_{x_3,v}} + o(1)\Vert (\mathbf{I}-\mathbf{P})\hat{f}\Vert_{L^p_k L^2_{T,x_3,\nu}} + \Big\Vert \int_{\mathbb{R}^2} \Vert \hat{f}(k-\ell)\Vert_{L^\infty_T L^2_{x_3,v}} \Vert \hat{f}(\ell)\Vert_{L^2_T L^\infty_{x_3}L^2_\nu} \dd \ell \Big\Vert_{L^p_k}. 
\end{align*}
In the last line, we have applied \eqref{gamma_product}.

We also need the following degenerate macroscopic dissipation estimate
\begin{align}
  &   \Big\Vert \frac{|k|}{\sqrt{1+|k|^2}}(\hat{a},\hat{\mathbf{b}},\hat{c})\Big\Vert_{L^1_k L^2_{T,x_3}} \lesssim \Vert (\mathbf{I}-\mathbf{P})\hat{f}\Vert_{L^1_k L^2_{T,x_3,v}}+ |(I-P_\gamma)\hat{f}|_{L^1_k L^2_{T,\gamma_+}} \notag\\
  &+ \Big\Vert \int_{\mathbb{R}^2} \Vert \hat{f}(k-\ell)\Vert_{L^\infty_T L^2_{x_3,v}} \Vert \hat{f}(\ell)\Vert_{L^2_T L^\infty_{x_3}L^2_\nu} \dd \ell \Big\Vert_{L^1_k}  + \Vert \hat{f}\Vert_{L^1_k L^\infty_T L^2_{x_3,v}} + \Vert \hat{f}_0\Vert_{L^1_k L^2_{x_3,v}}. 
  \label{degenerate_macro}
\end{align}
The proof of this estimate is the same as Lemma 5 in \cite{chenduanzhang2024}. We remark here that this dissipation estimate also holds when $x_3\in (0,\infty)$ is unbounded, since the Poincaré inequality is not required in Lemma 5 of \cite{chenduanzhang2024}. To be specific, we take the $\hat{a}$ estimate as an example. In the weak formulation argument, we choose a test function as
\begin{align}
  & \psi_a =  \sqrt{\mu} (|v|^2-10) (-i\bar{v}\cdot k + v_3 \p_{x_3} )\phi_a. \notag
\end{align}
Here $\phi_a$ satisfies the elliptic equation
\begin{align}
\begin{cases}
  & (|k|^2 - \p_{x_3}^2)\phi_a(k,x_3) = -\bar{\hat{a}}(k,x_3) \frac{|k|^2}{1+|k|^2} , \ x_3\in (0,\infty), \\
  &  \p_{x_3} \phi_a(k, \pm 1) = 0.  
\end{cases}\label{elliptic_a}
\end{align}
Here $\bar{\hat{a}}$ stands for the complex conjugate of $\hat{a}$.

Multiplying \eqref{elliptic_a} by $\bar{\phi}_a$, the complex conjugate of $\phi_a$, we obtain
\begin{align*}
  & |k|^2 \Vert \phi_a \Vert_{L^2_{x_3}}^2 + \Vert \p_{x_3}\phi_a \Vert_{L^2_{x_3}}^2 \lesssim \frac{|k|^2}{1+|k|^2} \Vert \hat{a}\Vert_{L^2_{x_3}}^2 + o(1)\frac{|k|^2}{1+ |k|^2} \Vert \phi_a \Vert_{L^2_{x_3}}^2, \\
  & \Vert |k| \phi_a\Vert_{L^2_{x_3}}^2 + \Vert \p_{x_3}\phi_a\Vert_{L^2_{x_3}}^2 \lesssim \frac{|k|^2}{1+|k|^2} \Vert \hat{a}\Vert_{L^2_{x_3}}^2.
\end{align*}

Multiplying \eqref{elliptic_a} by $|k|^2 \bar{\phi}_a$, we obtain
\begin{align*}
  &  |k|^4 \Vert \phi_a\Vert_{L^2_{x_3}}^2 + |k|^2 \Vert \p_{x_3}\phi_a\Vert_{L^2_{x_3}}^2 \lesssim o(1) |k|^4 \Vert \phi_a \Vert_{L^2_{x_3}}^2 + \frac{|k|^4}{(1+|k|^2)^2}\Vert \hat{a}\Vert_{L^2_{x_3}}^2, \\
  & \Vert |k|^2 \phi_a\Vert_{L^2_{x_3}}^2 + \Vert |k|\p_{x_3}\phi_a\Vert_{L^2_{x_3}}^2 \lesssim \frac{|k|^4}{1+|k|^4} \Vert \hat{a}\Vert_{L^2_{x_3}}^2 \lesssim \frac{|k|^2}{1+|k|^2}\Vert \hat{a}\Vert_{L^2_{x_3}}^2.
\end{align*}

This leads to the estimate that
\begin{align}
  &  \Vert \p_{x_3}^2 \phi_a \Vert_{L^2_{x_3}} \lesssim |k|^2 \Vert \phi_a\Vert_{L^2_{x_3}} + \frac{|k|^2}{1+|k|^2} \Vert \hat{a}\Vert_{L^2_{x_3}} \lesssim \frac{|k|}{\sqrt{1+|k|^2}} \Vert \hat{a}\Vert_{L^2_{x_3}}, \notag \\
  & \Vert ( |k|+ |k|^2) \phi_a\Vert_{L^2_{x_3}} + \Vert (1+|k|) \p_{x_3}\phi_a\Vert_{L^2_{x_3}} + \Vert \p_{x_3}^2 \phi_a\Vert_{L^2_{x_3}} \lesssim \frac{|k|}{\sqrt{1+|k|^2}} \Vert \hat{a}\Vert_{L^2_{x_3}}.\label{H2_est_a}
\end{align}
By trace theorem, using \eqref{H2_est_a} we conclude that
\begin{align}
  &  | |k| \phi_a(k,0)| \lesssim \frac{|k|}{\sqrt{1+|k|^2}} \Vert \hat{a}\Vert_{L^2_{x_3}}, \ | \p_{x_3} \phi_a(k,0)| \lesssim  \frac{|k|}{\sqrt{1+|k|^2}} \Vert \hat{a}\Vert_{L^2_{x_3}}. \notag
\end{align}
All computations only use the H\"older inequality without the Poincaré inequality. With the trace estimate and regularity estimate \eqref{H2_est_a}, the rest of the proof is the same as Lemma 5 in \cite{chenduanzhang2024} by changing the domain $(-1,1)$ to $(0,\infty)$.
\end{proof}

\begin{remark}\label{remark:continous_frequency}
The degenerate macroscopic dissipation estimate does not work when the domain is $\mathbb{T}^2\times \mathbb{R}^+$, because the horizontal Fourier variable becomes discrete. In this setting, the zero Fourier mode $k=0$ cannot be ignored; it corresponds to a one-dimensional half-space nonlinear problem, which remains an open problem.
\end{remark}

Then we include the time weight into the energy estimate.
\begin{lemma}[\textbf{Energy estimate with time decay}]\label{lemma:full_energy_decayexg}
Let $p>2$ and $\sigma = 2(1-1/p)-2\e$ with $\e>0$ small enough, then under the assumption in Lemma \ref{prop:full energyexg}, we have
\begin{align*}
 & \Vert (1+t)^{\sigma/2} \hat{f}\Vert_{L^1_k L^\infty_T L^2_{x_3,v}} + \Vert (1+t)^{\sigma/2} (\mathbf{I}-\mathbf{P}) \hat{f} \Vert_{L^1_k L^2_T L^2_{x_3,\nu}} + |(1+t)^{\sigma/2}(1-P_\gamma)\hat{f}|_{L^1_k L^2_{T,\gamma_+}} \\
 & + \Big\Vert (1+t)^{\sigma/2} \frac{|k|}{\sqrt{1+|k|^2}}(\hat{a},\hat{\mathbf{b}},\hat{c})\Big\Vert_{L^1_k L^2_{T,x_3}} \\
 & \lesssim \Vert \hat{f}_0\Vert_{L^1_k L^2_{x_3,v}} + \Vert \hat{f}_0\Vert_{L^p_k L^2_{x_3,v}} + \Big\Vert \int_{\mathbb{R}^2} \Vert (1+t)^{\sigma/2} \hat{f}(k-\ell)\Vert_{L^\infty_T L^2_{x_3,v}} \Vert \hat{f}(\ell)\Vert_{L^2_T L^\infty_{x_3}L^2_\nu} \dd \ell \Big\Vert_{L^1_k} \\
 & + \Big\Vert \int_{\mathbb{R}^2} \Vert \hat{f}(k-\ell)\Vert_{L^\infty_T L^2_{x_3,v}} \Vert \hat{f}(\ell)\Vert_{L^2_T L^\infty_{x_3}L^2_\nu} \dd \ell \Big\Vert_{L^p_k} . 
\end{align*} 
\end{lemma}
\begin{proof}
The proof of this lemma is the same as Proposition 5 in \cite{chenduanzhang2024} with replacing the spatial domain $(-1,1)$ by $(0,\infty)$ since the Poincaré inequality is not required.
\end{proof}

Last, we give the $L^\infty$ estimate. Recall the velocity weight is defined in \eqref{weight_w}.
\begin{lemma}[\textbf{$L^1_k L^\infty_{T,x_3,v}$ estimate with time decay}]\label{lemma:linftyexg}
Let $\hat{f}$ be the solution to \eqref{f_eqn_r2} with initial data $f_0$ satisfying \eqref{initial_assumptionexg}, then we have the following $L^1_k L^\infty_{T,x_3,v}$ control with time decay:
\begin{align*}
 & \Vert (1+t)^{\sigma/2} w\hat{f}\Vert_{L^1_k L^\infty_{T,x_3,v}} \lesssim \Vert w\hat{f}_0\Vert_{L^1_k L^\infty_{x_3,v}} + \Vert (1+t)^{\sigma/2} \hat{f}\Vert_{L^1_k L^\infty_T L^2_{x_3,v}} \\
 & + \Big\Vert \int_{\mathbb{R}^2}\Vert (1+t)^{\sigma/2} w\hat{f}(k-\ell)\Vert_{L^\infty_{T,x_3,v}} \Vert (1+t)^{\sigma/2} w\hat{f}(\ell) \Vert_{L^\infty_{T,x_3,v}} \dd \ell \Big\Vert_{L^1_k}.
\end{align*}
\end{lemma}

\begin{proof}
With the $L^\infty_T L^2_{x_3,v}$ estimate in Lemma \ref{lemma:full_energy_decayexg}, the proof of this lemma follows from a standard $L^2_{x_3,v}-L^\infty_{x_3,v}$ argument, which is the same as Proposition 6 in \cite{chenduanzhang2024} . In fact, the proof of this lemma would be simpler, since the characteristic only interacts with the boundary once. We omit the detail.
\end{proof}

In the following, we summarize previous results and obtain the a priori estimate, which is devoted to Theorem \ref{thm:l1k_lpkexg}.
\begin{proposition}\label{prop:aprioi_fxmexg}
Let $\hat{f}$ be the solution to \eqref{f_eqn_r2} such that the initial condition $f_0$ satisfy \eqref{initial_assumptionexg}, and 
\begin{align}\label{ad.prop.apexg}
 &\Vert (1+t)^{\sigma/2} \hat{f}\Vert_{L^1_k L^\infty_T L^2_{x_3,v}}+\Vert \hat{f}\Vert_{L^p_k L^\infty_T L^2_{x_3,v}}+\Vert (1+t)^{\sigma/2} w\hat{f}\Vert_{L^1_k L^\infty_{T,x_3,v}}<\infty,
\end{align}
then for some $C>1$,
\begin{align*}
 & \Vert (1+t)^{\sigma/2} \hat{f}\Vert_{L^1_k L^\infty_T L^2_{x_3,v}}+ \Vert (1+t)^{\sigma/2} w\hat{f} \Vert_{L^1_k L^\infty_{T,x_3,v}} + \Vert \hat{f}\Vert_{L^p_k L^\infty_T L^2_{x_3,v}} \\
& \leq C\big[ \Vert \hat{f}_0\Vert_{L^1_k L^2_{x_3,v}} +\Vert w\hat{f}_0\Vert_{L^1_k L^\infty_{x_3,v}} + \Vert \hat{f}_0\Vert_{L^p_k L^2_{x_3,v}} \\
& + \big(\Vert (1+t)^{\sigma/2} \hat{f}\Vert_{L^1_k L^\infty_T L^2_{x_3,v}}+ \Vert (1+t)^{\sigma/2} w\hat{f} \Vert_{L^1_k L^\infty_{T,x_3,v}} + \Vert \hat{f}\Vert_{L^p_k L^\infty_T L^2_{x_3,v}}\big)^2\big] .
\end{align*}
\end{proposition}

\begin{proof}
Combining Lemma \labelcref{prop:full energyexg,lemma:full_energy_decayexg,lemma:linftyexg}, we obtain
 \begin{align}
 & \Vert (1+t)^{\sigma/2} \hat{f}\Vert_{L^1_k L^\infty_T L^2_{x_3,v}}+ \Vert (1+t)^{\sigma/2} w\hat{f} \Vert_{L^1_k L^\infty_{T,x_3,v}} + \Vert \hat{f}\Vert_{L^p_k L^\infty_T L^2_{x_3,v}} \notag\\
 & \leq C\big[\Vert \hat{f}_0\Vert_{L^1_k L^2_{x_3,v}} + \Vert w\hat{f}_0\Vert_{L^1_k L^\infty_{x_3,v}} + \Vert \hat{f}_0\Vert_{L^p_k L^2_{x_3,v}} + \Vert (1+t)^{\sigma/2} w \hat f\Vert_{L^1_kL^\infty_{T,x_3,v}}^2\notag\\
 & + \Vert (1+t)^{\sigma/2} \hat{f}\Vert_{L^1_k L^\infty_T L^2_{x_3,v}} \Vert w \hat{f}\Vert_{L^1_k L^2_T L^\infty_{x_3,v}} +\Vert \hat{f}\Vert_{L^p_k L^\infty_T L^2_{x_3,v}} \Vert w \hat{f}\Vert_{L^1_k L^2_T L^\infty_{x_3,v}}\big]  \notag \\
 & \leq C\big[\Vert \hat{f}_0\Vert_{L^1_k L^2_{x_3,v}} + \Vert w\hat{f}_0\Vert_{L^1_k L^\infty_{x_3,v}} + \Vert \hat{f}_0\Vert_{L^p_k L^2_{x_3,v}} + \Vert (1+t)^{\sigma/2} w \hat f\Vert_{L^1_kL^\infty_{T,x_3,v}}^2\notag\\
 & + \Vert (1+t)^{\sigma/2} \hat{f}\Vert_{L^1_k L^\infty_T L^2_{x_3,v}} \Vert (1+t)^{\sigma/2}w \hat{f}\Vert_{L^1_k L^\infty_{T,x_3,v}} +\Vert \hat{f}\Vert_{L^p_k L^\infty_T L^2_{x_3,v}} \Vert (1+t)^{\sigma/2} w \hat{f}\Vert_{L^1_k L^\infty_{T,x_3,v}}\big] \notag.
\end{align}
In the first inequality, we have applied Young's convolution inequality for all $\ell$ integrations. In the second inequality, we have applied $\sigma>1$ for the $L^2_T$ integral. We complete the proof.
\end{proof}

\begin{remark}\label{remark:L2_Linfty}
In \cite{chenduanzhang2024}, since the vertical variable is bounded as $x_3\in (-1,1)$, one directly has $\Vert (1+t)^{\sigma/2} f \Vert_{L^1_k L^\infty_{T} L^2_{(-1,1),v}} \lesssim \Vert (1+t)^{\sigma/2} w \hat{f}\Vert_{L^1_k L^\infty_{T,(-1,1),v}}$. In our problem, since $x_3$ is unbounded, we need to close the estimate in $L^2_{x_3}\cap L^\infty_{x_3}$ as \eqref{ad.prop.apexg}.

\end{remark}

\begin{proof}[\textbf{Proof of Theorem \ref{thm:l1k_lpkexg}}]
With the a priori estimate in Proposition \ref{prop:aprioi_fxmexg}, we can apply the standard sequential argument to construct a unique solution to \eqref{f_eqn_r2} that satisfies \eqref{f_estimateexg} %and \eqref{f_estimate_2}
. The positivity also follows from a standard sequential argument; we refer to details in \cite{MR4688691}. Note that the a priori assumption \eqref{ad.prop.apexg} can be closed due to the smallness of initial data as in \eqref{initial_assumption}. 
\end{proof}

\section{Global analytic regularity in $x_\parallel$ over $\mathbb{R}^2\times \mathbb{R}^+$}\label{anainx}

In this section, we will establish the propagation of analytic in $x_\parallel$ estimate in the domain of $\mathbb{R}^2\times\mathbb{R}^+$. Recall $L_{\rho,m,|\alpha|}$ defined in \eqref{r2dalp}. In this subsection, we start from the analytic $x_\parallel$ estimate; thus we set $|\alpha| = 0$.

\begin{theorem}\label{thminxm}
Let $\hat{f}$ be the solution to \eqref{f_eqn_r2} such that the initial data satisfies 
\begin{equation}
\sup_{m\geq 0} L_{\rho,m,0}\Big[\Vert \avk^{m}\hat{f}_0\Vert_{L^1_k L^2_{x_3,v}} +\Vert\avk^{m} w\hat{f}_0\Vert_{L^1_k L^\infty_{x_3,v}} + \Vert \avk^{m}\hat{f}_0\Vert_{L^p_k L^2_{x_3,v}}\Big] < \delta_0 \label{iniassxxx}
\end{equation}
for some $\rho>0$, then the following estimate is satisfied
\begin{align}\label{f_estxxx}
 & \sup_{m\geq 0} L_{\rho,m,0} \Big[\Vert (1+t)^{\sigma/2}\avk^m \hat{f}\Vert_{L^1_k L^\infty_T L^2_{x_3,v}} + \Vert \avk^m\hat{f}\Vert_{L^p_k L^\infty_T L^2_{x_3,v}} + \Vert (1+t)^{\sigma/2} \avk^{m} w\hat{f}\Vert_{L^1_k L^\infty_{T,x_3,v}}\Big]\notag
 \\& \leq C\sup_{m\geq 0} L_{\rho,m,0}\Big[\Vert \avk^{m}\hat{f}_0\Vert_{L^1_k L^2_{x_3,v}} +\Vert \avk^{m}w\hat{f}_0\Vert_{L^1_k L^\infty_{x_3,v}} + \Vert \avk^{m}\hat{f}_0\Vert_{L^p_k L^2_{x_3,v}}\Big]. 
 \end{align}

\end{theorem}
\begin{remark}\label{rmk:analytic_x}
Theorem \ref{thminxm} implies that with the analytic in $x_\parallel$ assumption on the initial data, the unique solution constructed in Theorem \ref{thm:l1k_lpkexg} becomes analytic in $x_\parallel$ as well. Moreover, the analytic radius is consistent with the initial data.
\end{remark}

We consider the problem with $x_\parallel$ derivative:
\begin{align}
\begin{cases}
  & \p_t \langle k\rangle^m \hat{f} + i k\cdot v_\parallel \langle k\rangle^m\hat{f} + v_3 \p_{x_3} \langle k\rangle^m\hat{f} + \mathcal{L}\langle k\rangle^m\hat{f} = \langle k\rangle^m\hat{\Gamma}(\hat{f},\hat{f}), \\
 & \langle k\rangle^m \hat{f}|_{\gamma_-} = P_\gamma \langle k\rangle^m \hat{f}, \\
 & \langle k\rangle^m\hat{f}(0,k,x_3,v) = \langle k\rangle^m f_0(k,x_3,v). 
\end{cases}\label{kmf_eqn}
\end{align}

Since the structure of \eqref{kmf_eqn} is almost the same as \eqref{f_eqn_r2}, the well-posedness of the problem \eqref{kmf_eqn} can be justified using the same argument as in Section \ref{sec:withoutder}; thus we only focus on the following a priori estimate:

\begin{proposition}\label{prop:aprioi_fxm}
Let $\langle k\rangle^m\hat{f}$ be the solution to \eqref{kmf_eqn}, and suppose
\begin{align*}
 \sup_{m\geq 0} L_{\rho,m,0}\Big[\Vert (1+t)^{\sigma/2}\avk^{m}\hat{f}\Vert_{L^1_k L^\infty_T L^2_{x_3,v}}+ \Vert (1+t)^{\sigma/2}\avk^{m} w\hat{f} \Vert_{L^1_k L^\infty_{T,x_3,v}} + \Vert \avk^{m}\hat{f}\Vert_{L^p_k L^\infty_T L^2_{x_3,v}} \Big]<\infty,
 \end{align*}
then for some $C>1$, it holds
\begin{align*}
 & \sup_{m\geq 0} L_{\rho,m,0}\Big[\Vert (1+t)^{\sigma/2}\avk^{m}\hat{f}\Vert_{L^1_k L^\infty_T L^2_{x_3,v}}+ \Vert (1+t)^{\sigma/2}\avk^{m} w\hat{f} \Vert_{L^1_k L^\infty_{T,x_3,v}} + \Vert \avk^{m}\hat{f}\Vert_{L^p_k L^\infty_T L^2_{x_3,v}} \Big]\\
& \leq C \sup_{m\geq 0} L_{\rho,m,0}\Big[\Vert \avk^{m}\hat{f}_0\Vert_{L^1_k L^2_{x_3,v}} + \Vert \avk^{m}w\hat{f}_0\Vert_{L^1_k L^\infty_{x_3,v}} + \Vert \avk^{m}\hat{f}_0\Vert_{L^p_k L^2_{x_3,v}}\Big]\\
& + C\bigg[\sup_{m\geq 0} L_{\rho,m,0}\Big[\Vert (1+t)^{\sigma/2}\avk^{m}\hat{f}\Vert_{L^1_k L^\infty_T L^2_{x_3,v}}+ \Vert (1+t)^{\sigma/2}\avk^{m} w\hat{f} \Vert_{L^1_k L^\infty_{T,x_3,v}} + \Vert \avk^{m}\hat{f}\Vert_{L^p_k L^\infty_T L^2_{x_3,v}} \Big]\bigg]^2.
\end{align*}
\end{proposition}
\begin{proof}
Since \eqref{kmf_eqn} has the same structure as the equation of $\hat{f}$ in \eqref{f_eqn_r2}, we can follow the same argument as Lemma \ref{prop:full energyexg} to \ref{lemma:linftyexg} and obtain
\begin{align}\label{apriestmxxr2}
 & L_{\rho,m,0}\big[\Vert (1+t)^{\sigma/2} \avk^{m}\hat{f}\Vert_{L^1_k L^\infty_T L^2_{x_3,v}}+ \Vert (1+t)^{\sigma/2}\avk^{m} w\hat{f} \Vert_{L^1_k L^\infty_{T,x_3,v}} + \Vert \avk^{m}\hat{f}\Vert_{L^p_k L^\infty_T L^2_{x_3,v}} \big]\notag\\
& \leq C L_{\rho,m,0}\Big[\Vert \avk^{m}\hat{f}_0\Vert_{L^1_k L^2_{x_3,v}} + \Vert \avk^{m}w\hat{f}_0\Vert_{L^1_k L^\infty_{x_3,v}} + \Vert \avk^{m}\hat{f}_0\Vert_{L^p_k L^2_{x_3,v}} \notag\\
 & + \Big\Vert \avk^{m}\int_{\mathbb{R}^2} \Vert \hat{f}(k-\ell)\Vert_{L^\infty_T L^2_{x_3,v}} \Vert \hat{f}(\ell)\Vert_{L^2_T L^\infty_{x_3}L^2_\nu} \dd \ell \Big\Vert_{L^p_k} \notag\\
 & + \Big\Vert \avk^{m}\int_{\mathbb{R}^2} \Vert (1+t)^{\sigma/2} \hat{f}(k-\ell)\Vert_{L^\infty_T L^2_{x_3,v}} \Vert \hat{f}(\ell)\Vert_{L^2_T L^\infty_{x_3}L^2_\nu} \dd \ell \Big\Vert_{L^1_k} \notag\\
 & + \Big\Vert \avk^{m}\int_{\mathbb{R}^2}\Vert (1+t)^{\sigma/2} w\hat{f}(k-\ell)\Vert_{L^\infty_{T,x_3,v}} \Vert (1+t)^{\sigma/2} w\hat{f}(\ell) \Vert_{L^\infty_{T,x_3,v}} \dd \ell \Big\Vert_{L^1_k} \Big].
\end{align}

We need to estimate the last three nonlinear terms in \eqref{apriestmxxr2}. Then we control $\langle k\rangle^m$ as 
\begin{equation}\label{klm}
 \avk^m\leq (\avkl+\avl)^m = \sum_{j=0}^{m}\binom{m}{j}\avkl^{j} \avl^{m-j}.
\end{equation}

For the first nonlinear term, we apply \eqref{klm}, the Young's convolution inequality and $\sigma>1,w^{-1}\in L^2_v$ to obtain
\begin{align*}
&\Big\Vert \avk^{m}\int_{\mathbb{R}^2} \Vert \hat{f}(k-\ell)\Vert_{L^\infty_T L^2_{x_3,v}} \Vert \hat{f}(\ell)\Vert_{L^2_T L^\infty_{x_3}L^2_\nu} \dd \ell \Big\Vert_{L^p_k}\\
 & \lesssim \Big\Vert \sum_{j=0}^m \binom{m}{j} \int_{\mathbb{R}^2} \Vert \langle k-\ell\rangle^{j}\hat{f}(k-\ell)\Vert_{L^\infty_T L^2_{x_3,v}} \Vert \langle \ell\rangle^{m-j}\hat{f}(\ell)\Vert_{L^2_T L^\infty_{x_3}L^2_\nu} \dd \ell \Big\Vert_{L^p_k}\\
 &\lesssim \sum_{j=0}^{m}\binom{m}{j}\Vert \avk^{j}\hat{f}\Vert_{L^p_k L^\infty_T L^2_{x_3,v}}\Vert (1+t)^{\sigma/2}\avk^{m-j}w\hat{f}\Vert_{L^1_k L^\infty_{T,x_3,v} }.
\end{align*} 

Including the analytic weight, we further have
\begin{align}\label{estrr2tirpk}
& L_{\rho,m,0} \sum_{j=0}^{m}\binom{m}{j}\Vert \avk^{j}\hat{f}\Vert_{L^p_k L^\infty_T L^2_{x_3,v}}\Vert (1+t)^{\sigma/2} \avk^{m-j}w\hat{f}\Vert_{L^1_k L^\infty_{T, x_3,v} }\notag \\
& \lesssim \sum_{j=0}^{m}\binom{m}{j}\frac{ L_{\rho,m,0}}{L_{\rho,j,0}L_{\rho,m-j,0}}\Big(L_{\rho,j,0}\Vert \avk^{j}\hat{f} \Vert_{L^p_k L^\infty_T L^2_{x_3,v}}\Big)\Big( L_{\rho,m-j,0}\Vert (1+t)^{\sigma/2} \avk^{m-j}w\hat{f}\Vert_{L^1_k L^\infty_{T, x_3,v} }\Big) \notag 
\\
& \lesssim \bigg(\sup_{m\geq 0} L_{\rho,m,0} \Vert \avk^m\hat{f}\Vert_{L^p_k L^\infty_T L^2_{x_3,v}}\bigg)\bigg(\sup_{m\geq 0} L_{\rho,m,0}\Vert(1+t)^{\sigma/2} w\avk^m\hat{f}\Vert_{L^1_k L^\infty_{T, x_3,v}}\bigg).
\end{align}
In the last inequality, we have applied \eqref{sumlrhom0}.

Applying a similar computation to the second nonlinear term in the RHS of \eqref{apriestmxxr2}, we obtain 
 \begin{align}\label{gamma_estxsec}
 & L_{\rho,m,0} \Big\Vert \avk^{m}\int_{\mathbb{R}^2} \Vert (1+t)^{\sigma/2} \hat{f}(k-\ell)\Vert_{L^\infty_T L^2_{x_3,v}} \Vert \hat{f}(\ell)\Vert_{L^2_T L^\infty_{x_3}L^2_\nu} \dd \ell \Big\Vert_{L^1_k} \notag\\
 &\lesssim \bigg(\sup_{m\geq 0} L_{\rho,m,0} \Vert(1+t)^{\sigma/2} \avk^m\hat{f}\Vert_{L^1_k L^\infty_T L^2_{x_3,v}}\bigg)\bigg(\sup_{m\geq 0} L_{\rho,m,0}\Vert(1+t)^{\sigma/2} w\avk^m\hat{f}\Vert_{L^1_k L^\infty_{T, x_3,v}}\bigg) .
\end{align}

Finally, we apply \eqref{klm} and Young's convolutional inequality to compute the third nonlinear term in \eqref{apriestmxxr2}:
\begin{align}
& L_{\rho, m,0} \Big\Vert \avk^{m}\int_{\mathbb{R}^2}\Vert (1+t)^{\sigma/2} w\hat{f}(k-\ell)\Vert_{L^\infty_{T,x_3,v}} \Vert (1+t)^{\sigma/2} w\hat{f}(\ell) \Vert_{L^\infty_{T,x_3,v}} \dd \ell \Big\Vert_{L^1_k}\notag\\ 
& \lesssim L_{\rho, m,0} \sum_{j=0}^{m}\binom{m}{j}\Vert (1+t)^{\sigma/2}\avk^{j} w \hat{f}\Vert_{ L^1_kL^\infty_{T,x_3,v}}\Vert (1+t)^{\sigma/2}\avk^{m-j} w \hat{f}\Vert_{ L^1_k L^\infty_{T,x_3,v}}\notag\\
&\lesssim \sum_{j=0}^{m}\binom{m}{j}\frac{L_{\rho,m,0}}{L_{\rho,j,0}L_{\rho,m-j,0}}\Big(L_{\rho,j,0}\Vert (1+t)^{\sigma/2}\avk^{j} w \hat{f}\Vert_{ L^1_kL^\infty_{T,x_3,v}}\Big) \notag\\
&\times\Big(L_{\rho,m-j,0}\Vert (1+t)^{\sigma/2}\avk^{m-j} w \hat{f}\Vert_{ L^1_k L^\infty_{T,x_3,v}}\Big)\notag\\
& \lesssim \bigg( \sup_{m\geq 0} L_{\rho, m,0}\Vert (1+t)^{\sigma/2}\avk^{m} w \hat{f}\Vert_{ L^1_kL^\infty_{T,x_3,v}}\bigg)^2 \label{gamma_estxthi}. 
\end{align}
In the last line, we have applied \eqref{sumlrhom0}.

Finally, substituting \eqref{estrr2tirpk}, \eqref{gamma_estxsec} and \eqref{gamma_estxthi} into \eqref{apriestmxxr2}, we conclude the proof of Proposition \ref{prop:aprioi_fxm}.
\end{proof}

\begin{proof}[\textbf{Proof of Theorem \ref{thminxm}}]
With the a priori estimate in Proposition \ref{prop:aprioi_fxm}, under the small data assumption in \eqref{iniassxxx}, it is standard to prove \eqref{f_estxxx}. We omit the details.
\end{proof}

\section{Global Gevrey regularity in $v_\parallel$ over $\mathbb{R}^2\times \mathbb{R}^+$}\label{anainmixy}

Due to the transport operator $v\cdot \nabla_x$, the $v_\parallel$ derivative estimate depends on the $x_\parallel$ derivative estimate. By employing the analytic in $x_\parallel$ estimate in Section \ref{anainx}, we investigate the Gevrey in $v_\parallel$ estimate in this section.

Our main goal is the following result.

\begin{theorem}\label{thm:apriestxvr2}
 Under the assumptions in Theorem \ref{gevr2}, the following estimate holds:
\begin{align*}
 &\sup_{m, \a\geq 0}L_{\rho, m, |\a|}
 \Vert (1+t)^{\sigma/2}w \avk^m\p_{v_\parallel}^{\a} \hat{ f}(t)\Vert_{ L^1_kL^\infty_{T, x_3,v}}\\&\leq 
C\bigg[\sup_{m, \a\geq 0}L_{\rho, m, |\a|} \Vert w \langle k \rangle ^{m}\p_{v_\parallel}^{\a}\hat{ f}_0 \Vert_{L^1_k L^\infty_{x_3,v}}+\sup_{m \geq 0}L_{\rho, m,0 }\Big(\Vert \avk^{m}\hat{f}_0\Vert_{L^1_k L^2_{x_3,v}} + \Vert \avk^{m}\hat{f}_0\Vert_{L^p_k L^2_{x_3,v}}\Big) \bigg]
\end{align*}
for any $T>0$. 
\end{theorem}

To establish the propagation of Gevrey regularity with mixed derivatives, we mainly focus on the following a priori estimate.
\begin{proposition}\label{prop:inftyxvr2}
 Under the same assumptions of Theorem \ref{thm:apriestxvr2}, and we further assume 
 \begin{align*}
 &\sup_{m, \a\geq 0}L_{\rho, m, |\a|}
 \Vert (1+t)^{\sigma/2}w \avk^m\p_{v_\parallel}^{\a} \hat{ f}(t)\Vert_{ L^1_kL^\infty_{T, x_3,v}}< \infty,
 \end{align*}
 the following estimate is satisfied:
\begin{align*}
 &\sup_{m, \a\geq 0}L_{\rho, m, |\a|}
 \Vert (1+t)^{\sigma/2}w \avk^m\p_{v_\parallel}^{\a} \hat{ f}(t)\Vert_{ L^1_kL^\infty_{T, x_3,v}}\\&\leq C\bigg[\sup_{m, \a\geq 0}L_{\rho, m, |\a|} \Vert w \langle k \rangle ^{m}\p_{v_\parallel}^{\a}\hat{ f}_0 \Vert_{L^1_k L^\infty_{x_3,v}} 
+\sup_{m\geq 0 }L_{\rho, m,0 }\Big(\Vert \avk^{m}\hat{f}_0\Vert_{L^1_k L^2_{x_3,v}} + \Vert \avk^{m}\hat{f}_0\Vert_{L^p_k L^2_{x_3,v}}\Big)\bigg]\\&+ C\bigg(\sup_{m,\a\geq 0}L_{\rho, m, |\a|} \Vert (1+t)^{\sigma/2} w \langle k\rangle^{m} \p_{v_\parallel}^{\a} \hat f(t)\Vert_{ L^1_kL^\infty_{T, x_3,v}}\bigg)^2
\end{align*}
for any $T>0$. 
\end{proposition}

To prove the proposition, we mainly apply the method of characteristics. First we define the stochastic cycle. We note that in the half-space problem, the characteristic only interacts once with the boundary, thus we only define the backward exit time and backward exit position in the physical space $(x_3,v)\in \mathbb{R}^+ \times \{v\in \mathbb{R}^3|v_3>0\}$:
	\begin{equation*}%\label{xb_tb}
		\begin{split}
			&\tb(x_3,v) : = \sup\{s\geq 0, x_3-sv_3 \in \mathbb{R}^+\}, \notag\\
			&t^1: = t-\tb(x_3,v), \ \ \xb(x_3,v) : = x_3 - \tb(x_3,v)v_3 =0.
		\end{split}
	\end{equation*}
We rewrite \eqref{f_eqn_r2} into the following formulation:
	\begin{align*}
		& \p_t \hat{f} + iv_\parallel\cdot k \hat{f} + v_3 \p_{x_3} \hat{f} + \nu(v) \hat{f} = K\hat{f} + \hat{\Gamma}(\hat{f},\hat{f}).
	\end{align*}
We apply the method of characteristics to have
	\begin{align}
		& \hat{f}(t,k_1,k_2,x_3,v) \notag\\
		&= \mathbf{1}_{\tb>t} e^{-\nu(v) t - i(v_\parallel\cdot k)t} \hat{f}_0(k_1,k_2,x_3 - t v_3, v) \label{chara:initial} \\
		& + \int^t_{\max\{0, t-\tb\}} e^{-\nu(v)(t-s) - i(v_\parallel\cdot k)(t-s)} \int_{\mathbb{R}^3} \mathbf{k}(v,u) \hat{f}(s,k_1,k_2,x_3-(t-s)v_3,u) \dd u\dd s \label{chara:K}\\
		& + \int^t_{\max\{0,t-\tb\}} e^{-\nu(v)(t-s)-i(v_\parallel\cdot k)(t-s)} \hat{\Gamma}(\hat{f},\hat{f})(s,k_1,k_2,x_3-(t-s)v_3,v) \dd s \label{chara:Gamma}
 \\
		& + \mathbf{1}_{\tb \leq t} e^{-\nu(v)\tb - i(v_\parallel\cdot k)\tb}\hat{f}(t^1,k_1,k_2,0,v). \label{chara:bdr}
	\end{align}
Here the boundary term \eqref{chara:bdr} is represented using the diffuse boundary condition: 
	\begin{align*}
		& \eqref{chara:bdr}= \mathbf{1}_{\tb \leq t} c_{\mu}e^{-\nu(v) \tb- i(v_\parallel\cdot k)\tb} \sqrt{\mu(v)}\int_{u_3<0} \hat{f}(t^1, k,0,u)\sqrt{\mu(u)}|u_3| \dd u. 
	\end{align*}

\begin{proof}[\textbf{Proof of Proposition \ref{prop:inftyxvr2}}]

For the proof, we apply $w\avk^m\p_{v_\parallel}^\a$ to \eqref{chara:initial}, \eqref{chara:K}, \eqref{chara:Gamma}, \eqref{chara:bdr} and then multiply $L_{\rho, m, |\a|}$. We proceed through four steps to derive the upper bound of the above terms.

\textit{Step 1: estimate of \eqref{chara:initial}.}

Applying Leibniz's formula yields 
\begin{align}
&L_{\rho, m, |\a|}| w(v)\avk^m\p_{v_\parallel}^\a\eqref{chara:initial}|
\notag\\
& \leq L_{\rho, m, |\a|}\sum_{\b\leq \a}\sum_{\g\leq \b}\binom{\a}{\b}\binom{\b}{\g}\Big|\Big(\p_{v_\parallel}^{\a-\b} e^{- \nu(v) t} \Big) \notag\\
&\times\Big(\p_{v_\parallel}^{\b-\g} e^{- i(v_\parallel\cdot k)t} \Big)w(v)\avk^m\p_{v_\parallel}^{\g}\hat{ f}_0(k_1, k_2,x_3-t v_3, v) \Big| \notag \\
&\leq e^{-\nu_0t/4}\sum_{\b\leq \a}\sum_{\g\leq \b }\binom{\a}{\b}\binom{\b}{\g}L^{|\a-\b|}|\a-\b|!|\b-\g|!\bigg(\frac{8}{\nu_0}\bigg)^{|\b-\g|}\frac{L_{\rho, m, |\a|}}{ L_{\rho, m+|\b-\g|,|\g|} }\notag\\ & \times
 L_{\rho, m+|\b-\g|,|\g|} \Vert w\langle k \rangle ^{m+|\b-\g|}\p_{v_\parallel}^{\g}\hat{ f}_0 \Vert_{ L^\infty_{ x_3,v}}\notag\\
&\leq e^{-\nu_0t/4}\bigg(\sup_{m, \a\geq 0}L_{\rho, m, |\a|} \Vert w \langle k \rangle ^{m}\p_{v_\parallel}^{\a}\hat{ f}_0 \Vert_{ L^\infty_{x_3,v}}\bigg)\notag\\ & \times\sum_{\b\leq \a}\sum_{\g\leq \b }\binom{\a}{\b}\binom{\b}{\g}L^{|\a-\b|}|\a-\b|!|\b-\g|!\bigg(\frac{8}{\nu_0}\bigg)^{|\b-\g|}\frac{L_{\rho, m, |\a|}}{ L_{\rho, m+|\b-\g|,|\g|} }. \label{middr2fir}
\end{align}
In the fourth line, we have applied \eqref{tenut}, \eqref{derienut} and the following computation:
\begin{align}\label{tole}
 & |\p_{v_\parallel}^{\alpha-\beta} e^{-\nu(v)t} \p_{v_\parallel}^{\beta-\gamma}e^{-i(v_\parallel\cdot k)t}| \leq L^{|\alpha-\beta|} |\alpha-\beta|! e^{-\nu(v)t/4}e^{-\nu(v)t/4}|t|^{\beta-\gamma}\langle k\rangle^{|\beta-\gamma|} \notag\\
 & \leq L^{|\alpha-\beta|} |\alpha-\beta|! e^{-\nu(v)t/4} \Big(\frac{8}{\nu_0} \Big)^{|\beta-\gamma|} |\beta-\gamma|! \langle k\rangle^{|\beta-\gamma|}.
\end{align}

Using \eqref{absab} with $\alpha>\beta$, we compute
\begin{align}\label{lrhredu}
&\binom{\a}{\b}\binom{\b}{\g}L^{|\a-\b|}|\a-\b|!|\b-\g|!\bigg(\frac{8}{\nu_0}\bigg)^{|\b-\g|}\frac{L_{\rho, m, |\a|}}{ L_{\rho, m+|\b-\g|,|\g|} }
%\notag\\&\leq\binom{|\a|}{|\b|}\binom{|\b|}{|\g|}(|\a-\b|+1)!|\b-\g|!\bigg(\frac{8}{\nu_0}\bigg)^{|\a-\b|+|\b-\g|}\frac{L_{\rho, m, |\a|}}{ L_{\rho, m+|\b-\g|,|\g|} }
\notag\\
&\leq \frac{|\a|!}{|\b|!(|\a|-|\b|)!}\frac{|\b|!}{|\g|!(|\b|-|\g|)!}L^{|\a-\b|}|\a-\b|!|\b-\g|!\bigg(\frac{8}{\nu_0}\bigg)^{|\b-\g|} \notag\\&\quad\frac{{(16/\nu_0)}^{m}\rho^{m+|\a|+1}(m+2)^2(|\a|+2)^2}{(m+|\a|)!|\a|!}\frac{(m+|\b|)!|\g|!}{(16/\nu_0)^{m+|\b-\g|}\rho^{m+|\b|+1}(m+|\b-\g|+2)^2(|\g|+2)^2}
\notag\\
&\leq (L\rho)^{|\a-\b|}\bigg(\frac{1}{2}\bigg)^{|\b-\g|} \frac{(|\a|+2)^2}{(|\g|+2)^2}\frac{(m+|\b|)!}{(m+|\a|)!}\frac{(m+2)^2}{(m+|\b-\g|+2)^2} \notag\\
& \leq \frac{(|\alpha|+2)^2}{(|\gamma|+2)^2}\frac{(L\rho)^{|\a-\b|}}{|\alpha-\beta|!}\bigg(\frac{1}{2}\bigg)^{|\b-\g|} .
\end{align}

Taking summation in $\gamma,\beta$, we apply \eqref{abdiv2} and \eqref{inmati} and obtain 
\begin{align}\label{sumconfir}
&\sum_{\beta\leq \alpha}\sum_{\gamma\leq \beta} \eqref{lrhredu}=\sum_{\b\leq \a}\frac{(|\a|+2)^2}{(|\b|+2)^2}\frac{(L\rho)^{|\a-\b|}}{|\a-\b|!}\sum_{\g\leq \b }\frac{(|\b|+2)^2}{(|\g|+2)^2}\bigg(\frac{1}{2}\bigg)^{|\b-\g|} \leq C.
\end{align}

Substituting the above inequality into \eqref{middr2fir}, we conclude
\begin{align}\label{mixdivr2infi}
L_{\rho, m, |\a|} |w(v)\avk^m\p_{v_\parallel}^\a\eqref{chara:initial}|
 \leq Ce^{-\nu_0t/4}\bigg(\sup_{m, \a\geq 0}L_{\rho, m, |\a|} \Vert w \langle k \rangle ^{m}\p_{v_\parallel}^{\a}\hat{ f}_0 \Vert_{ L^\infty_{x_3,v}}\bigg).
\end{align}

\textit{Step 2: estimate of \eqref{chara:K}.}

We use Leibniz's formula to write
\begin{align}\label{mixdivr2sec}
&L_{\rho, m, |\a|}w(v) | \avk^m\p_{v_\parallel}^{\a} \eqref{chara:K} |
\notag\\
& \leq L_{\rho, m, |\a|}\sum_{\b\leq \a}\binom{\a}{\b}\int_{\max\{0, t-\tb\}}^{t} \Big|\Big(\p_{v_\parallel}^{\a-\b}e^{-\nu(v) (t-s)}\Big) \Big(\p_{v_\parallel}^{\b}e^{ - i(v_\parallel\cdot k)(t-s)}\Big) \Big| \notag \\
&\times\int_{\mathbb{R}^3}\mathbf{k}(v,u) w(v)\langle k \rangle^{m} |\hat f(s,x_3-(t-s)v_3,u)| \dd u \dd s \notag\\& +L_{\rho, m, |\a|}\sum_{\b\leq \a}\sum_{\iota\leq \g\leq \b}\binom{\a}{\b}\binom{\b}{\g}\int_{\max\{0, t-\tb\}}^{t} \Big(\p_{v_\parallel}^{\a-\b}e^{-\nu(v) (t-s)}\Big) \Big|\Big(\p_{v_\parallel}^{\b-\g}e^{ - i(v_\parallel\cdot k)(t-s)}\Big) \Big|\notag\\
& \times \Big|\int_{\mathbb{R}^3}[\p_{v_\parallel}^{\g}\mathbf{k}(v,u)]w(v)\langle k \rangle^{m}\hat f(s,x_3-(t-s)v_3,u) \dd u \dd s \Big|\notag\\
 & =:\mathcal{A}_1+\mathcal{A}_2.
 \end{align}

For the term $\mathcal{A}_1$ in \eqref{mixdivr2sec}, we use \eqref{tole} with replacing $t$ and $\b-\g$ by $t-s$ and $\b$ respectively to derive
 \begin{align*}
| \mathcal{A}_1 | 
 & \leq L_{\rho, m, |\a|}\sum_{\b\leq \a}\binom{\a}{\b}L^{|\a-\b|}|\a-\b|!|\b|!\bigg(\frac{8}{\nu_0}\bigg)^{|\b|} \int_{\max\{0, t-\tb\}}^{t} \int_{\mathbb{R}^3}e^{-\nu(v) (t-s)/4}\mathbf{k}(v,u)\\& \times w(v)\langle k \rangle^{m+|\b|} |\hat f(s,x_3-(t-s)v_3,u)| \dd u \dd s\notag\\
 & \leq L_{\rho, m, |\a|}\sum_{\b\leq \a}\binom{\a}{\b}L^{|\a-\b|}|\a-\b|!|\b|!\bigg(\frac{8}{\nu_0}\bigg)^{|\b|}\bigg(\sup_{0\leq s\leq t}\Vert (1+s)^{\sigma/2}w\langle k \rangle^{m+|\b|} \hat f(s)\Vert_{ L^\infty_{ x_3,v}}\bigg) \\& \times \int_{\max\{0, t-\tb\}}^{t} \int_{\mathbb{R}^3} e^{-\nu(v) (t-s)/4}(1+s)^{-\sigma/2}\frac{w(v)}{w(u)}\mathbf{k}(v,u) \dd u \dd s. 
 \end{align*}
 Then we use \eqref{k_theta} and 
\begin{align}
 & e^{-\nu(v) (t-s)/8} (1+s)^{-\sigma/2} \leq C(1+t)^{-\sigma/2},\label{eettss}\\
 & \int_{\max\{0, t-\tb\}}^{t} e^{-\nu(v) (t-s)/8}\dd s \leq C \nu^{-1}(v)\label{eettssint}
\end{align}
 to obtain 
 \begin{align*}%\label{esia1com}
| \mathcal{A}_1 | & \leq C(1+t)^{-\sigma/2}\sum_{\b\leq \a} \frac{L_{\rho, m, |\a|}}{L_{\rho, m+|\b|,0}}\binom{\a}{\b}L^{|\a-\b|}|\a-\b|!|\b|!\bigg(\frac{8}{\nu_0}\bigg)^{|\b|}\notag\\
& \times L_{\rho, m+|\b|,0} \bigg(\sup_{0\leq s\leq t}\Vert (1+s)^{\sigma/2}w\langle k \rangle^{m+|\b|} \hat f(s)\Vert_{ L^\infty_{ x_3,v}}\bigg) 
\notag\\
 & \leq C(1+t)^{-\sigma/2}\bigg(\sup_{m\geq 0}L_{\rho, m,0}\sup_{0\leq s\leq t}\Vert(1+s)^{\sigma/2} w\langle k \rangle^{m} \hat f(s)\Vert_{ L^\infty_{x_3,v}}\bigg)\notag \\& \times\sum_{\b\leq \a}\binom{\a}{\b}L^{|\a-\b|}|\a-\b|!|\b|!\bigg(\frac{8}{\nu_0}\bigg)^{|\b|}\frac{L_{\rho, m, |\a|}}{L_{\rho, m+|\b|,0}}.
\end{align*}

We use \eqref{absab} to compute
\begin{align*}
&\sum_{\beta\leq \alpha}\binom{\a}{\b}L^{|\a-\b|}|\a-\b|!|\b|!\bigg(\frac{8}{\nu_0}\bigg)^{|\b|}\frac{L_{\rho, m, |\a|}}{L_{\rho, m+|\b|,0}}
\notag\\
&\leq \sum_{\beta\leq \alpha}\frac{|\a|!}{|\b|!(|\a|-|\b|)!}L^{|\a-\b|}|\a-\b|!|\b|!\bigg(\frac{8}{\nu_0}\bigg)^{|\b|} \frac{{(16/\nu_0)}^{m}\rho^{m+|\a|+1}(m+2)^2(|\a|+2)^2}{(m+|\a|)!|\a|!}\notag\\&\times\frac{(m+|\b|)!}{(16/\nu_0)^{m+|\b|}\rho^{m+|\b|+1}(m+|\b|+2)^2}
\notag\\
&=\sum_{\beta\leq \alpha}( L\rho)^{|\a-\b|}(|\a|+2)^2\bigg(\frac{1}{2}\bigg)^{|\b|} \frac{(m+|\b|)!}{(m+|\a|)!}\frac{(m+2)^2}{(m+|\b|+2)^2} \notag \\
& \leq \sum_{\beta\leq \alpha}\frac{( L\rho)^{|\a-\b|}}{|\a-\b|!}\frac{(|\a|+2)^2}{(|\b|+2)^2}(|\b|+2)^2\bigg(\frac{1}{2}\bigg)^{|\b|} \leq C\sum_{\beta\leq \alpha}\frac{( L\rho)^{|\a-\b|}}{|\a-\b|!}\frac{(|\a|+2)^2}{(|\b|+2)^2}
\leq C.
\end{align*}
In the last inequality, we have applied \eqref{inmati}.

Thus we conclude
 \begin{align}\label{mixdivr2secq2a1}
| \mathcal{A}_1 | \leq C(1+t)^{-\sigma/2}\bigg(\sup_{m\geq 0}L_{\rho, m,0}\sup_{0\leq s\leq t}\Vert(1+s)^{\sigma/2} w\langle k \rangle^{m} \hat f(s)\Vert_{ L^\infty_{ x_3,v}}\bigg).
\end{align}
For the term $\mathcal{A}_2$ in \eqref{mixdivr2sec},  Leibniz's formula yields 
 \begin{align*}
 &|\mathcal{A}_2 |\leq L_{\rho, m, |\a|} \sum_{\b\leq \a}\sum_{\iota\leq \g\leq \b}\binom{\a}{\b}\binom{\b}{\g}\int_{\max\{0, t-\tb\}}^{t} \Big| \Big(\p_{v_\parallel}^{\a-\b}e^{-\nu(v) (t-s)}\Big) \Big(\p_{v_\parallel}^{\b-\g}e^{ - i(v_\parallel\cdot k)(t-s)}\Big) \Big| \\
	&\times \Big|\int_{\mathbb{R}^3}\big[\p_{v_\parallel}^{\iota}(\p_{v_\parallel}+\p_{u_\parallel}-\p_{u_\parallel})^{\g-\iota}\mathbf{k}(v,u)\big]w(v) \langle k \rangle^{m}\hat f(s,x_3-(t-s)v_3,u) \dd u \dd s \Big| \\
 & \leq L_{\rho, m, |\a|} \sum_{\b\leq \a}\sum_{\iota\leq \g\leq \b}\sum_{\tau\leq \g-\iota}\binom{\a}{\b}\binom{\b}{\g}\binom{\g-\iota}{\tau} \Big|\int_{\max\{0, t-\tb\}}^{t}\Big(\p_{v_\parallel}^{\a-\b}e^{-\nu(v) (t-s)}\Big) \Big(\p_{v_\parallel}^{\b-\g}e^{ - i(v_\parallel\cdot k)(t-s)}\Big) \\
	&\times \int_{\mathbb{R}^3}\big[\p_{v_\parallel}^{\iota}(-\p_{u_\parallel})^{\g-\iota-\tau}(\p_{v_\parallel}+\p_{u_\parallel})^{\tau} \mathbf{k}(v,u)\big]w(v)\langle k \rangle^{m}\hat f(s,x_3-(t-s)v_3,u) \dd u \dd s \Big|\\
 & \leq L_{\rho, m, |\a|} \sum_{\b\leq \a}\sum_{\iota\leq \g\leq \b}\sum_{\tau\leq \g-\iota}\binom{\a}{\b}\binom{\b}{\g}\binom{\g-\iota}{\tau}\Big| \int_{\max\{0, t-\tb\}}^{t} \Big(\p_{v_\parallel}^{\a-\b}e^{-\nu(v) (t-s)}\Big) \Big(\p_{v_\parallel}^{\b-\g}e^{ - i(v_\parallel\cdot k)(t-s)}\Big) \\
	&\times \int_{\mathbb{R}^3}\big[\p_{v_\parallel}^{\iota}(\p_{v_\parallel}+\p_{u_\parallel})^{\tau} \mathbf{k}(v,u) \big] w(v) \langle k \rangle^{m}\p_{u_\parallel}^{\g-\iota-\tau}\hat f(s,x_3-(t-s)v_3,u) \dd u \dd s\Big|.
 \end{align*}
In the last line, we have applied integration by parts in $\dd u_\parallel$.

Then we use \eqref{tole} with replacing $t$ by $t-s$ to obtain
\begin{align*}
&|\mathcal{A}_2| \leq L_{\rho, m, |\a|}\sum_{\b\leq \a}\sum_{\iota\leq \g\leq \b}\sum_{\tau\leq \g-\iota}\binom{\a}{\b}\binom{\b}{\g}\binom{\g-\iota}{\tau}L^{|\a-\b|}|\a-\b|! |\b-\g|!\bigg(\frac{8}{\nu_0}\bigg)^{|\b-\g|} \\ 
&\times \int^t_{\max\{0,t-\tb\}}\int_{\mathbb{R}^3}\big|\p_{v_\parallel}^{\iota}(\p_{v_\parallel}+\p_{u_\parallel})^{\tau} \mathbf{k}(v,u) \big|e^{-\nu(v) (t-s)/4}w(v) \\ 
&\times|\langle k \rangle^{m+|\b-\g|}\p_{u_\parallel}^{\g-\tau-\iota}\hat f(s,x_3-(t-s)v_3,u)| \dd u \dd s\\ 
& \leq L_{\rho, m, |\a|} \sum_{\b\leq \a}\sum_{\iota\leq \g\leq \b}\sum_{\tau\leq \g-\iota}\binom{\a}{\b}\binom{\b}{\g}\binom{\g-\iota}{\tau}L^{|\a-\b|}|\a-\b|! |\b-\g|!\bigg(\frac{8}{\nu_0}\bigg)^{|\b-\g|} \\
&\times \int_{\max\{0, t-\tb\}}^{t}e^{-\nu(v) (t-s)/4} (1+s)^{-\sigma/2}\int_{\mathbb{R}^3} \frac{w(v)}{w(u)}\big|\p_{v_\parallel}^{\iota}(\p_{v_\parallel}+\p_{u_\parallel})^{\tau} \mathbf{k}(v,u)\big| \dd u \dd s\\ 
& \times\sup_{0\leq s\leq t}\Vert (1+s)^{\sigma/2}w \langle k \rangle^{m+|\b-\g|}\p_{v_\parallel}^{\g-\tau-\iota}\hat f(s)\Vert_{ L^\infty_{ x_3,v}}.
 \end{align*}

Combining Lemma \ref{estdivk1ta}, \eqref{eettss}, \eqref{eettssint}, we derive
 \begin{align}\label{r2a2mid}
 &|\mathcal{A}_2| \notag\\
 & \leq C\sum_{\b\leq \a}\sum_{\iota\leq \g\leq \b}\sum_{\tau\leq \g-\iota}\frac{L_{\rho, m, |\a|}}{L_{\rho, m+|\b-\g|,|\g-\tau|-1}} \binom{\a}{\b}\binom{\b}{\g}\binom{\g-\iota}{\tau}L^{|\a-\b|}|\a-\b|! |\b-\g|!\bigg(\frac{8}{\nu_0}\bigg)^{|\b-\g|} |\tau|!\notag\\
 &\times L_{\rho, m+|\b-\g|,|\g-\tau|-1}(1+t)^{-\sigma/2}\bigg(\sup_{0\leq s\leq t}\Vert (1+s)^{\sigma/2}w \langle k \rangle^{m+|\b-\g|}\p_{v_\parallel}^{\g-\tau-\iota}\hat f(s)\Vert_{ L^\infty_{ x_3,v}}\bigg)
\notag \\ 
& \leq C (1+t)^{-\sigma/2}\bigg(\sup_{m, \a\geq 0}L_{\rho, m, |\a|} \sup_{0\leq s\leq t}
 \Vert(1+s)^{\sigma/2} w \avk^m\p_{v_\parallel}^{\a} \hat f(s)\Vert_{L^\infty_{x_3,v}}\bigg) \sum_{\b\leq \a}\sum_{\iota\leq \g\leq \b}\sum_{\tau\leq \g-\iota}\notag\\
 &\times\binom{\a}{\b}\binom{\b}{\g}\binom{\g-\iota}{\tau}L^{|\a-\b|}|\a-\b|! |\b-\g|! \bigg(\frac{8}{\nu_0}\bigg)^{|\b-\g|}|\tau|!
	\frac{L_{\rho, m, |\a|}}{L_{\rho, m+|\b-\g|,|\g-\tau|-1}}.
\end{align}

We apply \eqref{absab} and compute that
 \begin{align*}
 &\eqref{r2a2mid} = \binom{\a}{\b}\binom{\b}{\g}\binom{\g-\iota}{\tau}L^{|\a-\b|}|\a-\b|! |\b-\g|!\bigg(\frac{8}{\nu_0}\bigg)^{|\b-\g|}
\frac{L_{\rho, m, |\a|}}{L_{\rho, m+|\b-\g|,|\g-\tau|-1}}|\tau|!
\notag\\
&\leq \frac{|\a|!}{|\b|!(|\a|-|\b|)!}\frac{|\b|!}{|\g|!(|\b|-|\g|)!}\frac{(|\g|-1)!}{|\tau|!(|\g|-1-|\tau|)!}L^{|\a-\b|}|\a-\b|!|\b-\g|! \bigg(\frac{8}{\nu_0}\bigg)^{|\b-\g|}|\tau|!\notag\\&\quad\frac{{(16/\nu_0)}^{m}\rho^{m+|\a|+1}(m+2)^2(|\a|+2)^2}{(m+|\a|)!|\a|!}\frac{(m+|\b-\tau|-1)!(|\g-\tau|-1)!}{(16/\nu_0)^{m+|\b-\g|}\rho^{m+|\b-\tau|}(m+|\b-\g|+2)^2(|\g-\tau|+1)^2}
\notag\\
&\leq \rho\frac{(|\a|+2)^2}{(|\g-\tau|+1)^2}( L\rho)^{|\a-\b|}\bigg(\frac{1}{2}\bigg)^{|\b-\g|}\rho^{|\tau|}\frac{(m+|\b-\tau|-1)!}{(m+|\a|)!|\g|}\frac{(m+2)^2}{(m+|\b-\g|+2)^2}
%\notag\\&\leq \rho\frac{(|\a|+2)^2}{(|\g-\tau|+1)^2}\frac{( L\rho)^{|\a-\b|}}{|\a-\b|!}\bigg(\frac{1}{2}\bigg)^{|\b-\g|}\frac{\rho^{|\tau|}}{|\tau|!}
.
\end{align*}

We take summation in $\tau,\gamma,\beta$ and obtain
 \begin{align*} 
 &\sum_{\b\leq \a}\sum_{\iota\leq \g\leq \b}\sum_{\tau\leq \g-\iota}\eqref{r2a2mid}\leq \rho\sum_{\b\leq \a}\sum_{\iota\leq \g\leq \b}\sum_{\tau\leq \g-\iota}\frac{(|\a|+2)^2}{(|\g-\tau|+1)^2}\frac{( L\rho)^{|\a-\b|}}{|\a-\b|!}\bigg(\frac{1}{2}\bigg)^{|\b-\g|}\frac{\rho^{|\tau|}}{(|\tau|+1)!} \notag \\
&\leq 4\rho \sum_{\b\leq \a}\frac{(|\a|+2)^2}{(|\b|+2)^2}\frac{( L\rho)^{|\a-\b|}}{|\a-\b|!}\sum_{\iota\leq \g\leq \b}\frac{(|\b|+2)^2}{(|\g|+2)^2}\bigg(\frac{1}{2}\bigg)^{|\b-\g|}\sum_{\iota\leq \tau\leq \g}\frac{(|\g|+2)^2}{(|\tau|+2)^2}\frac{\rho^{|\gamma-\tau|}}{|\gamma-\tau|!}\leq C\rho. 
 \end{align*}
In the last line, we have applied $\frac{1}{(|\gamma-\tau|+1)^2}\leq \frac{4}{(|\gamma-\tau|+2)^2}$, and applied the change of variable in the summation: $ \tau \to \gamma-\tau $. In the last inequality, we have applied \eqref{abdiv2} and \eqref{inmati}.

 %\begin{align*} 
% &\sum_{\tau\leq \g-\iota}\frac{(|\g|+2)^2}{(|\g-\tau|+1)^2}\frac{\rho^{|\tau|}}{(|\tau|+1)!}=
% \\&\leq \sum_{|\tau|=0}^{[ |\g-e|/2]}\frac{1}{[(|\tau|+1)!]^{1/2}}\frac{\rho^{|\tau|}}{[(|\tau|+1)!]^{1/2}}+\sum_{|\tau|=[ |\g-e|/2]+1}^{[ |\g-e|]}\frac{(|\g|+2)^2}{[(|\tau|+1)!]^{1/2}}\frac{\rho^{|\tau|}}{[(|\tau|+1)!]^{1/2}}\leq C.
% \end{align*}

 Thus, we conclude
 \begin{equation}\label{mixdivr2secq2a2}
 |\mathcal{A}_2| \leq \rho C (1+t)^{-\sigma/2}\bigg(\sup_{m, \a\geq 0}L_{\rho, m, |\a|} \sup_{0\leq s\leq t}
 \Vert(1+s)^{\sigma/2} w \avk^m\p_{v_\parallel}^{\a} \hat f(s)\Vert_{L^\infty_{x_3,v}}\bigg).
\end{equation}
We combine the estimate \labelcref{mixdivr2secq2a1,mixdivr2secq2a2} with \eqref{mixdivr2sec} to conclude 
\begin{align}\label{mixdivr2kf}
&L_{\rho, m, |\a|} | w(v)\avk^m\p_{v_\parallel}^{\a}\eqref{chara:K} |\notag\\&\leq C(1+t)^{-\sigma/2}\bigg(\sup_{m\geq 0}L_{\rho, m,0}\sup_{0\leq s\leq t}\Vert(1+s)^{\sigma/2} w\langle k \rangle^{m} \hat f(s)\Vert_{ L^\infty_{ x_3,v}}\bigg)\notag\\&+\rho C (1+t)^{-\sigma/2}\bigg(\sup_{m, \a\geq 0}L_{\rho, m, |\a|} \sup_{0\leq s\leq t}
 \Vert(1+s)^{\sigma/2} w \avk^m\p_{v_\parallel}^{\a} \hat f(s)\Vert_{L^\infty_{x_3,v}}\bigg).
\end{align}

\textit{Step 3: estimate of \eqref{chara:Gamma}.}

In view of  \eqref{lingata}, we use Leibniz's formula to obtain 
\begin{align*}
& L_{\rho, m, |\a|} |w(v)\avk^m\p_{v_\parallel}^\a\eqref{chara:Gamma}|\notag\\
	& =L_{\rho, m, |\a|} \sum_{\b\leq \a}\sum_{ \g\leq \b}\binom{\a}{\b}\binom{\b}{\g}\int_{\max\{0, t-\tb\}}^{t} \Big|\Big(\p_{v_\parallel}^{\a-\b} e^{-\nu(v)(t-s)}\Big)\Big(\p_{v_\parallel}^{\b-\g} e^{- i(v_\parallel\cdot k)(t-s)}\Big )\avk^m \notag \\
	&\times w(v)\p_{v_\parallel}^{\g}\hat {\mathcal T}(\hat f,\hat f,\mu^{1/2})(s,k_1, k_2,x_3-(t-s)v_3,v)\Big| \dd s\notag\\ 
	 & \leq L_{\rho, m, |\a|} \sum_{\b\leq \a}\sum_{ \g\leq \b}\binom{\a}{\b}\binom{\b}{\g}L^{|\a-\b|}|\a-\b|! |\b-\g|!\bigg(\frac{8}{\nu_0}\bigg)^{|\b-\g|}\avk^{m+|\b-\g|} \notag \\
	&\times\int_{\max\{0, t-\tb\}}^{t} e^{- \nu(v)(t-s)/4} |w\p_{v_\parallel}^{\g}\hat{\mathcal T}( \hat f, \hat f,\mu^{1/2})(s,k_1, k_2,x_3-(t-s)v_3,v)|\dd s. 
 \end{align*}
In the second last line, we have used \eqref{tole} with replacing $t$ by $t-s$.

Then we apply \eqref{gamma_dexhk} to the nonlinear term and have
\begin{align*}
& L_{\rho, m, |\a|} |w(v)\avk^m\p_{v_\parallel}^\a\eqref{chara:Gamma}|\notag\\
	 & \leq C L_{\rho, m, |\a|} \sum_{\b\leq \a}\sum_{ \g\leq \b}\sum_{\delta\leq \g }\sum_{ \tau \leq\delta}\binom{\a}{\b}\binom{\b}{\g}\binom{\g}{\delta}\binom{\delta}{\tau}L^{|\a-\b|}|\a-\b|! |\b-\g|!\bigg(\frac{8}{\nu_0}\bigg)^{|\b-\g|} \notag\\
	&\times 8^{|\tau|+1}|\tau|!\avk^{m+|\b-\g|}\int_{\max\{0, t-\tb\}}^{t} \nu(v)e^{- \nu(v)(t-s)/4}(1+s)^{-\sigma/2}\dd s\notag \\
	&\times\int_{\mathbb{R}^2}\Big(\sup_{0\leq s\leq t}\Vert (1+s)^{\sigma/2}w\p_{v_\parallel}^{\g-\delta} \hat f(s,k-\ell)\Vert_{ L^\infty_{ x_3,v}}\Big)\Big(\sup_{0\leq s\leq t}\Vert (1+s)^{\sigma/2} w\p_{v_\parallel}^{\delta-\tau} \hat f(s,\ell)\Vert_{ L^\infty_{x_3,v}}\Big)\dd \ell \notag\\
	&\leq C(1+t)^{-\sigma/2}L_{\rho, m, |\a|} \sum_{\b\leq \a}\sum_{ \g\leq \b}\sum_{\delta\leq \g }\sum_{\tau\leq \delta}\sum_{j\leq m+|\b-\g|}\binom{\a}{\b}\binom{\b}{\g}\binom{\g}{\delta}\binom{\delta}{\tau}\binom{ m+|\b-\g|}{j}L^{|\a-\b|}\notag\\
	& \times|\a-\b|!|\b-\g| !\bigg(\frac{8}{\nu_0}\bigg)^{|\b-\g|}8^{|\tau|+1}|\tau|! \int_{\mathbb{R}^2}\Big(\sup_{0\leq s\leq t}\Vert (1+s)^{\sigma/2}w\langle k-\ell\rangle^{j} \p_{v_\parallel}^{\g-\delta} \hat f(s,k-\ell)\Vert_{ L^\infty_{ x_3,v}}\Big)\notag\\
	& \times\Big(\sup_{0\leq s\leq t}\Vert (1+s)^{\sigma/2} w \langle \ell\rangle^{m+|\b-\g|-j} \p_{v_\parallel}^{\delta-\tau} \hat f(s,\ell)\Vert_{ L^\infty_{ x_3,v}}\Big)\dd \ell.
 \end{align*}
In the second inequality, we have applied \eqref{klm} to $\langle k\rangle$, and applied \eqref{eettss} and \eqref{eettssint} to the $\dd s$ integration.
 
Including the Gevrey weight and rearranging the order, we further have
 \begin{align*}
& L_{\rho, m, |\a|} |w\avk^m\p_{v_\parallel}^\a\eqref{chara:Gamma}|
 \notag\\
 & \leq C(1+t)^{-\sigma/2} \int_{\mathbb{R}^2} \Big(\sup_{m,\a\geq 0} L_{\rho, m,|\a|}\sup_{0\leq s\leq t}\Vert (1+s)^{\sigma/2}w\langle k-\ell\rangle^{m} \p_{v_\parallel}^{\a} \hat f(s,k-\ell)\Vert_{ L^\infty_{ x_3,v}}\Big)\notag\\
	& \times \Big(\sup_{m,\a\geq 0} L_{\rho, m,|\a|}\sup_{0\leq s\leq t}\Vert (1+s)^{\sigma/2} w \langle \ell\rangle^{m} \p_{v_\parallel}^{\a} \hat f(s,\ell)\Vert_{ L^\infty_{ x_3,v}}\Big)\dd \ell\sum_{\b\leq \a}\sum_{ \g\leq \b}\sum_{\delta\leq \g }\sum_{\tau\leq \delta}\sum_{j\leq m+|\b-\g|}\binom{\a}{\b}\binom{\b}{\g}\notag\\
	&\times \binom{\g}{\delta} \binom{\delta}{\tau}\binom{ m+|\b-\g|}{j} L^{|\a-\b|}|\a-\b|! |\b-\g|!\bigg(\frac{8}{\nu_0}\bigg)^{|\b-\g|}8^{|\tau|+1}|\tau|! \frac{L_{\rho, m, |\a|}}{L_{\rho, j,|\g-\delta|}L_{\rho, m+|\b-\g|-j,|\delta-\tau|}}.
\end{align*}

Using \eqref{absab}, direct verification of the combinatorics terms shows that
\begin{align}
&\binom{\a}{\b}\binom{\b}{\g}\binom{\g}{\delta}\binom{\delta}{\tau}\binom{ m+|\b-\g|}{j} L^{|\a-\b|}|\a-\b|! |\b-\g|!\bigg(\frac{8}{\nu_0}\bigg)^{|\b-\g|}8^{|\tau|+1}|\tau|! \notag\\
 & \times\frac{L_{\rho, m, |\a|}}{L_{\rho, j,|\g-\delta|}L_{\rho, m+|\b-\g|-j,|\delta-\tau|}}
\notag \\
 &\leq \frac{|\a|!}{|\b|!(|\a|-|\b|)!}\frac{|\b|!}{|\g|!(|\b|-|\g|)!}\frac{|\g|!}{|\delta|!(|\g|-|\delta|)!}\frac{|\delta|!}{|\tau|!(|\delta|-|\tau|)!}\frac{(m+|\b-\g|)!}{j!(m+|\b-\g|-j)!}\notag\\
 & \times L^{|\a-\b|} |\a-\b|!|\b-\g|!\bigg(\frac{8}{\nu_0}\bigg)^{|\b-\g|}8^{|\tau|+1}|\tau|! \frac{(16/\nu_0)^{m}\rho^{m+|\a|+1}(m+2)^2(|\a|+2)^2}{(m+|\a|)!|\a|!} \notag\\
 & \times\frac{(j+|\g-\delta|)!|\g-\delta|!}{(16/\nu_0)^{j}\rho^{j+|\g-\delta|+1}(j+2)^2(|\g-\delta|+2)^2} \notag \\
 & \times \frac{(m+|\b-\g|-j+|\delta-\tau|)!|\delta-\tau|!}{(16/\nu_0)^{m+|\b-\g|-j}\rho^{m+|\b-\g|-j+|\delta-\tau|+1}(m+|\b-\g|-j+2)^2(|\delta-\tau|+2)^2}
\notag\\
&= \frac{(m+2)^2}{(j+2)^2(m+|\b-\g|-j+2)^2}\frac{ (|\a|+2)^2}{(|\g-\delta|+2)^2(|\delta-\tau|+2)^2} ( L\rho)^{|\a-\b|}\bigg(\frac{1}{2}\bigg)^{|\b-\g|} \notag \\
& \times 8^{|\tau|+1}\rho^{|\tau|-1} \frac{(m+|\b-\g|)!}{(m+|\a|)!}\frac{(j+|\g-\delta|)!(m+|\b-\g|-j+|\delta-\tau|)!}{j!(m+|\b-\g|-j)!}. \label{nonlinear_combinatoric}
\end{align}
We compute the last term in \eqref{nonlinear_combinatoric} using \eqref{framn} and \eqref{framn2} as
\begin{align*}%\label{sumtaal}
  &  \frac{(m+|\b-\g|)!}{(m+|\a|)!}\frac{(j+|\g-\delta|)!(m+|\b-\g|-j+|\delta-\tau|)!}{j!(m+|\b-\g|-j)!} \notag \\
  &\leq \frac{(m+|\b-\g|)!}{(m+|\a|)!}\frac{(m+|\b-\g|+|\g-\delta|)!}{(m+|\b-\g|)!} \frac{(m+|\b-\g|+|\delta-\tau|)!}{(m+|\b-\g|)!} \notag \\
  &= \frac{(m+|\b-\g|+|\g-\delta|)!}{(m+|\b-\g|+|\alpha|-|\b-\g|)!} \frac{(m+|\b-\g|+|\delta-\tau|)!}{(m+|\b-\g|)!}\notag \\
  & \leq \frac{(m+|\b-\g|)!}{(m+|\b-\g|+|\a-\b|+ |\delta|)!} \frac{(m+|\b-\g|+|\delta-\tau|)!}{(m+|\b-\g|)!} \leq \frac{(m+|\b-\g|+|\delta-\tau|)!}{(m+|\b-\g|+|\a-\b|+ |\delta|)!} \notag \\
  &\leq \frac{1}{(|\a-\b|+ |\tau|)!} \leq \frac{1}{|\a-\b|! |\tau|!}.
\end{align*}

Then we take summation in $j,\tau,\delta\,\gamma,\beta$ and obtain 
\begin{align*}
&\sum_{\b\leq \a}\sum_{ \g\leq \b}\sum_{\delta\leq \g }\sum_{\tau\leq \delta}\sum_{j\leq m+|\tau|} \eqref{nonlinear_combinatoric}\notag\\
 &\leq 8\rho^{-1}\sum_{\b\leq \a}\frac{ (|\a|+2)^2}{(|\b|+2)^2}\frac{( L\rho)^{|\a-\b|}}{|\a-\b|!}\sum_{ \g\leq \b}\frac{(|\b|+2)^2}{(|\g|+2)^2}\bigg(\frac{1}{2}\bigg)^{|\b-\g|}\sum_{\delta\leq \g }\frac{(|\g|+2)^2}{(|\g-\delta|+2)^2(|\delta|+2)^2} \notag\\
 & \times\sum_{\tau\leq \delta}\frac{(8\rho )^{|\tau|}}{|\tau|!}\frac{ (|\delta|+2)^2}{(|\delta-\tau|+2)^2} \sum_{j\leq m+|\b-\g|}\frac{(m+2)^2}{(j+2)^2(m+|\b-\g|-j+2)^2} \leq C.
 \end{align*}
In the last inequality, we have applied \eqref{able} to $j,\delta$ summation, \eqref{abdiv2} to $\gamma$ and applied \eqref{inmati} to $\beta$ summation, the $\tau$ summation is computed as
 \begin{align}\label{sum_ta}
 &\sum_{\tau\leq \delta}\frac{(8\rho )^{|\tau|}}{|\tau|!}\frac{ (|\delta|+2)^2}{(|\delta-\tau|+2)^2} \leq \sum_{\tau\leq \delta}\frac{(8\rho )^{|\tau|}}{|\tau|!}\frac{ (|\delta-\tau|+2)^2+(|\tau|+2)^2}{(|\delta-\tau|+2)^2}\notag \\&\leq 2\sum_{\tau\leq \delta}\frac{(8\rho )^{|\tau|}}{|\tau|!}(|\tau|+2)^2 \leq C.
 \end{align}
In \eqref{sum_ta}, we have applied \eqref{sum_tau_poly}. 
Therefore, we conclude
\begin{align}
 \label{mixdivr2nonli}
 &L_{\rho, m, |\a|} |w\avk^m\p_{v_\parallel}^\a\eqref{chara:Gamma}|\notag \\&\leq C (1+t)^{-\sigma/2} \int_{\mathbb{R}^2} \Big(\sup_{m,\a\geq 0} L_{\rho, m,|\a|}\sup_{0\leq s\leq t}\Vert (1+s)^{\sigma/2}w\langle k-\ell\rangle^{m} \p_{v_\parallel}^{\a} \hat f(s,k-\ell)\Vert_{ L^\infty_{ x_3,v}}\Big)\notag\\
	& \times \Big(\sup_{m,\a\geq 0} L_{\rho, m,|\a|}\sup_{0\leq s\leq t}\Vert (1+s)^{\sigma/2} w \langle \ell\rangle^{m} \p_{v_\parallel}^{\a} \hat f(s,\ell)\Vert_{ L^\infty_{ x_3,v}}\Big)\dd \ell.
\end{align}

\textit{Step 4: estimate of \eqref{chara:bdr}.}

Leibniz's formula gives
\begin{align}\label{divr2bdyb123}
 & L_{\rho, m, |\a|} | w\avk^m\p_{v_\parallel}^\a\eqref{chara:bdr} |
 \notag\\&=\mathbf{1}_{\tb \leq t} c_{\mu} L_{\rho, m, |\a|} \sum_{\b\leq \a} \sum_{\g\leq \b}\binom{\a}{\b}\binom{\b}{\g} \Big|\Big(\p_{v_\parallel}^{\a-\b}e^{-\nu(v) \tb}\Big)\Big(\p_{v_\parallel}^{\b-\g}e^{- i(v_\parallel\cdot k)\tb}\Big)\Big(\p_{v_\parallel}^{\g}\sqrt{\mu(v)}\Big)
 \notag\\
 &\times w(v)\int_{u_3<0}\avk^m \hat{f}(t^1, k,0,u)\sqrt{\mu(u)}|u_3| \dd u\Big|.
\end{align}
We use \eqref{kmu} and \eqref{tole} with replacing $t$ by $\tb$ to derive
 \begin{align}
 & L_{\rho, m, |\a|} | w\avk^m\p_{v_\parallel}^\a\eqref{chara:bdr} |
 \notag\\
 &\leq L_{\rho, m, |\a|} \sum_{\b\leq \a} \sum_{\g\leq \b}\binom{\a}{\b}\binom{\b}{\g}L^{|\a-\b|}|\a-\b|!|\b-\g|!\bigg(\frac{8}{\nu_0}\bigg)^{|\b-\g|}e^{-\nu(v) \tb/4}8^{|\g|+1}|\g|!\mu^{1/4}(v)w(v)
 \notag\\
 & \times \int_{u_3<0}\avk^{m+|\b-\g|} \hat{f}(t^1, k,0,u)\sqrt{\mu(u)}|u_3| \dd u
\notag\\
&\leq \sum_{\b\leq \a} \sum_{\g\leq \b}\binom{\a}{\b}\binom{\b}{\g}L^{|\a-\b|}|\a-\b|!|\b-\g|!\bigg(\frac{8}{\nu_0}\bigg)^{|\b-\g|}e^{-\nu(v) \tb/4}8^{|\g|+1}|\g|! \frac{ L_{\rho, m, |\a|}}{L_{\rho,m+|\beta-\gamma|,0}} (1+t^1)^{-\sigma/2}\notag\\
& \times L_{\rho,m+|\beta-\gamma|,0}\bigg(\sup_{0\leq t^1\leq t}\Vert (1+t^1)^{\sigma/2}w\avk^{m+|\b-\g|} \hat{f}(t^1)\Vert_{ L^\infty_{ x_3,v}}\bigg)\int_{u_3<0}w^{-1}(u)\sqrt{\mu(u)}|u_3| \dd u \notag\\
& \leq(1+t)^{-\sigma/2}\bigg(\sup_{m\geq 0}L_{\rho, m, 0}\sup_{0\leq t^1\leq t}\Vert (1+t^1)^{\sigma/2}w\avk^{m} \hat{f}(t^1)\Vert_{ L^\infty_{ x_3,v}}\bigg)\sum_{\b\leq \a} \sum_{\g\leq \b}
 \notag\\
& \times \binom{\a}{\b}\binom{\b}{\g}L^{|\a-\b|}|\a-\b|!|\b-\g|!\bigg(\frac{8}{\nu_0}\bigg)^{|\b-\g|}8^{|\g|+1}|\g|!\frac{L_{\rho, m, |\a|}}{ L_{\rho, m+|\b-\g|, 0}}. \label{divr2bdyno}
\end{align}

In the second inequality, we have used $\mu^{1/4}(v)w(v) \leq 1.$ In the last inequality, we have used $e^{-\nu(v) \tb/4}(1+t^1)^{-\sigma/2} \leq C(1+t)^{-\sigma/2} $ and $\int_{u_3<0}w^{-1}(u)\sqrt{\mu(u)}|u_3| \dd u \leq C$.

 We compute
 \begin{align}\label{r2bdysumcomfs}
 &\eqref{divr2bdyno} = \binom{\a}{\b}\binom{\b}{\g}L^{|\a-\b|}|\a-\b|!|\b-\g|!\bigg(\frac{8}{\nu_0}\bigg)^{|\b-\g|}8^{|\g|+1}|\g|!\frac{ L_{\rho, m, |\a|} }{ L_{\rho, m+|\b-\g|, 0} } \notag\\&\leq \frac{|\a|!}{|\b|!(|\a|-|\b|)!}\frac{|\b|!}{|\g|!(|\b|-|\g|)!}L^{|\a-\b|}|\a-\b|!|\b-\g|!\bigg(\frac{8}{\nu_0}\bigg)^{|\b-\g|} 8^{|\g|+1}|\g|!\notag\\&\quad\frac{{(16/\nu_0)}^{m}\rho^{m+|\a|+1}(m+2)^2(|\a|+2)^2}{(m+|\a|)!|\a|!}\frac{(m+|\b-\g|)!}{(16/\nu_0)^{m+|\b-\g|}\rho^{m+|\b-\g|+1}(m+|\b-\g|+2)^2}
\notag\\&\leq8 (L\rho)^{|\a-\b|}\bigg(\frac{1}{2}\bigg)^{|\b-\g|}(8\rho)^{\g} (|\a|+2)^2\frac{(m+|\b-\g|)!}{(m+|\a|)!}\frac{(m+2)^2}{(m+|\b-\g|+2)^2}.
\end{align}
Taking summation in $\alpha,\beta$, we obtain
\begin{align*}
 & \sum_{\b\leq \a} \sum_{\g\leq \b}\eqref{r2bdysumcomfs} \leq 8\sum_{\b\leq \a} \sum_{\g\leq \b}\frac{(|\a|+2)^2}{(|\g|+2)^2}\frac{(L\rho)^{|\a-\b|}}{|\a-\b|!}\bigg(\frac{1}{2}\bigg)^{|\b-\g|}\frac{(8\rho)^{\g}(|\g|+2)^2}{|\g|!}\\
&\leq 8\sum_{\b\leq \a} \sum_{\g\leq \b}\frac{(|\a|+2)^2}{(|\g|+2)^2}\frac{(L\rho)^{|\a-\b|}}{|\a-\b|!}\bigg(\frac{1}{2}\bigg)^{|\b-\g|}\leq C.
\end{align*}
In the last inequality, we have applied the same computation as \eqref{sumconfir}.

Therefore, we conclude
 \begin{align}\label{mixdivr2bdy.add}
L_{\rho, m, |\a|} | w\avk^m\p_{v_\parallel}^\a\eqref{chara:bdr} |\leq C (1+t)^{-\sigma/2}\bigg(\sup_{m\geq 0}L_{\rho, m, 0}\sup_{0\leq t^1\leq t}\Vert (1+t^1)^{\sigma/2}w\avk^{m} \hat{f}(t^1)\Vert_{ L^\infty_{ x_3,v}}\bigg).
\end{align}

We collect \labelcref{mixdivr2infi,mixdivr2kf,mixdivr2nonli,mixdivr2bdy.add} to obtain 
\begin{align*}
& \sup_{m,\a\geq 0}L_{\rho, m, |\a|} \Vert w \langle k\rangle^{m} \p_{v_\parallel}^{\a} \hat f(k)\Vert_{ L^\infty_{ x_3,v}} \leq Ce^{-\nu_0 t/4}\sup_{m, \a\geq 0}L_{\rho, m, |\a|} \Vert w \langle k \rangle ^{m}\p_{v_\parallel}^{\a}\hat{ f}_0 \Vert_{ L^\infty_{x_3,v}}
\\&+ C(1+t)^{-\sigma/2}\bigg(\sup_{m\geq 0}L_{\rho, m,0}\sup_{0\leq s\leq t}\Vert(1+s)^{\sigma/2} w\langle k \rangle^{m} \hat f(s)\Vert_{ L^\infty_{ x_3,v}}\bigg)\notag\\&+\rho C (1+t)^{-\sigma/2}\bigg(\sup_{m, \a\geq 0}L_{\rho, m, |\a|} \sup_{0\leq s\leq t}
 \Vert(1+s)^{\sigma/2} w \avk^m\p_{v_\parallel}^{\a} \hat f(s)\Vert_{L^\infty_{x_3,v}}\bigg)\\&+C (1+t)^{-\sigma/2} \int_{\mathbb{R}^2} \Big(\sup_{m,\a\geq 0} L_{\rho, m,|\a|}\sup_{0\leq s\leq t}\Vert (1+s)^{\sigma/2}w\langle k-\ell\rangle^{m} \p_{v_\parallel}^{\a} \hat f(s,k-\ell)\Vert_{ L^\infty_{ x_3,v}}\Big)\notag\\
	& \times \Big(\sup_{m,\a\geq 0} L_{\rho, m,|\a|}\sup_{0\leq s\leq t}\Vert (1+s)^{\sigma/2} w \langle \ell\rangle^{m} \p_{v_\parallel}^{\a} \hat f(s,\ell)\Vert_{ L^\infty_{ x_3,v}}\Big)\dd \ell.
\end{align*}
Finally, we conclude that, for any $T$ and some constant $C$ that does not depend on $T$, 
\begin{align*}
& \sup_{m,\a\geq 0}L_{\rho, m, |\a|} \Vert (1+t)^{\sigma/2}w \langle k\rangle^{m} \p_{v_\parallel}^{\a} \hat f(k)\Vert_{ L^\infty_{T, x_3,v}} \\&\leq C\sup_{m, \a\geq 0}L_{\rho, m, |\a|} \Vert w \langle k \rangle ^{m}\p_{v_\parallel}^{\a}\hat{ f}_0 \Vert_{ L^\infty_{x_3,v}}
+ C\bigg(\sup_{m\geq 0}L_{\rho, m,0}\Vert(1+t)^{\sigma/2} w\langle k \rangle^{m} \hat f\Vert_{ L^\infty_{T, x_3,v}}\bigg)\notag\\&+\rho C \bigg(\sup_{m, \a\geq 0}L_{\rho, m, |\a|} 
 \Vert(1+t)^{\sigma/2} w \avk^m\p_{v_\parallel}^{\a} \hat f\Vert_{L^\infty_{T, x_3,v}}\bigg)\\&+C \int_{\mathbb{R}^2} \Big(\sup_{m,\a\geq 0} L_{\rho, m,|\a|}\Vert (1+t)^{\sigma/2}w\langle k-\ell\rangle^{m} \p_{v_\parallel}^{\a} \hat f(k-\ell)\Vert_{ L^\infty_{T, x_3,v}}\Big)\notag\\
	& \times \Big(\sup_{m,\a\geq 0} L_{\rho, m,|\a|}\Vert (1+t)^{\sigma/2} w \langle \ell\rangle^{m} \p_{v_\parallel}^{\a} \hat f(\ell)\Vert_{ L^\infty_{ T, x_3,v}}\Big)\dd \ell.
\end{align*}
We take $L^1_k-$norm and use
Young's convolution inequality to obtain 
\begin{align*}
& \sup_{m,\a\geq 0}L_{\rho, m, |\a|} \Vert (1+t)^{\sigma/2}w \langle k\rangle^{m} \p_{v_\parallel}^{\a} \hat f(k)\Vert_{L^1_k L^\infty_{T, x_3,v}}\\& \leq C\sup_{m, \a\geq 0}L_{\rho, m, |\a|} \Vert w \langle k \rangle ^{m}\p_{v_\parallel}^{\a}\hat{ f}_0 \Vert_{L^1_k L^\infty_{x_3,v}}
+ C\bigg(\sup_{m\geq 0}L_{\rho, m,0}\Vert(1+t)^{\sigma/2} w\langle k \rangle^{m} \hat f\Vert_{L^1_k L^\infty_{T, x_3,v}}\bigg)\notag\\&+\rho C \bigg(\sup_{m, \a\geq 0}L_{\rho, m, |\a|} 
 \Vert(1+t)^{\sigma/2} w \avk^m\p_{v_\parallel}^{\a} \hat f\Vert_{L^1_kL^\infty_{T, x_3,v}}\bigg)\\&+C \bigg(\sup_{m,\a\geq 0} L_{\rho, m,|\a|}\Vert (1+t)^{\sigma/2}w\langle k\rangle^{m} \p_{v_\parallel}^{\a} \hat f(k)\Vert_{L^1_k L^\infty_{T, x_3,v}}\bigg)^2.
\end{align*}
We absorb the third line with $\rho\leq 1/(2C)$.
Combining with \eqref{f_estxxx}, we conclude the proof of Proposition \ref{prop:inftyxvr2}.
\end{proof}

\begin{proof}[\textbf{Proof of Theorem \ref{thm:apriestxvr2}}]
With the a priori estimate in Proposition \ref{prop:inftyxvr2}, the proof of the theorem is similar to Theorem \ref{thm:l1k_lpkexg}. So we omit the details.
\end{proof}

\appendix

\section{Elementary identities and inequalities}\label{appendi}

To prove Lemma \ref{estdivk1ta}, we need the following identity.  

\begin{lemma}\label{lemma:n_derivative}
 Suppose $f'(x)=g(x), \ g'(x)=h(x), \ h'(x)=0$. The $n$-th order derivative of $(e^f)^{(n)}$ has the following expression:
 \begin{equation}\label{induc}
(e^f)^{(n)}=e^f\sum_{m=0}^{[n/2]}\frac{n!}{m!(n-2m)!}\bigg(\frac{h}{2}\bigg)^mg^{n-2m}.
\end{equation} 
\end{lemma}
\begin{proof}
We use induction on $n$ to prove \eqref{induc}.
When $n=0$, $$LHS=e^f=RHS.$$ When $n=1$, $$LHS=(e^f)'=e^fg=RHS.$$ 
Suppose 
\eqref{induc}
holds true for $n$. We will prove that %it still holds true for $p=n+1$. 
 \begin{equation*}
(e^f)^{(n+1)}=e^f\sum_{m=0}^{[(n+1)/2]}\frac{(n+1)!}{m!(n+1-2m)!}\bigg(\frac{h}{2}\bigg)^mg^{n+1-2m}.
\end{equation*}

When $n=2l, \ l\in \mathbb{Z}_+$,
\begin{align*}
(e^f)^{n+1}&=\frac{\dd }{\dd x}\bigg[e^f\sum_{m=0}^{l}\frac{(2l)!}{m!(2l-2m)!}\bigg(\frac{h}{2}\bigg)^mg^{2l-2m}\bigg]
\\&=\frac{\dd}{\dd x}\bigg[e^f\frac{(2l)!}{l!}\bigg(\frac{h}{2}\bigg)^l+e^f\sum_{m=0}^{l-1}\frac{(2l)!}{m!(2l-2m)!}\bigg(\frac{h}{2}\bigg)^mg^{2l-2m}\bigg]
\\&=e^f\frac{(2l)!}{l!}\bigg(\frac{h}{2}\bigg)^lg+e^f\sum_{m=0}^{l-1}\frac{(2l)!}{m!(2l-2m)!}\bigg(\frac{h}{2}\bigg)^mg^{2l-2m+1}\\&+2e^f\sum_{m=0}^{l-1}\frac{(2l)!}{m!(2l-2m-1)!}\bigg(\frac{h}{2}\bigg)^{m+1}g^{2l-2m-1} \\
&=e^f\frac{(2l)!}{l!}\bigg(\frac{h}{2}\bigg)^lg+e^f\sum_{m=0}^{l-1}\frac{(2l)!}{m!(2l-2m)!}\bigg(\frac{h}{2}\bigg)^mg^{2l-2m+1}\\&+2e^f\sum_{m=1}^{l}\frac{(2l)!}{(m-1)!(2l-2m+1)!}\bigg(\frac{h}{2}\bigg)^{m}g^{2l-2m+1}.
\end{align*}
For the last equality, we applied the change of variable $m+1\rightarrow m$ for the last term.

For the second term in the second last line, we isolate $m=0$, in the last line, we isolate $m=l$, then we further have
\begin{align*}
(e^f)^{(n+1)}&=e^fg^{2l+1}+e^f\frac{2(2l)!}{(l-1)!}\bigg(\frac{h}{2}\bigg)^lg+e^f\frac{(2l)!}{l!}\bigg(\frac{h}{2}\bigg)^lg\\&+e^f\sum_{m=1}^{l-1}\bigg[\frac{(2l)!}{m!(2l-2m)!}+\frac{2(2l)!}{(m-1)!(2l-2m+1)!}\bigg]\bigg(\frac{h}{2}\bigg)^mg^{2l-2m+1}
\\&=e^fg^{2l+1}+e^f\frac{(2l+1)!}{l!}\bigg(\frac{h}{2}\bigg)^lg+e^f\sum_{m=1}^{l-1}\frac{(2l+1)!}{m!(2l+1-2m)!}\bigg(\frac{h}{2}\bigg)^mg^{2l-2m+1}
\\&=e^f\sum_{m=0}^{l}\frac{(2l+1)!}{m!(2l+1-2m)!}\bigg(\frac{h}{2}\bigg)^mg^{2l-2m+1}
\\&=e^f\sum_{m=0}^{[(n+1)/2]}\frac{(n+1)!}{m!(n+1-2m)!}\bigg(\frac{h}{2}\bigg)^mg^{n+1-2m}.
\end{align*}
In the second equality, we have used 
\begin{equation*}
\frac{2(2l)!}{(l-1)!}+\frac{(2l)!}{l!}=\frac{2l(2l)!}{l!}+\frac{(2l)!}{l!}=\frac{(2l+1)!}{l!},
\end{equation*}
and 
\begin{align*}
 &\frac{(2l)!}{m!(2l-2m)!}+\frac{2(2l)!}{(m-1)!(2l-2m+1)!} \\
 &=\frac{(2l-2m+1)(2l)!}{m!(2l-2m+1)!}+\frac{2m(2l)!}{m!(2l-2m+1)!} =\frac{(2l+1)!}{m!(2l+1-2m)!}.
\end{align*}

When $n=2l+1, \ l\in \mathbb{Z}_+$,
\begin{align*}
(e^f)^{(n+1)}&=\frac{\dd }{\dd x}\bigg[e^f\sum_{m=0}^{l}\frac{(2l+1)!}{m!(2l+1-2m)!}\bigg(\frac{h}{2}\bigg)^mg^{2l+1-2m}\bigg]
\\
&=e^f\sum_{m=0}^{l}\frac{(2l+1)!}{m!(2l+1-2m)!}\bigg(\frac{h}{2}\bigg)^mg^{2l+2-2m}+2e^f\sum_{m=0}^{l}\frac{(2l+1)!}{m!(2l-2m)!}\bigg(\frac{h}{2}\bigg)^{m+1}g^{2l-2m}\\
&=e^f\sum_{m=0}^{l}\frac{(2l+1)!}{m!(2l+1-2m)!}\bigg(\frac{h}{2}\bigg)^mg^{2l+2-2m}\\&+2e^f\sum_{m=1}^{l+1}\frac{(2l+1)!}{(m-1)!(2l+2-2m)!}\bigg(\frac{h}{2}\bigg)^{m}g^{2l+2-2m}.
\end{align*}
In the last inequality, we have applied the change of variable $m+1\rightarrow m$ for the second term.

Isolating $m=0$ in the second last line and $m=l+1$ in the last line, we further obtain
\begin{align*}
(e^f)^{(n+1)}&=e^fg^{2l+2}+2e^f\frac{(2l+1)!}{l!}\bigg(\frac{h}{2}\bigg)^{l+1}\\&+e^f\sum_{m=1}^{l}\bigg[\frac{(2l+1)!}{m!(2l+1-2m)!}+\frac{2(2l+1)!}{(m-1)!(2l+2-2m)!}\bigg]\bigg(\frac{h}{2}\bigg)^mg^{2l+2-2m}
\\&=e^fg^{2l+2}+e^f\frac{(2l+2)!}{(l+1)!}\bigg(\frac{h}{2}\bigg)^{l+1}+e^f\sum_{m=1}^{l}\frac{(2l+2)!}{m!(2l+2-2m)!}\bigg(\frac{h}{2}\bigg)^mg^{2l+2-2m}
\\&=e^f\sum_{m=0}^{l+1}\frac{(2l+2)!}{m!(2l+2-2m)!}\bigg(\frac{h}{2}\bigg)^mg^{2l+2-2m}
\\&=e^f\sum_{m=0}^{[(n+1)/2]}\frac{(n+1)!}{m!(n+1-2m)!}\bigg(\frac{h}{2}\bigg)^mg^{n+1-2m}.
\end{align*} 
In the second equality, we have used 
\begin{align*}
2\frac{(2l+1)!}{l!}=\frac{(2l+2)!}{(l+1)!},
\end{align*}
and 
\begin{align*}
\frac{(2l+1)!}{m!(2l+1-2m)!}+\frac{2(2l+1)!}{(m-1)!(2l+2-2m)!}=&\frac{(2l+2-2m)(2l+1)!}{m!(2l+2-2m)!}+\frac{2m(2l+1)!}{m!(2l+2-2m)!}\\&=\frac{(2l+2)!}{m!(2l+2-2m)!}.
\end{align*}
We conclude the proof of Lemma \ref{lemma:n_derivative}.
\end{proof}

Moreover, for the proof of Lemma \ref{estdivk1ta}, we also need the following estimates. 

\begin{lemma}\label{lemma:derk12} 
It holds that
\begin{align}\label{derk1} 
&\sum_{m=0}^{[\tau_2/2]}\frac{1}{m!(\tau_2-2m)!}\sum_{j=0}^{[\tau_1/2]}\frac{1}{j!(\tau_1-2j)!}\bigg|\bigg(\frac{1}{4}\bigg)^m\bigg(\frac{1}{4}\bigg)^j\notag\\&\times\int_{\mathbb{R}^3}\frac{w(v)}{w(u)} \p_{v_\parallel}^\iota\bigg[|v-u|e^{-\frac{1}{4}(|v|^2+|u|^2)}\bigg(-\frac{v_2+u_2}{2}\bigg)^{\tau_2-2m}\bigg(-\frac{v_1+u_1}{2}\bigg)^{\tau_1-2j}\bigg]\dd u\bigg|\leq C\nu(v),
\end{align}
and
\begin{align}\label{derk2} 
&\sum_{m=0}^{[\tau_2/2]}\frac{1}{m!}\sum_{j=0}^{\tau_1}\frac{1}{(\tau_1-j)!}\frac{1}{[\tau_2-2m-(\tau_1-j)]!}\sum_{p=0}^{[j/2]}\frac{1}{p!(j-2p)!}\notag\\&\times \bigg|\int_{\mathbb{R}^3}\frac{w(v)}{w(u)} \p_{v_\parallel}^\iota\bigg[|v-u|^{-1}e^{-\frac{1}{8}|v-u|^{2}}\bigg(\frac{C_2}{2}\bigg)^me^A\bigg(\frac{C_1}{2}\bigg)^pB_1^{j-2p}B_2^{\tau_2-2m-(\tau_1-j)}D^{\tau_1-j}\bigg] \dd u \bigg|
\notag\\& \leq C \nu(v) .
\end{align}
\end{lemma}

\begin{proof}
Without loss of generality, we set $\iota = (1, 0)$. We begin by proving the first inequality \eqref{derk1}. We take derivative and have 
 \begin{align*}
& \p_{v_\parallel}^\iota\bigg[|v-u|e^{-\frac{1}{4}(|v|^2+|u|^2)}\bigg(-\frac{v_2+u_2}{2}\bigg)^{\tau_2-2m}\bigg(-\frac{v_1+u_1}{2}\bigg)^{\tau_1-2j}\bigg]
\\&= e^{-\frac{1}{4}(|v|^2+|u|^2)}\bigg(-\frac{v_2+u_2}{2}\bigg)^{\tau_2-2m}\bigg[\frac{v_1-u_1}{|v-u|}\bigg(-\frac{v_1+u_1}{2}\bigg)^{\tau_1-2j}-\frac{v_1|v-u|}{2}\bigg(-\frac{v_1+u_1}{2}\bigg)^{\tau_1-2j}\\&-\frac{\tau_1-2j}{2}|v-u|\bigg(-\frac{v_1+u_1}{2}\bigg)^{\tau_1-2j-1}\bigg] 
%\\&\leq e^{-\frac{1}{4}(|v|^2+|u|^2)}\big(v^2+u^2\big)^{(\tau_2-2m)/2}\\&\times \Big[\big(v^2+u^2\big)^{(\tau_1-2j)/2}+v_1\big(v^2+u^2\big)^{(\tau_1-2j+1)/2}+(\tau_1-2j)\big(v^2+u^2\big)^{(\tau_1-2j)/2}\Big]
.
\end{align*}
%where in the last inequality 
Then, we use $[(v_i+u_i)/2]^2\leq |v|^2+|u|^2,\ i=1,2$ and $|v-u|^2/2\leq |v|^2+|u|^2$ to derive
\begin{align*}
&\bigg|\p_{v_\parallel}^\iota\bigg[|v-u|e^{-\frac{1}{4}(|v|^2+|u|^2)}\bigg(-\frac{v_2+u_2}{2}\bigg)^{\tau_2-2m}\bigg(-\frac{v_1+u_1}{2}\bigg)^{\tau_1-2j}\bigg]\bigg|
\\&\leq e^{-\frac{1}{4}(|v|^2+|u|^2)}\big(|v|^2+|u|^2\big)^{(\tau_2-2m)/2} \Big[\big(|v|^2+|u|^2\big)^{(\tau_1-2j)/2}+|v_1|\big(|v|^2+|u|^2\big)^{(\tau_1-2j+1)/2}\\&+(\tau_1-2j)\big(|v|^2+|u|^2\big)^{(\tau_1-2j)/2}\Big]
\\&\leq C \nu(v)e^{-\frac{1}{8}(|v|^2+|u|^2)}[(\tau_2-2m)/2]!16^{(\tau_2-2m)/2}(1+\tau_1-2j)[(\tau_1-2j+1)/2]!16^{(\tau_1-2j+1)/2}
\\&\leq C \nu(v)e^{-\frac{1}{8}(|v|^2+|u|^2)}[(\tau_2-2m)/2]!16^{(\tau_2-2m)/2}[(\tau_1-2j+2)/2]!16^{(\tau_1-2j+1)/2}.
\end{align*}
Here in the fourth line, we use Lemma \ref{lemma:enut_control} with replacing $\nu(v)t$ and 8 by $|v|^2+|u|^2$ and 16, respectively.
%\begin{equation*}
% e^{-\frac{1}{16}(|v|^2+|u|^2)}\big(|v|^2+|u|^2\big)^{i}= e^{-\frac{1}{16}(|v|^2+|u|^2)}\bigg(\frac{|v|^2+|u|^2}{16}\bigg)^{i}16^i\leq 16^ii!, \ \forall i\in \mathbb{Z}^+.
%\end{equation*}

This leads to
 \begin{align*}
& \int_{\mathbb{R}^3}\frac{w(v)}{w(u)} \p_{v_\parallel}^\iota\bigg[|v-u|e^{-\frac{1}{4}(|v|^2+|u|^2)}\bigg(-\frac{v_2+u_2}{2}\bigg)^{\tau_2-2m}\bigg(-\frac{v_1+u_1}{2}\bigg)^{\tau_1-2j}\bigg]\dd u\bigg| 
\\&\leq C\nu(v)[(\tau_2-2m)/2]!16^{(\tau_2-2m)/2}[(\tau_1-2j+2)/2]!16^{(\tau_1-2j+1)/2}\bigg|\int_{\mathbb{R}^3}\frac{w(v)}{w(u)}e^{-\frac{1}{8}(|v|^2+|u|^2)} \dd u\bigg| \\&\leq C\nu(v)[(\tau_2-2m)/2]!16^{(\tau_2-2m)/2}[(\tau_1-2j+2)/2]!16^{(\tau_1-2j+1)/2}.
\end{align*}

Thus, we conclude 
\begin{align*}
&\sum_{m=0}^{[\tau_2/2]}\frac{1}{m!(\tau_2-2m)!}\sum_{j=0}^{[\tau_1/2]}\frac{1}{j!(\tau_1-2j)!}\bigg|\bigg(\frac{1}{4}\bigg)^m\bigg(\frac{1}{4}\bigg)^j\\&\times\int_{\mathbb{R}^3}\frac{w(v)}{w(u)} \p_{v_\parallel}^\iota\bigg[|v-u|e^{-\frac{1}{4}(|v|^2+|u|^2)}\bigg(-\frac{v_2+u_2}{2}\bigg)^{\tau_2-2m}\bigg(-\frac{v_1+u_1}{2}\bigg)^{\tau_1-2j}\bigg]\dd u\bigg|
%\\&\leq C\nu(v)\sum_{m=0}^{[\tau_2/2]}\frac{1}{m!(\tau_2-2m)!}\sum_{j=0}^{[\tau_1/2]}\frac{1}{j!(\tau_1-2j)!}\bigg(\frac{1}{4}\bigg)^m\bigg(\frac{1}{4}\bigg)^j\\&\times[(\tau_2-2m)/2]!16^{(\tau_2-2m)/2}[(\tau_1-2j+2)/2]!16^{(\tau_1-2j+1)/2}
\\&\leq C\nu(v)\sum_{m=0}^{[\tau_2/2]}\frac{1}{m!}\bigg(\frac{1}{4}\bigg)^m\sum_{j=0}^{[\tau_1/2]}\frac{1}{j!}\bigg(\frac{1}{4}\bigg)^j\frac{[(\tau_2-2m)/2]!16^{(\tau_2-2m)/2}}{(\tau_2-2m)!}\frac{[(\tau_1-2j+2)/2]!16^{(\tau_1-2j+1)/2}}{(\tau_1-2j)!}
\\&\leq C\nu(v) .
\end{align*}
Here in the last inequality, we use 
\begin{align*}
&\frac{[(\tau_2-2m)/2]!16^{(\tau_2-2m)/2}}{(\tau_2-2m)!}\frac{[(\tau_1-2j+2)/2]!16^{(\tau_1-2j+1)/2}}{(\tau_1-2j)!}\leq C,
\end{align*}
which is following from \eqref{MN2}.

Next, we prove the second inequality \eqref{derk2}. Recall the notation defined in \eqref{ABCDEF}, we compute the derivatives as
\iffalse
\begin{align*}
& \p_{v_1}|v-u|^{-1}=-\frac{v_1-u_1}{|v-u|^{3}},
\\& \p_{v_1}e^{-\frac{1}{8}|v-u|^{2}}=-\frac{1}{4}e^{-\frac{1}{8}|v-u|^{2}}(v_1-u_1),
\\& \p_{v_1}\bigg(\frac{C_2}{2}\bigg)^m=m\bigg(\frac{C_2}{2}\bigg)^{m-1}\frac{(v_1-u_1)(v_2-u_2)^2}{|v-u|^{4}},
\\& \p_{v_1}e^A=e^A\bigg[-\frac{v_1(|v|^2-|u|^2)}{2|v-u|^{2}}+\frac{(v_1-u_1)(|v|^2-|u|^2)^2}{4|v-u|^{4}}\bigg],
\\& \p_{v_1}\bigg(\frac{C_1}{2}\bigg)^p=p\bigg[-\frac{v_1-u_1}{|v-u|^{2}}+\frac{(v_1-u_1)^3}{|v-u|^{4}}\bigg]\bigg(\frac{C_1}{2}\bigg)^{p-1},
\\& \p_{v_1}B_1^{j-2p}=(j-2p)\bigg[-\frac{v_1(v_1-u_1)}{|v-u|^{2}}-\frac{|v|^2-|u|^2}{2|v-u|^{2}}+\frac{(v_1-u_1)^2(|v|^2-|u|^2)}{|v-u|^{4}}\bigg]B_1^{j-2p-1},
\\& \p_{v_1}B_2^{\tau_2-2m-(\tau_1-j)}=[\tau_2-2m-(\tau_1-j)]B_2^{\tau_2-2m-(\tau_1-j)-1}\\&\qquad\qquad\qquad\qquad \times \bigg[-\frac{v_1(v_2-u_2)}{|v-u|^{2}}+\frac{(v_1-u_1)(v_2-u_2)(|v|^2-|u|^2)}{|v-u|^{4}}\bigg],
\\& \p_{v_1}D^{\tau_1-j}=(\tau_1-j)\bigg[-\frac{v_2-u_2}{|v-u|^{2}}+\frac{(v_1-u_1)^2(v_2-u_2)}{|v-u|^{4}}\bigg]D^{\tau_1-j-1}.
\end{align*}
 \fi
\begin{align*}
&\ \p_{v_\parallel}^\iota\bigg[|v-u|^{-1}e^{-\frac{1}{8}|v-u|^{2}}\bigg(\frac{C_2}{2}\bigg)^me^A\bigg(\frac{C_1}{2}\bigg)^pB_1^{j-2p}B_2^{\tau_2-2m-(\tau_1-j)}D^{\tau_1-j}\bigg]
\\&=\mathcal{K}_1+\mathcal{K}_2+\mathcal{K}_3+\mathcal{K}_4+\mathcal{K}_5+\mathcal{K}_6+\mathcal{K}_7+\mathcal{K}_8,
\end{align*}
where
\begin{align*}
&\mathcal{K}_1=-\frac{v_1-u_1}{|v-u|^{3}}e^{-\frac{1}{8}|v-u|^{2}}\bigg(\frac{C_2}{2}\bigg)^me^A\bigg(\frac{C_1}{2}\bigg)^pB_1^{j-2p}B_2^{\tau_2-2m-(\tau_1-j)}D^{\tau_1-j},
\\&\mathcal{K}_2=-\frac{v_1-u_1}{4}|v-u|^{-1}e^{-\frac{1}{8}|v-u|^{2}}\bigg(\frac{C_2}{2}\bigg)^me^A\bigg(\frac{C_1}{2}\bigg)^pB_1^{j-2p}B_2^{\tau_2-2m-(\tau_1-j)}D^{\tau_1-j},
\\&\mathcal{K}_3=m\frac{(v_1-u_1)(v_2-u_2)^2}{|v-u|^{4}}|v-u|^{-1}e^{-\frac{1}{8}|v-u|^{2}}\bigg(\frac{C_2}{2}\bigg)^{m-1}e^A\bigg(\frac{C_1}{2}\bigg)^pB_1^{j-2p}B_2^{\tau_2-2m-(\tau_1-j)}D^{\tau_1-j},
\\&\mathcal{K}_4=\bigg[-\frac{v_1(|v|^2-|u|^2)}{2|v-u|^{2}}+\frac{(v_1-u_1)(|v|^2-|u|^2)^2}{4|v-u|^{4}}\bigg]|v-u|^{-1}e^{-\frac{1}{8}|v-u|^{2}}\bigg(\frac{C_2}{2}\bigg)^me^A\bigg(\frac{C_1}{2}\bigg)^pB_1^{j-2p}\\&\quad\ \times B_2^{\tau_2-2m-(\tau_1-j)}D^{\tau_1-j}
,
\\&\mathcal{K}_5=p\bigg[-\frac{v_1-u_1}{|v-u|^{2}}+\frac{(v_1-u_1)^3}{|v-u|^{4}}\bigg]|v-u|^{-1}e^{-\frac{1}{8}|v-u|^{2}}\bigg(\frac{C_2}{2}\bigg)^me^A\bigg(\frac{C_1}{2}\bigg)^{p-1}B_1^{j-2p}B_2^{\tau_2-2m-(\tau_1-j)}\\&\quad\ \times D^{\tau_1-j},
\\&\mathcal{K}_6=(j-2p)\bigg[-\frac{v_1(v_1-u_1)}{|v-u|^{2}}-\frac{|v|^2-|u|^2}{2|v-u|^{2}}+\frac{(v_1-u_1)^2(|v|^2-|u|^2)}{|v-u|^{4}}\bigg]|v-u|^{-1}e^{-\frac{1}{8}|v-u|^{2}}\\&\quad\ \times \bigg(\frac{C_2}{2}\bigg)^me^A\bigg(\frac{C_1}{2}\bigg)^{p}B_1^{j-2p-1}B_2^{\tau_2-2m-(\tau_1-j)}D^{\tau_1-j},
\\&\mathcal{K}_7=[\tau_2-2m-(\tau_1-j)]\bigg[-\frac{v_1(v_2-u_2)}{|v-u|^{2}}+\frac{(v_1-u_1)(v_2-u_2)(|v|^2-|u|^2)}{|v-u|^{4}}\bigg]|v-u|^{-1}e^{-\frac{1}{8}|v-u|^{2}}\\&\quad\ \times \bigg(\frac{C_2}{2}\bigg)^me^A\bigg(\frac{C_1}{2}\bigg)^{p}B_1^{j-2p}B_2^{\tau_2-2m-(\tau_1-j)-1}D^{\tau_1-j},
\\&\mathcal{K}_8=(\tau_1-j)\bigg[-\frac{v_2-u_2}{|v-u|^{2}}+\frac{(v_1-u_1)^2(v_2-u_2)}{|v-u|^{4}}\bigg]|v-u|^{-1}e^{-\frac{1}{8}|v-u|^{2}}\bigg(\frac{C_2}{2}\bigg)^me^A\bigg(\frac{C_1}{2}\bigg)^{p}B_1^{j-2p}\\&\quad\ \times B_2^{\tau_2-2m-(\tau_1-j)}D^{\tau_1-j-1}.
\end{align*}
Here we remark that $\mathcal{K}_4$, $\mathcal{K}_6$ and $\mathcal{K}_7$ generate extra $v_1$ term in the computation. Note that $|C_1|+|C_2|+|D|\leq 3$. We get 
\begin{align*}
&|\mathcal{K}_1|\leq C|v-u|^{-2}e^{-\frac{1}{8}|v-u|^{2}}e^AB_1^{j-2p}B_2^{\tau_2-2m-(\tau_1-j)},
\\&|\mathcal{K}_2|\leq C e^{-\frac{1}{8}|v-u|^{2}}e^AB_1^{j-2p}B_2^{\tau_2-2m-(\tau_1-j)},
\\&|\mathcal{K}_3|\leq C\frac{m}{2^m}|v-u|^{-2}e^{-\frac{1}{8}|v-u|^{2}}e^AB_1^{j-2p}B_2^{\tau_2-2m-(\tau_1-j)},
\\&|\mathcal{K}_4|\leq C\nu(v) \bigg[\frac{(|v|^2-|u|^2)}{|v-u|^{3}}+\frac{(|v|^2-|u|^2)^2}{|v-u|^{4}}\bigg]e^{-\frac{1}{8}|v-u|^{2}}e^AB_1^{j-2p} B_2^{\tau_2-2m-(\tau_1-j)}
,
\\&|\mathcal{K}_5|\leq C\frac{p}{2^p}|v-u|^{-2}e^{-\frac{1}{8}|v-u|^{2}}e^AB_1^{j-2p}B_2^{\tau_2-2m-(\tau_1-j)},
\\&|\mathcal{K}_6|\leq C\nu(v)\frac{j-2p}{2^{j-2p}}\bigg[\frac{1}{|v-u|^{2}}+\frac{|v|^2-|u|^2}{|v-u|^{3}}\bigg]e^{-\frac{1}{8}|v-u|^{2}}e^A(2B_1)^{j-2p-1}B_2^{\tau_2-2m-(\tau_1-j)},
\\&|\mathcal{K}_7|\leq C\nu(v)\frac{\tau_2-2m-(\tau_1-j)}{2^{\tau_2-2m-(\tau_1-j)}}\bigg[\frac{1}{|v-u|^{2}}+\frac{|v|^2-|u|^2}{|v-u|^{3}}\bigg]e^{-\frac{1}{8}|v-u|^{2}}e^AB_1^{j-2p}(2B_2)^{\tau_2-2m-(\tau_1-j)-1},
\\&|\mathcal{K}_8|\leq C(\tau_1-j)|v-u|^{-2}e^{-\frac{1}{8}|v-u|^{2}}e^AB_1^{j-2p}B_2^{\tau_2-2m-(\tau_1-j)}.
\end{align*}

As a result, combining these estimates, we conclude that
\begin{align*}
&|\mathcal{K}_1|+|\mathcal{K}_2|+|\mathcal{K}_3|+|\mathcal{K}_4|+|\mathcal{K}_5|+|\mathcal{K}_8|\\&\leq C\nu(v)e^{-\frac{1}{8}|v-u|^{2}}e^{A}\bigg[1+(1+\tau_1-j)|v-u|^{-2}+\frac{|v|^2-|u|^2}{|v-u|^{3}}+\frac{(|v|^2-|u|^2)^2}{|v-u|^{4}}\bigg]B_1^{j-2p}B_2^{\tau_2-2m-(\tau_1-j)},
\\&|\mathcal{K}_6|\leq C\nu(v)e^{-\frac{1}{8}|v-u|^{2}}e^{A}\bigg[ |v-u|^{-2}+\frac{|v|^2-|u|^2}{|v-u|^{3}}\bigg](2B_1)^{j-2p-1}B_2^{\tau_2-2m-(\tau_1-j)},
\\&|\mathcal{K}_7|\leq C\nu(v)e^{-\frac{1}{8}|v-u|^{2}}e^{A}\bigg[|v-u|^{-2}+\frac{|v|^2-|u|^2}{|v-u|^{3}}\bigg]B_1^{j-2p}(2B_2)^{\tau_2-2m-(\tau_1-j)-1}.
\end{align*}
Together with $2|B_1|\leq \frac{||v|^2-|u|^2|}{|v-u|}, 2|B_2|\leq \frac{||v|^2-|u|^2|}{|v-u|}$, we obtain
\begin{align*}
&|\mathcal{K}_1|+|\mathcal{K}_2|+|\mathcal{K}_3|+|\mathcal{K}_4|+|\mathcal{K}_5|+|\mathcal{K}_8|\\&\leq C\nu(v)e^{-\frac{1}{8}|v-u|^{2}}e^{A/2}e^{A/6}\bigg[1+(1+\tau_1-j)|v-u|^{-2}+\frac{|v|^2-|u|^2}{|v-u|^{3}}+\frac{(|v|^2-|u|^2)^2}{|v-u|^{4}}\bigg]\\& \times e^{A/6} \bigg(\frac{|v|^2-|u|^2}{|v-u|}\bigg)^{j-2p}e^{A/6}\bigg(\frac{|v|^2-|u|^2}{|v-u|}\bigg)^{\tau_2-2m-(\tau_1-j)},
\\&|\mathcal{K}_6|\leq C\nu(v)e^{-\frac{1}{8}|v-u|^{2}}e^{A/2}e^{A/6}\bigg[ |v-u|^{-2}+\frac{|v|^2-|u|^2}{|v-u|^{3}}\bigg]e^{A/6}\bigg(\frac{|v|^2-|u|^2}{|v-u|}\bigg)^{j-2p-1}\\& \times e^{A/6}\bigg(\frac{|v|^2-|u|^2}{|v-u|}\bigg)^{\tau_2-2m-(\tau_1-j)},
\\&|\mathcal{K}_7|\leq C\nu(v)e^{-\frac{1}{8}|v-u|^{2}}e^{A/2}e^{A/6}\bigg[|v-u|^{-2}+\frac{|v|^2-|u|^2}{|v-u|^{3}}\bigg]e^{A/6}\bigg(\frac{|v|^2-|u|^2}{|v-u|}\bigg)^{j-2p}\\& \times e^{A/6}\bigg(\frac{|v|^2-|u|^2}{|v-u|}\bigg)^{\tau_2-2m-(\tau_1-j)-1}.
\end{align*}
Then we use Lemma \ref{lemma:enut_control} with replacing $\nu(v)t$ and 8 by $(|v|^2+|u|^2)/|v-u|$ and 48 respectively to control polynomial growth term by exponential growth term:
\begin{align*}
&|\mathcal{K}_1|+|\mathcal{K}_2|+|\mathcal{K}_3|+|\mathcal{K}_4|+|\mathcal{K}_5|+|\mathcal{K}_6|+|\mathcal{K}_7|+|\mathcal{K}_8|
\\&\leq C\nu(v)e^{-\frac{1}{8}|v-u|^{2}}e^{A/2}\big[1+(1+\tau_1-j)|v-u|^{-2} \big][(j-2p)/2]!48^{(j-2p)/2}[(\tau_2-2m-\tau_1+j)/2]!\\&\times48^{(\tau_2-2m-\tau_1+j)/2}.
\end{align*}

Thus we have
\begin{align*}
&\bigg|\ \p_{v_\parallel}^\iota\bigg[|v-u|^{-1}e^{-\frac{1}{8}|v-u|^{2}}\bigg(\frac{C_2}{2}\bigg)^me^A\bigg(\frac{C_1}{2}\bigg)^pB_1^{j-2p}B_2^{\tau_2-2m-(\tau_1-j)}D^{\tau_1-j}\bigg] \bigg|
\\& \leq C \nu(v) e^{-\frac{1}{8}|v-u|^{2}}e^{A/2}[(j-2p)/2]!6^{(j-2p)/2}[(\tau_2-2m-\tau_1+j)/2]!6^{(\tau_2-2m-\tau_1+j)/2}
\\& \times\big[1+(1+\tau_1-j)|v-u|^{-2} \big].
\end{align*}
Therefore, we derive
\begin{align*}
 &\bigg|\int_{\mathbb{R}^3}\frac{w(v)}{w(u)} \p_{v_\parallel}^\iota\bigg[|v-u|^{-1}e^{-\frac{1}{8}|v-u|^{2}}\bigg(\frac{C_2}{2}\bigg)^me^A\bigg(\frac{C_1}{2}\bigg)^pB_1^{j-2p}B_2^{\tau_2-2m-(\tau_1-j)}D^{\tau_1-j}\bigg] \dd u \bigg|
\\& \leq C\nu(v)(1+\tau_1-j)\bigg|\int_{\mathbb{R}^3}\frac{w(v)}{w(u)} e^{-\frac{1}{8}|v-u|^{2}}e^{A/2}(1+|v-u|^{-2}\big) \dd u \bigg|\\&\times[(j-2p)/2]!48^{(j-2p)/2}[(\tau_2-2m-\tau_1+j)/2]!48^{(\tau_2-2m-\tau_1+j)/2}.
\end{align*}

We can refer to \cite[Lemma 3]{G} to conclude that there exists a constant $\tilde\varrho>0$ such that 
 \begin{align*}
& \bigg|\int_{\mathbb{R}^3}\frac{w(v)}{w(u)}\big(1+|v-u|^{-2}\big)e^{-\frac{1}{8}|v-u|^{2}}e^{A/2} \dd u \bigg|
\\&\leq \bigg|\int_{\mathbb{R}^3}\big(1+|v-u|^{-2}\big)e^{-\tilde\varrho|v-u|^{2}}\dd u \bigg|\leq C.
\end{align*}
Therefore, we get
\begin{align*}
&\sum_{m=0}^{[\tau_2/2]}\frac{1}{m!}\sum_{j=0}^{\tau_1}\frac{1}{(\tau_1-j)!}\frac{1}{[\tau_2-2m-(\tau_1-j)]!}\sum_{p=0}^{[j/2]}\frac{1}{p!(j-2p)!}\\&\times \bigg|\int_{\mathbb{R}^3}\frac{w(v)}{w(u)} \p_{v_\parallel}^\iota\bigg[|v-u|^{-1}e^{-\frac{1}{8}|v-u|^{2}}\bigg(\frac{C_2}{2}\bigg)^me^A\bigg(\frac{C_1}{2}\bigg)^pB_1^{j-2p}B_2^{\tau_2-2m-(\tau_1-j)}D^{\tau_1-j}\bigg] \dd u \bigg|
\\& \leq C \nu(v) \sum_{m=0}^{[\tau_2/2]}\frac{1}{m!}\sum_{j=0}^{\tau_1}\frac{(1+\tau_1-j)}{(\tau_1-j)!}\sum_{p=0}^{[j/2]}\frac{1}{p!}\frac{[(\tau_2-2m-\tau_1+j)/2]!48^{(\tau_2-2m-\tau_1+j)/2}}{[\tau_2-2m-(\tau_1-j)]!}\\& \times\frac{[(j-2p)/2]!48^{(j-2p)/2}}{(j-2p)!}
 \leq C \nu(v).
\end{align*}
Here in the last inequality, we used \eqref{MN2} and have
\begin{align*}
\frac{[(\tau_2-2m-\tau_1+j)/2]!48^{(\tau_2-2m-\tau_1+j)/2}}{[\tau_2-2m-(\tau_1-j)]!}\frac{[(j-2p)/2]!48^{(j-2p)/2}}{(j-2p)!}\leq C.
\end{align*}
This concludes the proof of Lemma \ref{lemma:derk12}.
\end{proof}

\noindent {\bf Acknowledgment:}\,
The research of RJD was partially supported by the General Research Fund (Project No.~14301822) from RGC of Hong Kong and also by the grant from the National Natural Science Foundation of China (Project No.~12425109).

\medskip
\noindent\textbf{Data Availability Statement:}
Data sharing is not applicable to this article as no datasets were generated or analysed during the current study.

\noindent\textbf{Conflict of Interest:}
The authors declare that they have no conflict of interest.

\bibliographystyle{abbrv}
%\bibliography{citation}

\iffalse

\fi

\end{document}